\documentclass[10pt]{article}

\usepackage[utf8]{inputenc}
\usepackage[T1]{fontenc}

\usepackage[left=3.8cm,right=3.8cm,top=2.85cm,bottom=2.8cm]{geometry}

\usepackage[colorlinks=false, hidelinks]{hyperref}

\usepackage{amsmath,amsthm,amssymb}
\usepackage{bbm}

\usepackage{graphicx}
\usepackage{epstopdf}

\newcommand{\mtrx}[1]{\begin{pmatrix} #1 \end{pmatrix}}
\usepackage[font=small,labelfont=bf]{caption}

\usepackage{color}

\usepackage{customCommands}

\usepackage{cleveref}
\crefname{figure}{Figure}{Figures}
\crefname{table}{Table}{Tables}
\crefname{lemma}{Lemma}{Lemmas}
\crefname{section}{Section}{Sections}

\usepackage{multirow} 
\renewcommand{\(}{\left(}
\renewcommand{\)}{\right)}

\newcommand{\RN}{\mathbb{R}^N}

\usepackage{mathbbol}

\definecolor{grau}{gray}{.5}

\usepackage{txfonts}

\theoremstyle{plain}

\newtheorem{theorem}{Theorem}
\newtheorem{lemma}[theorem]{Lemma}
\newtheorem{proposition}[theorem]{Proposition}
\newtheorem{corollary}[theorem]{Corollary}

\theoremstyle{definition}

\newtheorem{remark}[theorem]{Remark}
\newtheorem{remark*}{Remark}
\newtheorem{example}[theorem]{Example}

\renewcommand{\epsilon}{\varepsilon}

\usepackage[sort&compress, numbers]{natbib}
\usepackage{snapshot}

\usepackage{booktabs}
\usepackage{multicol}
\usepackage[boxed, ruled, lined, linesnumbered]{algorithm2e}

\renewcommand{\argmin}{\operatorname*{argmin}}

\newcommand{\dint}[2]{{#1:#2}}
\newcommand{\RI}{\operatorname{R_{\text{ind}}}}

\usepackage[usenames,dvipsnames,svgnames]{xcolor}

\usepackage{enumerate}
\usepackage{csquotes}

\usepackage[font=footnotesize, justification=raggedright]{subcaption}

\title{Smoothing for signals with discontinuities \\
using higher order Mumford-Shah models}

\date{}

\author{  Martin Storath%
\thanks{Interdisciplinary Center for Scientific Computing (IWR), Universität Heidelberg, Germany, and 
Faculty of Applied Natural Sciences and Humanities, University of Applied Sciences Würzburg-Schweinfurt,
Germany}
 ~~~~~
 Lukas Kiefer%
	\thanks{Mathematical
Imaging Group, Universität Heidelberg, Germany}
\thanks{Department of Mathematics and Natural Sciences, Hochschule Darmstadt, Germany}
  ~~~~~
    Andreas Weinmann$^\ddagger$%
 \thanks{
 Institute of Computational Biology, Helmholtz Zentrum München, Germany}
}

\begin{document}
\maketitle

\begin{abstract}
Minimizing the Mumford-Shah functional is frequently used for smoothing
	signals or time series with discontinuities.
A significant limitation of the standard Mumford-Shah model is that linear trends -- and in general polynomial trends --
in the data are not well preserved. 
This can be improved by building on  splines of higher order
which leads to higher order Mumford-Shah models.
In this work, we study these models in the univariate situation:
we discuss important differences to the first order Mumford-Shah model,
and we obtain uniqueness results for their solutions.
As a main contribution, we derive fast minimization algorithms for Mumford-Shah models of arbitrary orders.
We show that the worst case complexity of all proposed schemes is quadratic in the length of the signal.
Remarkably, they thus achieve the worst case complexity of the fastest solver for the piecewise constant Mumford-Shah model (which is the simplest model of the class).
Further, we obtain stability results for the proposed algorithms.
We complement these results with a
numerical study. 
Our reference implementation processes signals with more than 10,000 elements in less than one second.

\end{abstract}

{ \small \textbf{Keywords:} piecewise smooth approximation, discontinuous signals, complexity penalized estimation,
changepoint estimation, segmented least squares, spline smoothing,  Mumford-Shah model, Potts model, Blake-Zisserman model.

\textbf{AMS subject classification (MSC2010):}
65D10, % Smoothing, curve fitting
65K05, % Mathematical programming methods
62G08, % Nonparametric regression
65K10, % Optimization and variational techniques
65D07.  % Splines
}

\section{Introduction}

Smoothing is an important processing step when working with measured signals or time series.
For signals without discontinuities, it is standard to use smoothing splines for this task.
In various applications however, the signals possess discontinuities.
Such applications are, for example, the cross-hybridization of DNA \cite{snijders2001assembly,drobyshev2003specificity,hupe2004analysis}, 
the reconstruction of brain stimuli \cite{winkler2005don}, 
single-molecule fluorescence resonance energy transfer \cite{joo2008advances},
cellular ion channel functionalities \cite{hotz2013idealizing},
photo-emission spectroscopy \cite{frick2014multiscale}
and the rotations 
of the bacterial flagellar motor 
\cite{sowa2005direct}; see also 
\cite{little2011generalized,little2011generalized2, frick2014multiscale} for further examples.
Frequently, it is important to preserve the discontinuities since they typically indicate a significant change.
Unfortunately, the locations of the discontinuities are in general
 unknown; they have to be estimated along with the signal.

One approach to this problem is to estimate the discontinuities locally and adapt the corresponding fitting operator 
	to the local situation in an explicit way; for instance 	\cite{arandiga2005interpolation,amat2017new,harten1996multiresolution}.
Other approaches use variational methods: an energy functional is considered and a corresponding minimizer yields a smoothed approximation to the data. Examples for discontinuity preserving methods  are total variation (TV)/Rudin-Osher-Fatemi models \cite{rudin1992nonlinear} which yield a kind of piecewise constant approximation. To account for linear and higher order trends in the data, higher order TV methods have been proposed  \cite{chambolle1997image,Bredies10}. TV models and their higher order extensions, however, do not explicitly incorporate the notion of discontinuities and segment boundaries.
A model taking these notions explicitly into account is the Mumford-Shah model.
It simultaneously estimates a discontinuity set and a corresponding piecewise smooth approximation \cite{mumford1989optimal}.
Its piecewise constant variant is also known as the Potts model \cite{potts1952some,geman1984stochastic,winkler2003image},
or as the Chan-Vese model for the  case of two phases \cite{chan2001active}.
Mumford-Shah and Potts models
are classical models for discontinuity preserving smoothing and for segmentation
\cite{blake1987visual,geman1984stochastic, mumford1985boundary, winkler2003image}.
More recent applications are smoothing of video sequences \cite{strekalovskiy2014real} and segmentation with shape priors \cite{isack2016hedgehog}. They are also used for stabilizing the reconstructions of inverse problems
\cite{rondi2001enhanced,rondi2008regularization, jiang2014regularizing,ramlau2007mumford,ramlau2010regularization, weinmann2015iterative}.
Corresponding existence results on minimizers were established in \cite{artina2013linearly,fornasier2013existence}.
	Besides the classical $\ell^2$-based models, $\ell^p$-based variants with $p\geq 1$ \cite{fornasier2010iterative,hohm2015algorithmic,kolmogorov2015total, friedrich2008complexity,weinmann2015l1potts,storath2017jump} and manifold-valued data spaces \cite{wang2005dti,weinmann2014mumford, storath2017jump}
	have been considered.
Discretizations have been studied in \cite{chambolle1995image,chambolle1999finite}.
	 Mumford-Shah and Potts problems  are known to be NP hard in the  multivariate case and in the univariate inverse problem setup \cite{veksler1999efficient,boykov2001fast,weinmann2015iterative}.
Approximate solution strategies are for example based on graduated non-convexity \cite{blake1987visual, blake1989comparison}, 
	approximation by elliptic functionals \cite{ambrosio1990approximation, rondi2001enhanced,bar2004variational}, 
	graph cuts \cite{boykov2001fast},
	active contours \cite{tsai2001curve}, convex relaxations \cite{strekalovskiy2012convex},
	iterative thresholding algorithms \cite{fornasier2010iterative}, ADMM splitting schemes  \cite{storath2014fast,hohm2015algorithmic},
	and iterative Potts minimization \cite{weinmann2015iterative,kiefer2018iterative}.
    In the univariate case, dynamic programming strategies yield exact solutions as discussed in more detail later on.

A significant limitation of the classical Mumford-Shah model is that it does not well preserve locally
linear or polynomial trends in the data.
Instead, it tends to produce spurious discontinuities when the slope of the signal is
too high, which has been termed the \enquote{gradient limit effect} by Blake and Zisserman~\cite{blake1987visual}.
The reason for this is that it penalizes deviations from a piecewise constant spline.
The preservation of linear or polynomial trends
can be accomplished by considering \emph{higher order Mumford-Shah models}.
They penalize 
	 the deviation from a piecewise polynomial instead of the deviation from a piecewise constant function.
	The multivariate discrete {higher order Mumford-Shah and Potts models} are particularly interesting in image processing for edge preserving smoothing of images with locally linear or polynomial trends; for example, second order methods are applied for piecewise approximation and segmentation of images 	\cite{wang2013level, zanetti2017piecewise,chen2017general, zanetti2016numerical} or regularization of flow fields \cite{yang2015dense,fortun2017fast}. The univariate discrete models are particularly interesting for smoothing time series with discontinuities; examples with biological applications are for instance \cite{nord2017catch, sowa2005direct}.
 Further they arise as subproblems in univariate and multivariate splitting schemes \cite{storath2014fast,hohm2015algorithmic, weinmann2015iterative,kiefer2018iterative}.	

We intend to systematically study higher order Mumford-Shah models, where, in this paper,  
	we consider the univariate and discrete situation.
It is given by the minimization problem
\begin{equation}\label{eq:penalizedProblem}\tag{$\Pc_{k,\beta,\gamma}$}
	(u^*, \Ic^*) = \argmin_{u \in \R^N,~\Ic \text{ partition} }~ \|  u - f \|_2^2 +  \sum_{I \in \Ic}   \beta^{2k} \| \nabla^k u_I \|_2^2  + \gamma \, | \Ic |.
\end{equation}
Here, $f \in \R^N$ denotes the given data, 
and the minimum is computed with respect to the target variables
$u$ and $\Ic,$ 
where $u$ is a discrete univariate signal of length $N$
and $\Ic$ is a partition of the domain $\Omega = \{1, \ldots, N\}.$
(The connection between $u$ and $\Ic$ is discussed in detail later in \cref{sec:basic_proporties}.)
The symbol $\nabla^k u_I$ denotes the $k$-th order finite difference operator
applied to the vector $u$ restricted to the \enquote{interval} $I$ of a partition~$\Ic$ of the domain. 
The functional value comprises a cost term for the data deviation,
a cost term for the inner energy of a spline on the single segments, and a cost term for the complexity of the partition (measured in terms of the number of segments $|\Ic|$).
The minimizing signal $u^*$ is a piecewise $k$-th order discrete spline approximation to $f$ with elasticity parameter $\beta$ 
which has discontinuities or breakpoints at the boundaries given by the partition $\Ic^*.$
Choosing a large parameter value $\beta$ leads to stronger smoothing on the segments, and choosing a large parameter value $\gamma$ leads to less segments.
It is interesting to look at the cases for very large parameters of $\beta$ and $\gamma.$
As the kernel of $\nabla^k$ consists of  polynomials of maximum degree $k-1,$
the limit situation $\beta \to \infty$ 
can be written as
\begin{equation}\label{eq:penalizedProblemPoly}\tag{$\Pc_{k,\infty,\gamma}$}
	\argmin_{u, \Ic}~ \|  u - f \|_2^2   + \gamma \, | \Ic |, \quad \text{s.t. $u_I$ is a polynomial of maximum degree $k-1$ for all $I \in \Ic.$}
\end{equation}
As the case $k=1$ is known as the Potts model (as a tribute to the work of R.~Potts \cite{potts1952some}) we  refer to \eqref{eq:penalizedProblemPoly}  as \emph{higher order Potts model}.
On the other hand, for sufficiently large $\gamma,$
it 
reduces to the (discrete) $k$-th order spline approximation
\begin{equation}\label{eq:splineProblem}\tag{$\Pc_{k,\beta,\infty}$}
	\argmin_{u}~\|  u - f \|_2^2   +   \beta^{2k} \| \nabla^k u \|_2^2,
\end{equation}
which is a classical method for smoothing data; see \cite{whittaker1922new, wahba1990spline}.
Depending on the application, 
there are different points of view for 
 Mumford-Shah-type models:
On the one hand, the optimal partition $\Ic^*$
can serve as a basis for identifying segment neighborhoods \cite{auger1989algorithms}
and as an indicator for changepoints of the signal \cite{killick2012optimal}.
On the other hand, the corresponding optimal signal $u^*$ can serve as a smoother 
for a signal with discontinuities \cite{blake1987visual,winkler2002smoothers}.

In the literature, the members of the higher order Mumford-Shah family
\eqref{eq:penalizedProblem} 
have been considered for the cases $k=1, 2$ 
and for $\beta < \infty$ or $\beta = \infty$ (strict piecewise polynomial model)
by various individual studies.
The classical (first order) Potts model $(\Pc_{1, \infty, \gamma})$ and closely related models were studied in various works; for example \cite{bruce1965optimum, liebscher1999potts, winkler2002smoothers, wittich2008complexity, killick2012optimal}. 
The strict piecewise linear model $(\Pc_{2, \infty, \gamma})$
was studied by Bellman and Roth~\cite{bellman1969curve}; see also
\cite{kleinberg2006algorithm}. We refer to it as affine Potts model or piecewise affine Mumford-Shah model. 
The first order problems $(\Pc_{1, \beta, \gamma})$ for arbitrary parameters $\beta$ 
have been studied in the seminal works of Mumford and Shah~\cite{mumford1985boundary,mumford1989optimal}.
(This motivates the denomination higher order Mumford-Shah model for the family \eqref{eq:penalizedProblem}.)
The same problem was studied at around the same time by Blake and Zisserman 
\cite{blake1987visual} under the name weak string model.
They  also introduced a second order extension, called the weak rod model  \cite{blake1987visual}.
This model is more general than the model $(\Pc_{2, \beta, \gamma})$ considered here
because it  has an extra penalty for discontinuities in the first derivative.
We refer to \cite{carriero2015survey,zanetti2016numerical} 
for a recent investigation of Blake-Zisserman models in 2D.
To our knowledge, the  models $(\Pc_{k, \beta, \gamma})$ have not been systematically studied  
for arbitrary orders $k.$

Although \eqref{eq:penalizedProblem}
is a non-convex problem it can be solved exactly by dynamic programming; see \cite{bellman1969curve, blake1989comparison,auger1989algorithms,winkler2002smoothers, jackson2005algorithm, friedrich2008complexity}.
 The state-of-the-art solver has worst case complexity 
$O(N^2 \phi(N))$ where $\phi(N)$ comprises the costs of computing 
a spline approximation error on an interval of maximum length $N$; see  \cite{winkler2002smoothers, kleinberg2006algorithm, friedrich2008complexity}.
Killick et al.~\cite{killick2012optimal} proposed a pruning strategy to accelerate the algorithm in a special yet practically relevant case:
if the expected number of segments $|\Ic^*|$ grows linearly in $N$
and if the expected log-likelihood fulfills certain estimates, detailed in~\cite{killick2012optimal},
the expected complexity is $\Oc_P(N \phi(N)).$
Another pruning scheme has been established in  \cite{storath2014fast}.
An algorithm for solving the first order Mumford-Shah problem for all parameters $\gamma$ simultaneously was proposed in~\cite{friedrich2008complexity}.
It is straightforward to devise an algorithm for $(\Pc_{k, \beta, \gamma})$ of complexity $\Oc(N^3).$
For the first order problem $(\Pc_{1, \beta, \gamma}),$
we proposed an $\Oc(N^2)$ algorithm \cite{hohm2015algorithmic} 
which utilizes a fast computation scheme
for the approximation errors proposed by Blake~\cite{blake1989comparison}. 
Unfortunately, as that scheme is based on algebraic recurrences,
a generalization to arbitrary orders of $k$ seems difficult.
By precomputations of moments one can achieve $\phi(N) = \Oc(1);$ see \cite{lemire2007better, friedrich2008complexity}.
Although that approach gives reasonable results for the low orders $k=1,2,$ 
it gets numerically unstable for higher orders.
A different dynamic programming approach was discussed in \cite{blake1987visual} which however only computes an approximate minimizer.

\paragraph{Contribution.}

This work deals with the analysis and with solvers for the
higher order Mumford-Shah and Potts problems~\eqref{eq:penalizedProblem}.
First, we discuss basic properties of higher order Mumford-Shah models,
we prove that the solutions are unique for almost all input data,
and we discuss connections with related models.
A main contribution is a fast non-iterative algorithm for minimizing the 
higher-order Mumford-Shah and Potts models of arbitrary order.
The proposed schemes are based on dynamic programming and recurrence relations. 
We prove that the proposed algorithms have the same 
worst case complexity as the state-of-the-art solver
for minimizing the (simpler) piecewise constant Mumford-Shah model,
i.e., their runtime grows quadratically with respect to the length of the signal.
Our reference implementation processes
signals of length over 10,000 in less than one second
on a standard desktop computer.
Further, we derive stability results for the proposed algorithms.
Eventually, we provide a numerical study 
where we in particular compare with the first order model 
with respect to runtime and reconstruction quality.

\paragraph{Organization of the paper.}

In Section~\ref{sec:higher_order}
we describe and discuss higher order Mumford-Shah and higher order Potts models.
In Section~\ref{sec:FastStableMumfordShahSolver},
we develop a fast solver for higher order Mumford-Shah problems
and for higher order Potts problems, and we analyze the stability.
A numerical study is given in Section~\ref{sec:numerical_results}.

\section{Higher order Mumford-Shah and Potts models}\label{sec:higher_order}

We start with some basic notations and definitions.
Our goal is to recover an 
unknown signal $g \in \R^N$
from its noisy samples 
\begin{equation}\label{eq:regressionModel}
	f_n = g_n + \eta_n, \quad\text{}n = 1,...,N,
\end{equation}
where the $\eta_n$ are independently distributed Gaussian random variables of zero mean and variance $\sigma^2.$
We write $\dint{l}{r}$ for a discrete \enquote{interval} from $l$ to $r$, i.e.
$\dint{l}{r} =\{l,l+1, \ldots r\}.$
It is convenient to use the Matlab-type notation $x_{I} = x_{l:r} = (x_l, x_{l+1}, \ldots, x_r)$
for indexing.
We say that $\Ic$ is a \emph{partition of $\Omega = \dint{1}{N}$ into intervals},
if  $I\cap J = \emptyset$ for all $I, J \in \Ic,$
if $\bigcup_{I \in \Ic} I = \Omega,$ 
and if all $I \in \Ic$ are discrete intervals  of the form $I= \dint{l}{r},$
with $1 \leq l \leq r \leq N.$
As we will only work with partitions into intervals here,
we briefly call $\Ic$ a partition.
Further, we use the notation $\|u\|_2^2 = \sum_{n=1}^N u_n^2$ for $u \in \R^N$
to denote the squared Euclidean length of $u.$

\subsection{First order Mumford-Shah models and the gradient limit effect}

Before considering higher order models, 
we first review some important properties of first order models.
The first order Mumford-Shah problem ($\Pc_{1, \beta, \gamma}$) on the discrete domain $\Omega$ can be written as
\begin{equation}\label{eq:firstOrderMS_discrete}
	(u^*, \Ic^*) = \argmin_{u \in \R^N,\,\Ic \text{ partition of }\Omega }~
	 \sum_{n=1}^N  (u_n - f_n)^2 + \beta^{2} \sum_{I \in \Ic}  \sum_{i =1}^{|I| - 1} ( (u_I)_{i+1} - (u_I)_i)^2  + \gamma \, | \Ic |.
\end{equation}
The two-fold minimization with respect to the signal $u$ and the partition $\Ic$ is instructive 
but cumbersome in practice.
We can remove
the explicit dependance  on the signal $u.$
The formulation in terms of partitions reads
\begin{equation}\label{eq:firstOrderMS_I}
	\Ic^* = \argmin_{\Ic \text{ partition of } \dint{1}{N} }  \sum_{I \in \Ic} \Big( \Ec^I + \gamma\Big),
	\quad \text{with } \Ec^I = \min_{v \in \R^{|I|}}~
\sum_{i =1}^{|I|}  (v_i - f_i)^2 + \beta^{2} \sum_{I \in \Ic}  \sum_{i =1}^{|I| - 1} (v_{i+1} - v_i)^2.
\end{equation}
Here $\Ec^I$ describes the error of the best first order discrete spline approximation on the interval $I.$
This formulation is typically used for derivation of algorithms based on dynamic programming 
\cite{blake1989comparison, friedrich2008complexity}.

For the first order model \eqref{eq:firstOrderMS_discrete}, it is possible to rewrite the functional without partitions.
The corresponding formulation reads
\begin{equation}\label{eq:firstOrderMS_u}
	u^* = \argmin_{u \in \R^N}~
	 \sum_{n=1}^N  (u_n - f_n)^2 +  \sum_{n=1}^{N-1} \min(\beta^{2} (u_{n+1} - u_n )^2, \gamma).
\end{equation}
This formulation is useful for derivation of algorithms based on iterative thresholding techniques \cite{fornasier2010iterative}.
The key property that makes the formulation in terms of $u$
possible is that, for the first order model, the signal $u^*$ and the partition $\Ic^*$  are 
equivalent in the sense that $\Ic^*$ can be recovered from $u^*$ 
and vice-versa.

\begin{figure}[t]
	\captionsetup[subfigure]{justification=centering}	
	\centering
		\begin{subfigure}{0.32\textwidth}
		\includegraphics[width=1\textwidth,trim=25 14 0 0, clip]{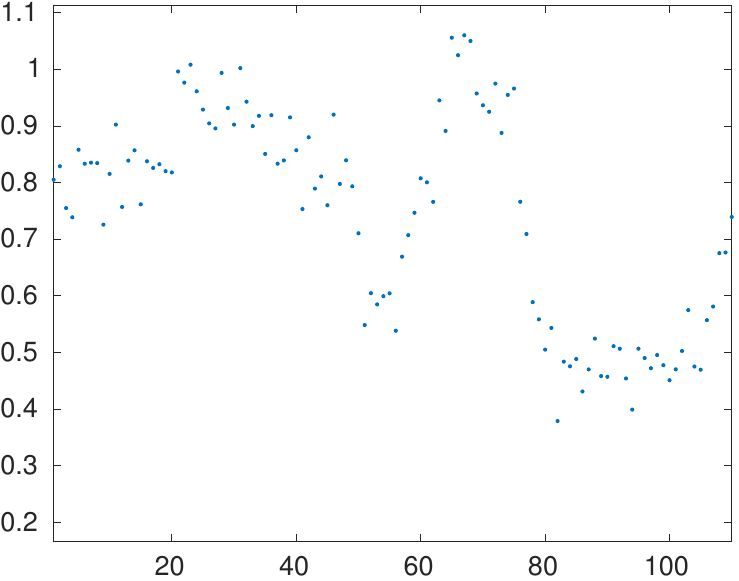}
	\caption{Noisy signal with discontinuities}
	\end{subfigure}\hspace{0.05\textwidth}
	\begin{subfigure}{0.32\textwidth}
		\includegraphics[width=1\textwidth,trim=25 14 0 0, clip]{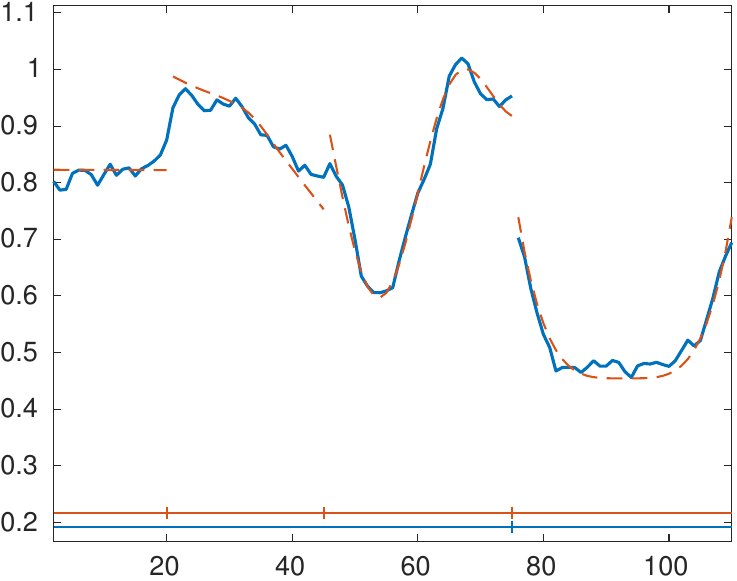}
	\caption{Result of $(\Pc_{1,\beta,\gamma})$ with optimal 
	$\beta,\gamma$}
	\end{subfigure}
	\\[2ex]
	\begin{subfigure}{0.32\textwidth}
		\includegraphics[width=1\textwidth,trim=25 14 0 0, clip]{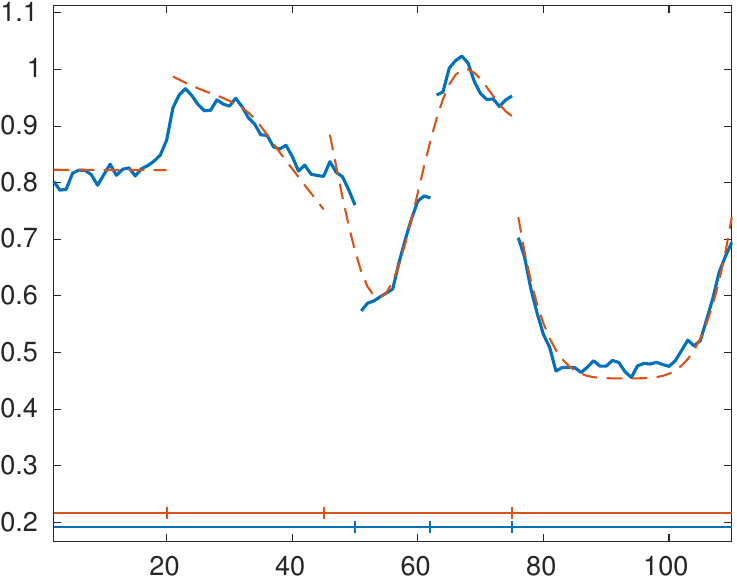}
		\caption{Smaller complexity penalty $\gamma$}
	\end{subfigure}\hspace{0.05\textwidth}
	\begin{subfigure}{0.32\textwidth}
		\includegraphics[width=1\textwidth,trim=25 14 0 0, clip]{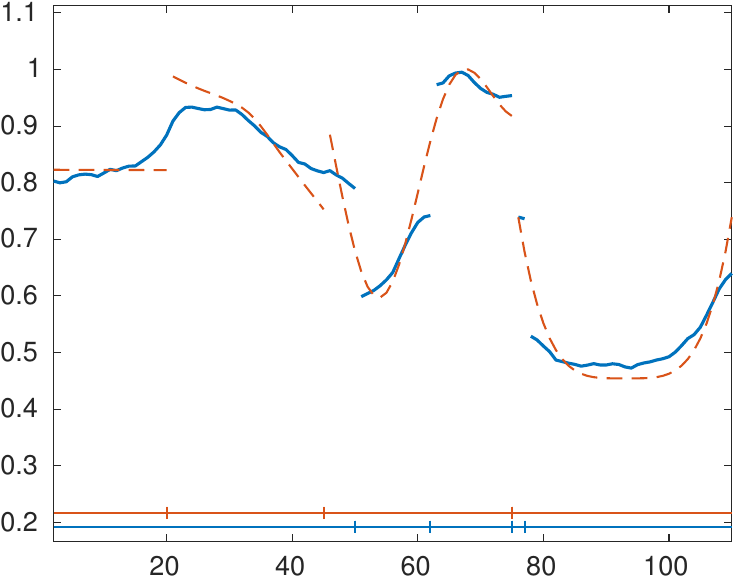}
		\caption{Larger elasticity parameter $\beta$}
	\end{subfigure}
	\caption{
		Limitations of the classical (first order) Mumford-Shah model:
		Subfigures (b) to (d) show reconstructions of the signal in (a)
		by the first order Mumford-Shah model  $(\Pc_{1,\beta,\gamma}).$
		The red  dashed lines depict the ground truth; the streaks at the bottom indicate the  discontinuities or segments of the ground truth (top, red) and the computed segmentations (bottom, blue).
		 \emph{(b)}~The model parameters were optimized with respect to the $\ell^2$ error to the ground truth
		($\gamma = 0.04$, $\beta = 1.3625$). The result provides unsatisfactory smoothing and  detection of discontinuities.
 \emph{(c)}~Decreasing the complexity penalty (here $\gamma= 0.02$) leads to more but dislocated discontinuities.
		\emph{(d)}~Increasing the elasticity parameter (here $\beta = 3$) leads to stronger smoothing, but to spurious segments as well.
		In either case, the first order model shows the tendency to create spurious discontinuities 
		at steep slopes which is known as the gradient limit effect.
	}
	\label{fig:GradientLimit_mixedEffects}
\end{figure}
As mentioned in the introduction,
a major limitation of the classical Mumford-Shah model is that 
data with locally linear or polynomial trends are not well approximated.
This undesirable effect, known as gradient limit effect, is illustrated in Figure~\ref{fig:GradientLimit_mixedEffects}.
We observe that the solution of the first order model (using model parameters optimized to the $\ell^2$ error)  
does not catch all discontinuities.
We also see that tuning the model parameters towards allowing for more discontinuities
leads to spurious discontinuities at steep slopes.
This shows that the first order models are not rich enough for dealing with 
data having locally linear or polynomial trends.

\subsection{Basic properties of higher order Mumford-Shah and Potts models}
\label{sec:basic_proporties}

We denote by $\nabla^k \in \R^{(q-k) \times q}$  the matrix that acts as $k$-th order finite difference on
the vector $u_I$ where $q= |I|.$
To fix ideas, the matrices $\nabla^k$ are given for $k=1$ and $k=2$ by
\[
	\nabla = \begin{pmatrix}
	-1  & 1    &  & & \\ 
	&	-1 &  1  &    &  \\ 
	&	     & \ddots   & \ddots & \\  
	&	&    & -1  & 1  \\ 
	\end{pmatrix} \in \R^{(q-1) \times q}
\qquad	\text{and} \quad
	\nabla^2 = \begin{pmatrix}
	1  & -2  & 1 &  & & \\ 
	&	1  & -2  & 1   & & \\ 
	&	  & \ddots  & \ddots   & \ddots & \\  
	&	& & 1  & -2  & 1  \\ 
	\end{pmatrix} \in \R^{(q-2) \times q}.
	\]
For higher orders $k \geq 3,$ 
the row pattern is equal to 
$t *  t * \ldots * t$ which denotes 
the $k$-fold convolution of the finite difference vector
$t = (-1, 1)$ with itself.

Using this notation the higher order Mumford-Shah problem \eqref{eq:penalizedProblem}
can be written as
\begin{equation}\label{eq:weakSplineDiscrete}
	(u^*, \Ic^*) = \argmin_{u \in \R^N,\,\Ic \text{ partition of }\Omega }~
	 \underbrace{\sum_{n=1}^N  (u_n - f_n)^2}_{\text{data penalty}} + \underbrace{\beta^{2k} \sum_{I \in \Ic}  \sum_{i =1}^{|I| - k} (\nabla^k u_I)_i^2 }_{\text{smoothness penalty}} + \underbrace{\gamma \, | \Ic |.}_{\text{complexity penalty}}
\end{equation}
As mentioned in the introduction,
the functional comprises a data penalty term,
a smoothness penalty term, and a complexity penalty term.
A minimizer $u^*$ is a $k$-th order discrete spline 
approximation to data $f$ on each segment of the partition $\Ic^*.$
The parameter $\gamma$ determines the penalty 
for opening a new segment.
The parameter $\beta$ controls influence of the smoothness penalty.
The order $k$ is the derivative order of the (discrete) spline.
Note that polynomials of order  $k-1$
on a segment do not have any smoothness penalty.

An illustration on the smoothing effect of higher order Mumford-Shah models 
in comparison to classical splines and to first order models is given in
\cref{fig:splineVsMS}.

\begin{figure}[t]
	\begin{subfigure}{0.32\textwidth}
		\includegraphics[width=1\textwidth,trim=25 14 0 0, clip]{Experiments/Synthetic_Data/Spline_Examples/Data.pdf}
	\caption{Noisy signal with discontinuities}
	\end{subfigure}
\hfill
	\begin{subfigure}{0.32\textwidth}
		\includegraphics[width=1\textwidth,trim=25 14 0 0, clip]{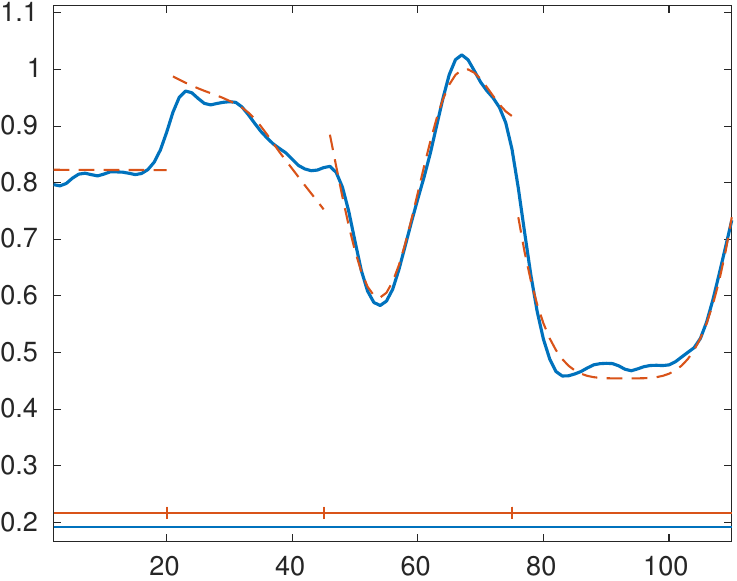}
		\caption{Smoothing spline}
	\end{subfigure}
\hfill
	\begin{subfigure}{0.32\textwidth}
		\includegraphics[width=1\textwidth,trim=25 14 0 0, clip]{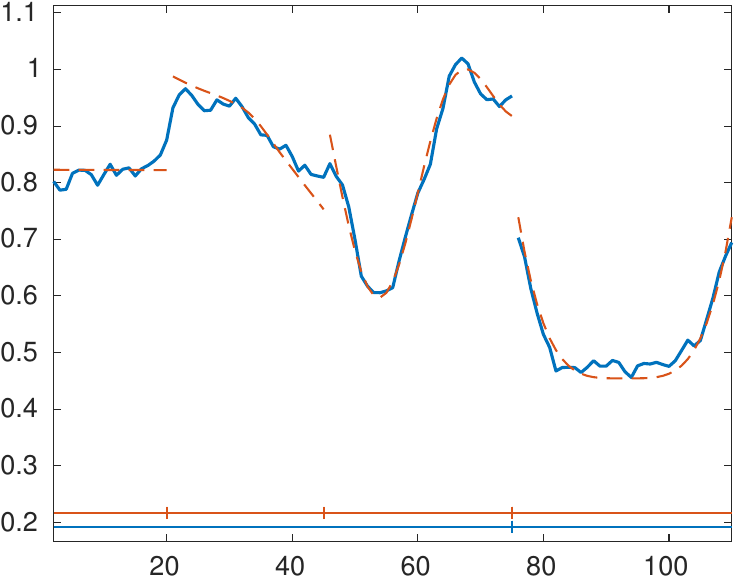}
		\caption{Classical Mumford-Shah $(\Pc_{1,\beta,\gamma})$}
	\end{subfigure}\\[0.7em]
	\begin{subfigure}{0.32\textwidth}
		\includegraphics[width=1\textwidth,trim=25 14 0 0, clip]{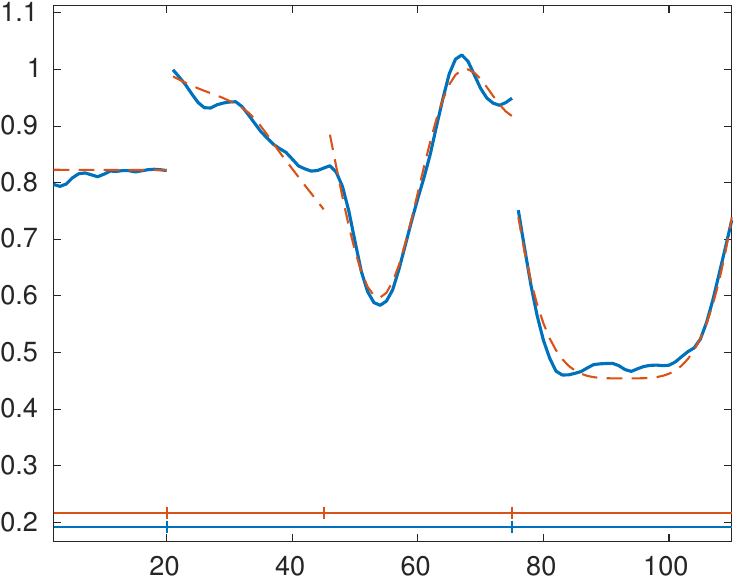}
		\caption{Higher order Mumford-Shah $(\Pc_{2,\beta,\gamma})$}
	\end{subfigure}\hfill
	\begin{subfigure}{0.32\textwidth}
		\includegraphics[width=1\textwidth,trim=25 14 0 0, clip]{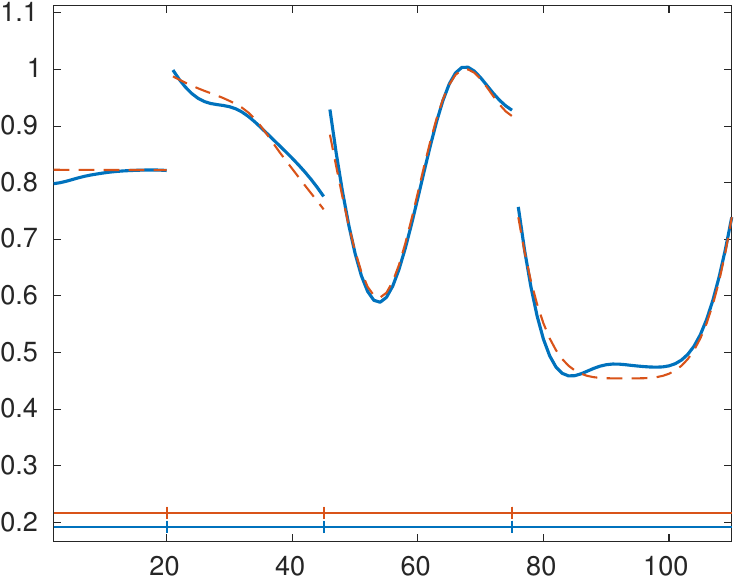}
		\caption{Higher order Mumford-Shah $(\Pc_{3,\beta,\gamma})$}
	\end{subfigure}\hfill
	\begin{subfigure}{0.32\textwidth}
		\includegraphics[width=1\textwidth,trim=25 14 0 0, clip]{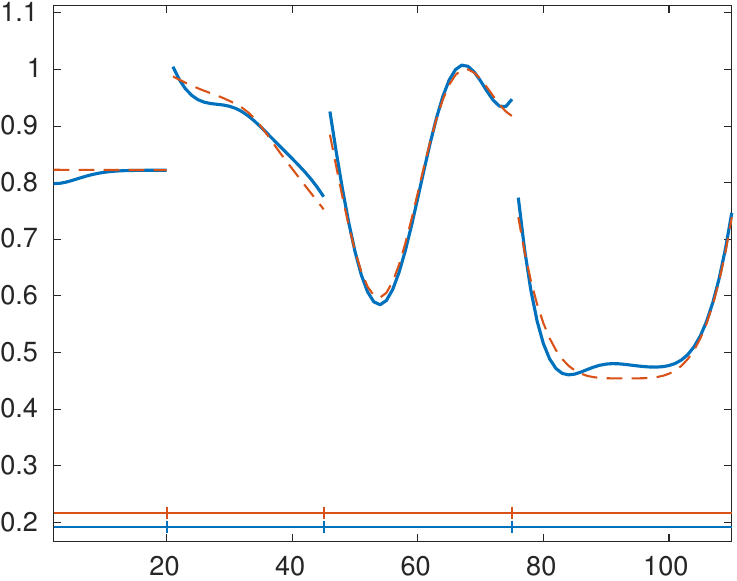}
		\caption{Higher order Mumford-Shah $(\Pc_{4,\beta,\gamma})$}
	\end{subfigure}	
	\caption{
		Smoothing a noisy signal with discontinuities   using various models. All model parameters are chosen with respect to optimal $\ell^2$-error.
		As in \cref{fig:GradientLimit_mixedEffects}, the ground truth 
		is depicted as red dashed line and the estimate as blue solid line.
		\emph{(b)}~Classical spline approximation smoothes out the discontinuities.
		 \emph{(c)}~The classical Mumford-Shah model allows for discontinuities, but the estimate misses most of them and the result remains noisy.
		 \emph{(d--f)}~The higher order Mumford-Shah results provide improved smoothing and segmentation.
		 In particular, the third and the fourth order models get the discontinuities of the groundtruth.
	}
	\label{fig:splineVsMS}
\end{figure}

\paragraph{Formulation as partitioning problem.} 
As in the first order case, it is convenient to formulate the higher order Mumford-Shah problem \eqref{eq:weakSplineDiscrete}
in terms of the partition only;
it reads
\begin{equation}\label{eq:partitioningProblem}
	\Ic^* = \argmin_{\Ic \text{ partition of } \dint{1}{N} } \sum_{I \in \Ic} (\Ec^I + \gamma),
\end{equation}
where $\Ec^{I}$ denotes the approximation error of 
the $k$-th order (discrete) smoothing spline on $I$ given by
\begin{equation}\label{eq:epsSpline}
	\Ec^I = \min_{v \in \R^{|I|}}~\|  v - f_I \|_2^2 +    \beta^{2k} \| \nabla^k v \|_2^2
			= \min_{v \in \R^{|I|}}~\sum_{i=1}^{|I|}  (v_i - (f_I)_i)^2 +   \sum_{i =1}^{|I| - k}  \beta^{2k} (\nabla^k v)_i^2.
\end{equation}
Note that the $k$-th order finite difference $\nabla^k$ is only 
well defined for vectors of length greater than $k,$
so $\Ec^{I} = 0$ if $|I| \leq k.$
The minimizing estimate $u^*$ can be recovered  from an optimal partition $\Ic^*$ 
by solving
\begin{equation}\label{eq:minimizerOnInterval}
	u_I^* = \argmin_{v \in \R^{|I|}}~\|  v - f_I \|_2^2 +    \beta^2 \| \nabla^k v \|_2^2, \quad\text{ for all }I \in \Ic^*.
\end{equation}
Hence, if we have computed an optimal partition $\Ic^*$ of the domain $\Omega,$
its accompanying signal estimate $u^*$ is uniquely determined.

To express the relation between  $u^*$ and $\Ic^*,$
it is convenient to introduce a formulation in terms of block matrices.
A partition $\Ic$ defines a block diagonal matrix  $L_\Ic$ by
\begin{equation}\label{eq:lpc}
	L_\Ic = 
	\begin{pmatrix}
		L_{|\Ic_1|} &   & &\\
		& L_{|\Ic_2|}    & &\\
		& &   \ddots &\\
		& &   & L_{|\Ic_M|} 
	\end{pmatrix},
	\qquad\text{with}\qquad
	L_n = \nabla^k \in \R^{(n-k) \times n},
\end{equation}
and with  $k$-th order finite difference matrices $\nabla^k$ of the appropriate size $(n-k) \times n$ defined as above.
Here, $|\Ic_m|$ denotes the cardinality of the $m$-th element of the partition $\Ic.$
If $n\leq k,$
we use the convention that $L_n$ is an \enquote{empty} block of length $n.$
The number of columns of $L_\Ic$ is equal to $N,$ 
and the number of rows depends on the size of the intervals with minimum length $k+1$ of the partition, i.e.,
 $\sum_{m=1}^M \max(|I_m|-k,0).$
 For example, for $k=2$ the partition $\Ic = \{ (1, \ldots, 4), (5, 6), (7, \ldots, 9)\}$
defines the matrix
\[
	L_\Ic = \begin{pmatrix}
	1  & -2  & 1 & & & & &\\ 
	   &  1  & -2  & 1 & & & & &\\ 
	& & &	& & & 1  & -2  & 1  \\ 
	\end{pmatrix} \in \R^{3 \times 9}.
	\]
	The block matrix notation \eqref{eq:lpc} allows to formulate 
the  minimization problem \eqref{eq:weakSplineDiscrete} 
in the compact form
  \begin{equation}\label{eq:minInUandI}
	\argmin_{u \in \R^N,\,\Ic \text{ partition}}~ \beta^{2k}\, \| L_\Ic u \|_2^2 + \| u - f\|_2^2 + \gamma\,|\Ic|.
\end{equation}
For a fixed partition $\Ic,$
taking derivatives with respect to $u$ reveals that 
a minimizer $u_{\Ic}$
satisfies the linear system
\begin{equation}\label{eq:minimizingProperty}
	2\beta^{2k} L_\Ic^T L_\Ic  u_{\Ic}  + 2( u_{\Ic} - f) = 0. 
\end{equation}
As the system has full rank for all $\beta \geq 0$
we get the unique solution
\begin{equation}\label{eq:GetUFromMatrix}
u_{f,\Ic} = S_{\Ic, \beta} f, \qquad\text{where } S_{\Ic,\beta} = (\beta^{2k} L_\Ic^T L_\Ic + \mathrm{id})^{-1}.
\end{equation}
We omit the subscript if the dependence on $\Ic$ or $\beta $ is clear.
Plugging
\eqref{eq:GetUFromMatrix}
 into \eqref{eq:weakSplineDiscrete}
 gives the  explicit expression 
 for the functional value of the higher order Mumford-Shah functional restricted to the partition~$\Ic$ as
\begin{align}\label{eq:functional_first}
	\Gc'_{\Ic}(f) =
	 \beta^{2k}\, \| L_\Ic S_\Ic f \|_2^2 + \| S_\Ic f - f\|_2^2
	+  \gamma\,|\Ic|.
\end{align}
Hence, the minimizing partition $\Ic^*$ 
is given as the minimizing argmuent of  $\Ic \mapsto \Gc'_{\Ic}(f).$

In contrast to the first order model,
expressing the problem only in terms of $u$ just like in \eqref{eq:firstOrderMS_u} is not feasible.
A reason for this is that
one solution $u$ may be the result of 
different partitions with different numbers of segments.
A simple example is the data $f = (0,1,0).$
For $\gamma < 2/3$ and for $\beta$ sufficiently large,
a minimizer is given by $u = f.$ 
The partitions $\big\{ (1,2), ( 3 )\big\},$
 $\big\{ (1), ( 2,3 )\big\}$ and $\big\{ (1), ( 2), (3 )\big\}$
lead to this $u = f.$

\paragraph{Minimum functional values and minimum segment lengths.}

We record the following basic property about minimizers.
Its proof follows an argument similar to the one used in \cite{blake1987visual} 
for a continuous domain second order problem.
\begin{lemma}\label{lem:functionalValueMinimizer}
Let $\Ic^*$ be a minimizing partition of \eqref{eq:weakSplineDiscrete}.
Then the minimal functional value is given by
\[
	\Gc'_{\Ic^*}(f) = \| f \|_2^2 - f^T S_{\Ic^*, \beta} f +  \gamma\,|\Ic^*|.
\]
\end{lemma}
\begin{proof}
Let $\tilde u = S_{\Ic^*, \beta} f.$
Expanding the functional  yields 
\[
\begin{split}
	\Gc'_{\Ic^*}(f) &= \beta^{2k} \| L_{\Ic^*} \tilde u \|_2^2 + \| \tilde u -f \|_2^2 + \gamma\, |\Ic^*| \\
		   &= \beta^{2k} \tilde u^T L_{\Ic^*}^T  L_{\Ic^*} \tilde u + (\tilde u -f)^T  (\tilde u -f) + \gamma\, |\Ic^*| \\
		   &= \beta^{2k} \tilde u^T L_{\Ic^*}^T  L_{\Ic^*} \tilde u + \tilde u^T
		   (\tilde u -f) +  f^T (\tilde u -f) + \gamma\, |\Ic^*| \\
		   &= \tilde u^T (\beta^{2k}  L_{\Ic^*}^T  L_{\Ic^*} \tilde u + (\tilde u -f)) -  f^T (\tilde u -f) + \gamma\, |\Ic^*|\\
		    &= -f^T (\tilde u -f)	+ \gamma\, |\Ic^*| = \| f \|_2^2 - f^T \tilde u + \gamma\, |\Ic^*|,
\end{split}
\]
where we used the minimality property \eqref{eq:minimizingProperty} in the penultimate line.
\end{proof}

Next we show that there is always an optimal partition $\Ic^*$
which has at most one segment with less than $k$ elements:
\begin{lemma}\label{lem:minimumSegmentLength}
	For each partition $\Ic$ there is a partition $\Ic'$
	such that all segments $I' \in \Ic'$ (except possibly the leftmost one) have length greater or equal than $k$ 
	and that 
	\[
		\sum_{I' \in \Ic'} \Ec^{I'} \leq \sum_{I\in \Ic} \Ec^I \quad \text{and}\quad |\Ic'| \leq |\Ic|.
	\]
	In particular $\Gc'_{\Ic'}(f) \leq \Gc'_\Ic(f).$
\end{lemma}
\begin{proof}
Let $\Ic$ be a partition and let $I$ be its right-most segment such that $|I| < k.$
Denote by $i$ the left boundary index of $I$.
If $i =1$ we are done.
Otherwise, we transfer the element $i-1$ from the left neighboring segment to the segment $I$
and denote the partition modified in this way by $\Ic'.$
(If the neighboring segment gets empty, we remove it from the partition.)
On the one hand, $|\Ic'| \leq |\Ic|.$
On the other hand, since $|I \cup \{i-1\}| \leq k$
we have that $\Ec^{I \cup \{i-1\}} = 0.$
Repeating the above procedure a finite number of times, 
we end up with a partition $\Ic''$ whose
 segments have length greater or equal to $k,$
 except possibly the leftmost segment. 
\end{proof}

\paragraph{Higher order Potts models.}

As mentioned in the introduction,
the higher order Potts model \eqref{eq:penalizedProblemPoly}
can be seen as the limit case of the higher order Mumford-Shah model for $\beta \to \infty.$
The main difference to the higher order Mumford-Shah model 
is that the approximation on a segment is
performed by a polynomial of maximum degree $k-1$
instead of a $k$-th order spline.
In consequence, the approximation error on a segment $I = \dint{l}{r}$ is given by
\begin{equation}\label{eq:defEpsLRPoly}
	\Ec^{I} =  \min_{\substack{v \text{ polynomial of}\\\text{degree $\leq k-1$ on } I }}
	\| v - f_I \|_2^2. 
\end{equation}
Complementing \eqref{eq:epsSpline} for the Mumford-Shah problem, 
\eqref{eq:defEpsLRPoly} is a least squares problem in the coefficients of the polynomial.

On the one hand, higher order Potts models are more restrictive
than genuine higher order Mumford-Shah models
since they enforce piecewise polynomial solutions.
On the other hand, due to the stronger prior,
they are more robust to noise.
From the computational side, one parameter less has to be determined 
for the Potts model.

Being the limit case $\beta \to \infty,$ 
the higher order Potts models has similar properties as the Mumford-Shah model.
In particular, if the data can be described by 
a polynomial of order $k-1$ on a segment,
then that segment does not get any approximation penalty.
In consequence, the assertion of 
\cref{lem:minimumSegmentLength} holds true for the higher order Potts models as well.

\subsection{Existence and uniqueness of minimizers}

It is straightforward to show the existence of minimizers.
\begin{theorem}
 The higher-order Mumford-Shah/Potts model \eqref{eq:penalizedProblem} has a minimizer
 for  each $k \in \N,$ $\gamma > 0,$  $\beta \in (0, \infty].$
\end{theorem}
\begin{proof}
	For a fixed partition $\Ic$, the problem \eqref{eq:penalizedProblem}
	reduces to least squares problems on the intervals of $\Ic$
	which all possess minimizers.
	As there are only finitely many partitions on $\Omega,$
	there is at least one solution with a minimal functional value.
\end{proof}

Uniqueness of the solution is more intricate.
The next example shows that 
the solutions of the higher order Mumford-Shah models
\eqref{eq:penalizedProblem} need not be unique.
For simplicity, we consider only the case $k= 2,$
 but analogous examples can be given for any order $k \geq 3.$
\begin{example}
Consider data $f = (0,1,0)$ and $k=2.$
The optimal signal corresponding to the partition $\Ic^3 = \{(1), (2), (3)\}$
is given by $u^3 = f$ and it has the functional value $2\gamma.$
The optimal solution of a partition with two segments 
is given by $u^2 = (0,1,0)$ as well, and $u^2$ has the lower functional value $\gamma.$
One can show that the the partition $\Ic^1 = \{ (1,2,3)\}$
yields the signal  $u^1 = \frac{1}{1+ 6 \beta^4} (2 \beta^4, 1+ 2\beta^4, 2\beta^4)^T$
and that the functional value is given by $\frac{4 \beta^4}{1+ 6 \beta^4}.$ 
Setting this equal to the energy of the two-segment solution, $\gamma,$
gives us the critical value
$
\frac{4 \beta^4}{1+ 6 \beta^4}  = \gamma
$ which is equivalent
to $\beta^4 = \gamma/(4 - 6 \gamma).$
Thus, for each $\gamma < 2/3$ 
there is $\beta >0$ such that  both the two-segment and the 
one-segment solutions are minimizers,
and that $u^1 \neq u^2.$
\end{example}

Fortunately, configurations as described above are very improbable, as we will see next.
As preparation we introduce 
a notion of equivalent partitions.
We say that two partitions $\Ic,\Jc$ are equivalent, i.e.,
\begin{align}
\Ic \sim \Jc \quad :\Leftrightarrow \quad
\left(
I \in \Ic \text{ and } |I|>k \Rightarrow I \in \Jc  
\quad \text{ and } \quad 
J \in \Jc \text{ and } |J|>k \Rightarrow J \in \Jc
\right),
\end{align}
if these partitions have the same intervals of minimum length $k+1.$ (The smaller intervals are irrelevant.)
Equivalent partitions $\Ic,\Jc$ define the same block matrices $L_\Ic,L_\Jc$, i.e., $L_\Ic = L_\Jc.$
Therefore, using \eqref{eq:GetUFromMatrix}, 
\begin{equation}
u_{f,\Ic} = S_{\Ic, \beta} f  = S_{\Jc, \beta} f = u_{f,\Jc}
\end{equation}
which tells that the minimizers w.r.t.\ the equivalent partitions $\Ic,\Jc$ are given by the same function. 
Together, each equivalence class of partitions $[\Ic]$ defines a unique restricted minimizer $u_{f,\Ic}.$ 
Further, to each $[\Ic]$ there is a unique matrix $L_\Ic.$ The latter correspondence is even one-to-one. Summing up,  
\begin{equation}\label{eq:bijective}
\text{both assigments} \quad  
[\Ic] \to L_\Ic,\,   [\Ic] \to u_{f,\Ic}    \quad
\text{are well-defined, and}\quad
[\Ic] \to L_\Ic \quad	 \text{is one-to-one.}
\end{equation}

In particular, the minimization problem \eqref{eq:minInUandI} 
may be recast in the form 
\begin{equation}\label{eq:minInUandIEquiv}
\argmin_{u \in \R^N,[\Ic]} F_{[\Ic]}(u), \quad \text{ where } \quad  
F_{[\Ic]}(u) =\beta^{2k}\, \| L_\Ic u \|_2^2 + \| u - f\|_2^2 + \gamma\,|[\Ic]|.
\end{equation}
Here, we let
\[
	|[\Ic]| = \min_{\Jc \in [\Ic]} | \Jc |.
\]
Using this notation, the functional  \eqref{eq:functional_first}
is well-defined w.r.t.\ the equivalence classes 
so that we can write
\begin{align}\label{eq:functional}
	\Gc_{[\Ic]}(f) =
	 \beta^{2k}\, \| L_\Ic S_\Ic f \|_2^2 + \| S_\Ic f - f\|_2^2
	+  \gamma\,|[\Ic]|.
\end{align}
With these preparations we get the following result on the uniqueness of minimizers:

\begin{theorem}\label{thm:AlmostEverywhereUniqueness}
Let $\gamma > 0,$ $\beta \in (0, \infty],$ and $k \in \N.$
 The minimizer $u^*$ of \eqref{eq:penalizedProblem} is unique for almost all input data $f.$
\end{theorem}
\begin{proof}		

Using the notation introduced right above, we may conclude that 
the solution of \eqref{eq:penalizedProblem} is unique for any  $f\in \mathcal{F}$ where the set $\mathcal{F}$ is given by  
\begin{align} \label{eq:TheCondition4Uniqueness} 
  f \in \mathcal{F}   \quad :\Leftrightarrow  &\quad  \text{ there is a partition $\Ic^\ast$ 
  	such that $  F_{[\Ic^\ast]}(u_{f,\Ic^\ast})  < F_{[\Ic]} (u_{f,\Ic})$ for all $\Ic \not\in  [\Ic^\ast].$   } 
\end{align}
We are going to show that the complement $\mathcal{F}^C$ of $\mathcal{F}$ in euclidean space is a negligible set in the sense that it has Lebesgue measure zero. Depending on the equivalence class of the partition $\Ic,$ we get that the minimal function value constraint to  $[\Ic]$ for data $f$ as 
is given by $\Gc_{[\Ic]}(f)$ defined in \eqref{eq:functional} as
	$\Gc_{[\Ic]}(f) =
	 \beta^{2k}\, \| L_\Ic S_\Ic f \|_2^2 + \| S_\Ic f - f\|_2^2
	+  \gamma\,|[\Ic]|.$
Hence, $\mathcal{F}^C \subset \mathcal{H},$ where  
\begin{align}\label{eq:defH}
\mathcal{H} = \{ f: 
\text{ there are $\Ic,\Ic'$ with $[\Ic] \neq [\Ic']$  such that } \Gc_{[\Ic]}(f)-\Gc_{[\Ic']}(f)= 0\}.
\end{align}
For fixed $\Ic,\Ic',$ both $\Gc_{\Ic},\Gc_{\Ic'}$ are quadratic forms w.r.t.\ the input $f.$ 
Since $[\Ic] \neq [\Ic']$ we have by \eqref{eq:bijective} that the quadratic form $f \mapsto \Gc_{[\Ic]}(f)-\Gc_{[\Ic']}(f)$ is nonzero.
Therefore, by the Morse-Sard theorem, the set $\{f:  \Gc_{[\Ic]}(f)-\Gc_{[\Ic']}(f)= 0 \}$ has Lebesgue measure zero. Forming the finite union w.r.t. $\Ic,\Ic',$ we see that $\mathcal H$ has Lebesgue measure zero. In turn,
the subset  $\mathcal{F}^C$ is a negligible set in the sense that it has Lebesgue measure zero which completes the proof.
\end{proof}

\subsection{Related models}
\paragraph{Relation to complexity-constrained models.}

Along with the complexity penalized models \eqref{eq:penalizedProblem}
it is natural to study the constrained variant
\begin{equation}\label{eq:constrainedProblem}\tag{$\Cc_{k, \beta,J}$}
(u^*, \Ic^*) = \argmin_{u \in \R^N, |\Ic| \leq J} \|  u - f \|_2^2 +  \beta^{2k}\sum_{I \in \Ic}  \| \nabla^k u_I \|_2^2. 
\end{equation} 
In fact, both variants are closely related:
Let us denote by $(u^J,\Ic^J)$ a solution of \eqref{eq:constrainedProblem}
for parameter $J.$
From the solutions for $J= 1, \ldots, N$, one can recover a solution  of \eqref{eq:penalizedProblem}
by simply choosing the solution $(u^{J^*},\Ic^{J^*})$  with the 
optimal functional value in \eqref{eq:penalizedProblem}; that is,
\[
	J^* \in   \argmin_{J=1, \ldots, N} ~\gamma J + \|  u^J - f \|_2^2 +  \beta^{2k}\sum_{I \in \Ic^J}  \| \nabla^k u^J_I \|_2^2.
\]
In \cite{blake1989comparison}, this relation was used for deriving a solver for
the first order problem
$(\Pc_{1, \beta, \gamma}).$

It is a particularly useful consequence of
 this relation that 
 the set of solutions of \eqref{eq:constrainedProblem} for 
 all  $J= 1, \ldots, N$,
can be used to compute minimizers of \eqref{eq:penalizedProblem}
for all $\gamma > 0$ simultaneously,
in the sense that we can determine a finite number of intervals for $\gamma$
where the corresponding solution does not change.

\paragraph{Relations to $\ell_0$-penalized problems.}

The classical first order Potts model ($\Pc_{1,\infty, \gamma}$)
can also be written in terms of $\ell_0$-\enquote{norm} of the target variable $u$ as
\begin{equation}\label{eq:potts_l0}
	u^* = \argmin_{u\in \R^N}~\gamma\, \|\nabla u \|_0 + \| u-f\|_2^2,
\end{equation}
where $\| v\|_0$   denotes the number of non-zero elements of a vector;
that is $\| v\|_0 = |\{n : v_n \neq 0\}|.$ 
We point out that 
plugging $\nabla^k$ in \eqref{eq:potts_l0}
does \emph{not} lead to an equivalent of the higher order Potts model ($\Pc_{k,\infty, \gamma}$); that is,  in general for $k \geq 2$
\begin{equation}\label{eq:kink}
	\argmin_{w\in \R^N}~\gamma \|\nabla^k w \|_0 + \| w-f\|_2^2, \neq u^* \text{}, 
	\quad\text{where $u^*$ is the minimizer of \eqref{eq:penalizedProblemPoly}.}
\end{equation}
For $k =2,$ the difference can be seen in the following example: 
Let $f = (-1,-1, 1,1).$
The optimal signal when restricting to 
the segmentation $\Ic^2 = \{ (1,2), (3,4)\}$ 
is given by $u_{\Ic^2} = f$ and thus the approximation error is equal to $0.$
Optimal signals with respect to other partitions with two elements yield a higher approximation error.
A simple calculation gives
that the best 
linear approximation on the one-segment partition $\Ic^1 = \{ (1,2,3, 4)\}$ is given by $u_{\Ic^1} = (-\frac{6}{5},-\frac{2}{5},\frac{2}{5}, \frac{6}{5})$
so that $\Ec^{(1:4)} = \frac{4}{5}.$
The functional values are given by
 $\frac{4}{5} + \gamma$ for $(\Ic^1, u_{\Ic^1})$ and by $2\gamma$ for $(\Ic^2, u_{\Ic^2}).$
Hence, ($\Pc_{2, \infty, \gamma}$) has the solution  $(\Ic^1,u_{\Ic^1})$ for $\gamma > \frac{4}{5},$ 
and $(\Ic^2,u_{\Ic^2})$ for $\gamma < \frac{4}{5}.$ 
(They are both optimal  for $\gamma = \frac{4}{5}$.)
In contrast, as $\|\nabla^2 u_{\Ic^1}\|_0 = 0$ and  
$\|\nabla^2 u_{\Ic^2}\|_0 = \|(2,-2)\|_0 = 2,$
the critical value for the model in \eqref{eq:kink}
is $\gamma = \frac{2}{5}.$
Thus,
the solutions of \eqref{eq:kink} for $k=2$ and ($\Pc_{2, \infty, \gamma}$)  are different 
for $\gamma \in (\frac{2}{5},\frac{4}{5}).$
The intuition behind that difference
is that in \eqref{eq:kink} for $k=2$ the number of kinks of $u$ are penalized,
 whereas in ($\Pc_{2, \infty, \gamma}$)  the number of changes in the 
 affine parameters are penalized. The model in \eqref{eq:kink} was studied in \cite{fearnhead2018detecting} for the case $k=2$.

It was observed in \cite{fortun2017fast}
that the second order Potts model 
can be formulated in terms of the $\ell_0$-\enquote{norm} 
of an affine parameter field.
For the higher order Potts model this can be accomplished as follows.
Let $C$ be a $\R^k$-valued function on $\Omega$
such that $C(n) = (a_0,\ldots,a_{k-1})^T$  describes a (column-) vector of polynomial coefficients
for each $n \in \Omega.$
Further, let 
$
\|\nabla C \|_0 = |\{ n : C(n) \neq C(n+1)\} |
$
count the number of changes of the polynomial parameter field $C.$ 
Then  the higher order Potts model can be formulated in terms of $C$ only:
\begin{equation}\label{eq:parameterField}
	C^* = \argmin_C ~\gamma\, \|\nabla C \|_0 + \sum_{n=1}^N \left( (1,n, \ldots, n^{k-1}) \,  C(n) - f_n\right)^2.
\end{equation}
A minimal partition $\Ic^*$ can be recovered 
by extracting the intervals of $C^*$ with a constant functional value.
A corresponding signal is obtained by $u^*_n = (1,n, \ldots, n^{k-1}) \,  C(n).$

\section{Fast and stable solver for higher order Mumford-Shah problems}
\label{sec:FastStableMumfordShahSolver}
 
We develop efficient and stable solvers for higher order Mumford-Shah and Potts problems \eqref{eq:penalizedProblem}
 for all $\gamma > 0,$ $\beta \in (0,\infty],$ and $k \geq 1.$
First we  recall a dynamic programming 
scheme commonly used for  partitioning problems. 
Then we develop a recurrence scheme for computing the 
required approximation errors which is key for the efficiency of the algorithm.
Eventually, we analyze the stability of the algorithm.

\subsection{Dynamic programming scheme for partitioning problems}
\label{section:dynamicProgrammingScheme}

Let us denote the functional in \eqref{eq:partitioningProblem} by $P,$ i.e.,
\[
	P(\Ic) = \sum_{I \in \Ic} (\Ec^{I} + \gamma).
\]
Note that the functional is well-defined also for a partition $\Ic$ on 
the reduced domain $\dint{1}{r},$ which we will utilize in the following.
Let the minimal functional value for the domain $\dint{1}{r}$ be denoted by
\[
P^*_r = \min_{\Ic \text{ partition on }\dint{1}{r}}  P(\Ic).
\]
The  value $P^*_r$ for the domain $\dint{1}{r}$ 
satisfies the Bellman equation
\begin{equation}\label{eq:recurrencePenalized}
    P^*_r 
    = \min_{l= 1, \ldots, r} \Big\{\Ec^\dint{l}{r} + \gamma + P^*_{l-1} \Big\},
\end{equation}
where we let $P^*_0 = 0.$
Recall that $\Ec^{\dint{1}{r}} = 0$ if $r-l+1 \leq k,$
so the minimum on the right hand side actually only has to be taken
over the values $l= 1, \ldots, r - k.$
By the dynamic programming principle,
we successively compute  $P^*_1,$ $P^*_2,$ until we reach $P^*_N.$ 
As our primary interest is the optimal partition $\Ic^*,$ rather than the minimal functional value $P^*_N$,
we keep track of a corresponding partition.
An economic way to do so is to store at step $r$ the minimizing argument $l^*$ 
of \eqref{eq:recurrencePenalized} as the value $J_r$ 
so that $J$ encodes the boundaries of an optimal partition;  see \cite{friedrich2008complexity}.

The above procedure has the complexity $\Oc(N^2 \phi(N))$
where $\phi$ is an upper bound for the effort of computing the
approximation errors $\Ec^I.$ 
The straightforward way to compute $\Ec^I$ 
is solving the least squares system \eqref{eq:epsSpline}.
This leads to  $\phi(N) = \Oc(N)$ as the involved matrices have a band structure. We develop a strategy that achieves
$\phi(N) = \Oc(1)$ in the next section.

In the following, we recall two strategies 
from \cite{storath2014fast} and \cite{killick2012optimal}, respectively,
to prune the search space.
In \cite{storath2014fast}, a pruning strategy was introduced by exploiting the relation $\Ec^{\dint{l}{r}} \leq \Ec^{\dint{l'}{r}}$ if $l' \leq l$. 
From \eqref{eq:recurrencePenalized} follows immediately that if the current value $P_r$ for $P_r^*$ satisfies
\begin{equation}
P_r < \Ec^{\dint{l}{r}} + \gamma \label{eq:AMpruning}
\end{equation}
for some $l$, one can omit checking all $l'<l$ for this $r$, thus, $P_r^* = P_r$. That is, we do not have to compute $\Ec^{l':r}$.
Another way to prune the dynamic program follows from the observation that the approximation errors satisfy the inequality
$
	\Ec^{\dint{l}{s}} + \Ec^{\dint{s+1}{r}} \leq \Ec^{\dint{l}{r}}, \text{for all } l \leq  s < r.
$
Killick et al.\,\cite{killick2012optimal}
deduced that if 
\begin{equation}\label{eq:KFpruning}
	P_s^* \leq P_l^* + \Ec^\dint{l+1}{s},
\end{equation}
then $l$ cannot be an optimal last changepoint at a future timepoint $r.$
That means,
the intervals $\dint{l+1}{r}$ for all $r = l+1, \ldots, N$
cannot be reached and consequently $l$ does not need to be considered again for any future timepoint 	$r$.

\subsection{Fast computation of the approximation errors for higher order Mumford-Shah problems}
\label{sec:comp_eps_mumfordShah}

Here, we develop a recurrence formula
for computing the $\Ec^\dint{l}{r}$ 
needed in \eqref{eq:recurrencePenalized}.
For notational simplicity, we describe
the basic scheme for the left bound $l = 1,$
i.e., computing $\Ec^\dint{1}{r}$ for $r = 1,\ldots,N.$
The procedure works analogously for any $l > 1.$ 

Recall that $\Ec^{\dint{1}{r}} = 0$ if $r \leq k$;
so we may assume that $r > k$ in the following.
Our starting point is to rewrite the minimization problem \eqref{eq:epsSpline} for $I = 1:N$
in matrix form as
\begin{equation}\label{eq:epsLRmatrix}
	\Ec^{\dint{1}{N}} = \min_{v \in \R^N}\| A v - y \|_2^2.
\end{equation}
Here,
\begin{equation}\label{eq:AmatrixGvector}
	A = \mtrx{ E_N \\
			\beta^k \nabla^k
	} 
	\in \R^{ (2N - k) \times N}
	\quad\text{and}\quad y = (f_1, \ldots, f_N, 0)^T \in \R^{ 2N - k},
\end{equation}
and $E_N$ represents the identity matrix of dimension $N.$
Further, determining $\Ec^{\dint{1}{r}}$ for $r < N$  amounts to solving 
the least squares problem of smaller size
\begin{equation}\label{eq:epsLRmatrix_small}
	\Ec^{\dint{1}{r}} =  \min_{v \in R^{r}}  \left\|A^{(r)} v - y^{(r)} \right\|_2^2,
\end{equation}
where $A^{(r)}$ is the submatrix of $A$ given by
\[
	A^{(r)} = \mtrx{A_{1:r, 1:r} \\ A_{(N+1:N+r-k), (1:r)}}, \quad \text{and} \quad y^{(r)} =  \mtrx{f_{1:r} \\ 0}.
\]
Note that we do not have to compute a minimizer $v^*$ of \eqref{eq:epsLRmatrix_small} 
to evaluate $\Ec^{\dint{1}{r}}$.
Instead, we develop a recurrence formula computing $\Ec^{\dint{1}{r}}$ directly
based on Givens rotations.
As preparation, we use the symbols $Q^{(r)}$ and $R^{(r)}$ to denote the QR decomposition of $A^{(r)},$ i.e.
\[A^{(r)} = Q^{(r)} \mtrx{R^{(r)} \\0}\] with an orthogonal matrix $Q^{(r)}$ and an upper triangular matrix $R^{(r)}.$
As the $\ell^2$-norm is invariant to orthogonal transformations we may represent $\Ec^{\dint{1}{r}}$ 
as
\begin{align}
	\Ec^{\dint{1}{r}}&=  \min_{v \in \R^{r}}  \left\|\mtrx{R^{(r)} \\ 0 } v - (Q^{(r)})^T y^{(r)} \right\|_2^2 \notag\\
	&= \min_{v\in\R^r} \| R^{(r)}v- ((Q^{(r)})^T y^{(r)})_{1:r}	\|^2_2
		+\| ((Q^{(r)})^T y^{(r)})_{r+1:2r-k} \|_2^2 \notag\\
	&= \| ((Q^{(r)})^T y^{(r)})_{r+1:2r-k} \|_2^2.  \label{eq:epsWithQR}
\end{align}
The first term in the second line vanishes since the corresponding linear system can be solved exactly.
All terms in the last line of \eqref{eq:epsWithQR} are explicitly given and do not involve minimization.
Our goal is to recursively compute  $\Ec^{\dint{1}{r+1}}$ 
without explicitly computing QR decompositions and without carrying out the summation
involved in the last line of  \eqref{eq:epsWithQR}.

The first step is the determination of recurrence coefficients.
To this end, assume that we have computed the QR decomposition 
of $A^{(r)}.$
We consider the auxiliary matrix $W^{(r)}$ containing the upper triangular matrix $R^{(r)}$
and the beginning of the $(N+r-k+1)$-th row of $A$:
\[
	W^{(r)} = \mtrx{R^{(r)} & 0 \\  0 & 1  \\ 0  & 0\\ \multicolumn{2}{c}{A_{N+r-k+1, (1:r+1)}} }.
\]
By the band structure of $A,$ 
 only the last $k+1$ entries of
$A_{N+r-k+1, (1:r+1)}$ are non-zero.
We aim at bringing $W^{(r)}$ to upper tridiagonal form using orthogonal transformations
without modifying the already present zeros.
To this end, we employ Givens rotations. (Note that Householder reflections would destroy the existing zero entries.)
Recall that a Givens rotation $G = G(j, m, \theta)$ is equal to the identity matrix with the $2\times 2$ submatrix $(G_{jj}, G_{jm}; G_{mj}, G_{mm})$
replaced by a planar rotation matrix; 
that is,
\begin{equation}
\label{eq:GivensRotation}
	G(j, m, \theta) =  \begin{pmatrix}   1   & \cdots &    0   & \cdots &    0   & \cdots &    0   \\
                      \vdots & \ddots & \vdots &        & \vdots &        & \vdots \\
	                     0   & \cdots & \cos(\theta)  	& \cdots &    \sin(\theta)  & \cdots &    0   \\
		              \vdots &        & \vdots & \ddots & \vdots   	 &        & \vdots \\
	                     0   & \cdots &   -\sin(\theta)   & \cdots  &\cos(\theta)& \cdots &    0   \\
                      \vdots &        & \vdots &        & \vdots & \ddots & \vdots \\
                         0   & \cdots &    0   & \cdots &    0   & \cdots &    1
       \end{pmatrix},
\end{equation}
where $\theta$ denotes the rotation angle.
In order to eliminate the matrix entry $a = A_{mj}$
 by the pivot element $b = A_{jj}$ 
 we use left multiplication by the Givens rotation $G(j,m,\Theta_{mj})$ with the parameters
\begin{equation}\label{eq:theta}
	\cos(\Theta_{mj}) = b/\rho, \quad \sin(\Theta_{mj}) = a/\rho
\end{equation}
where $\rho =\sign(b) \sqrt{a^2 + b^2}.$
Here, we have used the notation $\Theta_{mj}$ to denote
the rotation angle of the corresponding Givens rotation. 
Since  $G(j,m,\theta)$ only operates on the $j$-th and $m$-th row of a matrix
it does not destroy the zeros already present in other lines.
Hence, we eliminate the last row of $W^{(r)}$ 
by using $k+1$ Givens rotations with parameters chosen according to \eqref{eq:theta}
to obtain $R^{(r+1)}.$
This shows how to recursively compute $R^{(r+1)}$ 
given $R^{(r)}.$ 
The relevant quantities we need in the following 
are the rotation angles $\Theta_{mj}$
which serve as the recurrence coefficients. 

Having computed $\Theta_{mj},$
we now are able to carry out the error update step
from $\Ec^{\dint{1}{r}}$ to 
 $\Ec^{\dint{1}{r+1}}$ in $\Oc(1)$:
Assume that we have computed $\Ec^{\dint{1}{r}}$ 
and the vector $q^{(r)}$ defined by
\[
	q^{(r)} = (Q^{(r)})^T y^{(r)}.
\]
The  $q^{(r)}$ satisfy the recurrence relation
\begin{equation}\label{eq:recurrence_q}
	q^{(r+1)} = G^{(r+1)}  \mtrx{ q^{(r)}_{1:r} \\ f_{r+1} \\ q^{(r)}_{r+1:2r - k} \\ 0 },
\end{equation}
where $G^{(r+1)}$ denotes the elimination matrix composed of the above $k+1$ Givens rotations; that is,
\begin{equation}\label{eq:G_r_1}
	G^{(r+1)} = \prod_{j=1}^{k+1} G\Big(r-k+j,2(r+1)-k,\Theta_{N+r+1-k,r-k+j}\Big),
\end{equation}
where we use the convention $\prod_{j=1}^{k} Z_j = Z_kZ_{k-1}\cdots Z_{1}$.
Further, as $G^{(r+1)}$ only operates on the first $r+1$ lines and the last line of the vector on the right hand side of  \eqref{eq:recurrence_q} it follows that
\begin{equation}
	\| q^{(r+1)}_{r+1:2r+1-k} \|_2^2 = \| q^{(r)}_{r+1:2r-k} \|_2^2 + 
	(q^{(r+1)}_{2(r+1)-k})^2. 
\end{equation}
Therefore, the error update is given by
\begin{equation}
\label{eq:ErrorUpdate_new}
\Ec^{\dint{1}{r+1}}  =\Ec^{\dint{1}{r}} + (q^{(r+1)}_{2(r+1)-k})^2.
\end{equation}
To summarize the update scheme  consists of computing $q^{(k+1)}$ 
by \eqref{eq:recurrence_q} and updating $\Ec^{\dint{1}{r+1}}$ by \eqref{eq:ErrorUpdate_new}.

The errors $\Ec^{l:r}$  can be updated in the same 
fashion by applying the above procedure to the data $\tilde f = (f_l, \ldots f_{r}).$
An important practical aspect is that the recurrence coefficients $\Theta_{mj}$
do not depend on the data.
Thus, we only need to compute the $(N-k)(k+1)$ recurrence coefficients once
and can reuse them for computing all $\Ec^{l:r}.$ 

We briefly discuss the accuracy of the error update scheme \eqref{eq:ErrorUpdate_new}.
As Givens rotations are orthogonal 
they have the optimal condition number one. 
Hence, there is no inherent error amplification in the elimination steps.
The practical accuracy of the error update  is illustrated by the following numerical experiment. 
We compute the approximation errors of a polynomial of degree $k-1.$
As these are in the null space of $\nabla^k$ the approximation errors $\Ec^{1:r}$
are exact equal to $0$ for all $r = 1, \ldots, N.$ 
\cref{Fig Stability Mumford-Shah} shows that the proposed procedure reproduces
the exact results up to machine precision.
% !TEX root = ../../higherOrderBZ.tex
\begin{figure}[t]
	\centering
\begin{subfigure}{0.32\textwidth}
	\includegraphics[width=\textwidth]{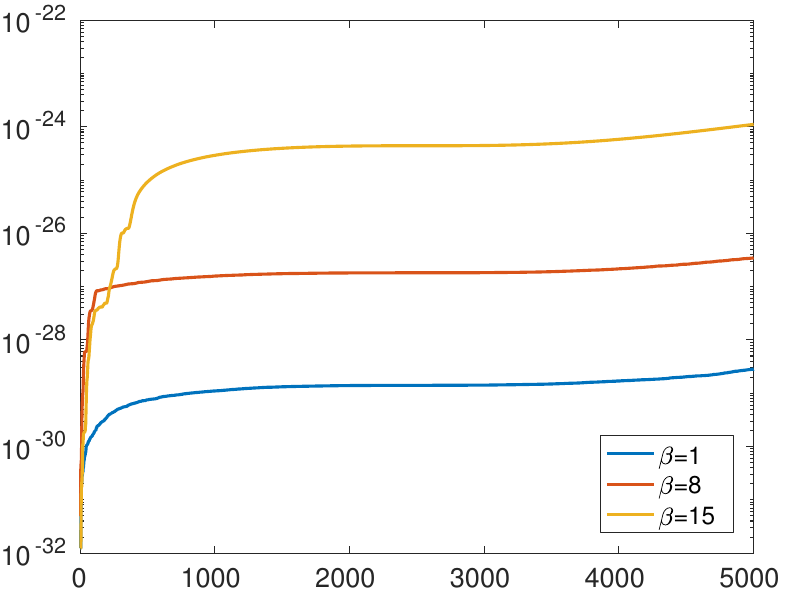}
	\subcaption{}
\end{subfigure}~~~
\begin{subfigure}{0.32\textwidth}
	\includegraphics[width=\textwidth]{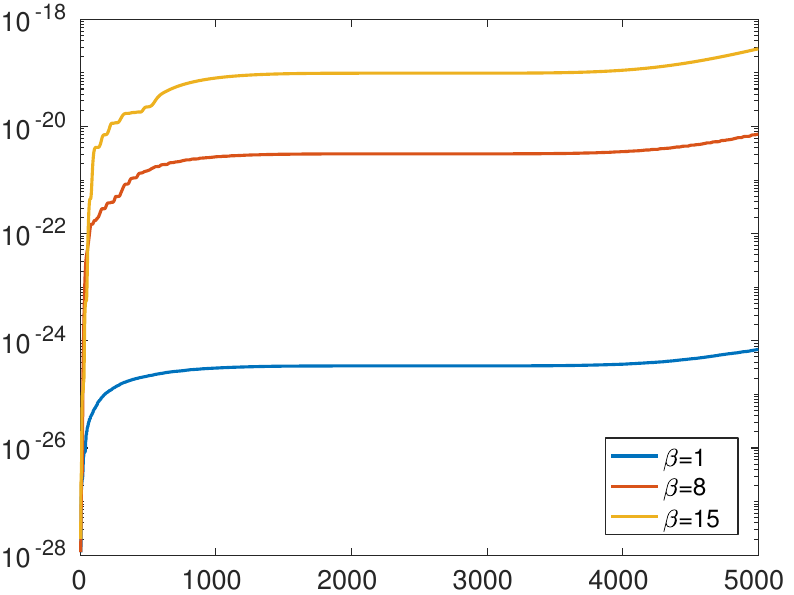}
	\subcaption{}
\end{subfigure}~~~
\begin{subfigure}{0.32\textwidth}
	\includegraphics[width=\textwidth]{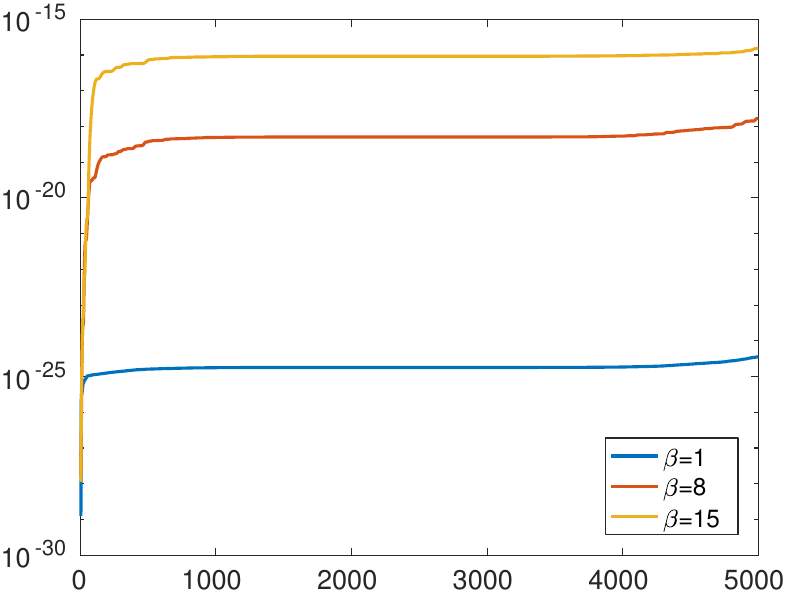}
	\subcaption{}
\end{subfigure}
		\caption{
\label{Fig Stability Mumford-Shah}
Approximation errors $\Ec^{\dint{1}{r}},$ $r=1,...,N,$ 
for \emph{(a)}  a linear polynomial and order $k=2,$ 
for \emph{(b)} a quadratic polynomial and order $k=3,$ 
and for \emph{(c)} a cubic polynomial and order $k=4.$ 
%The approximation error was computed 
%The values of $\beta$ are the same as in \cref{Fig Illustration of Alpha}.
The theoretical approximation errors are equal to zero;
the  approximation errors computed using the recurrence formula~\eqref{eq:ErrorUpdate_new} are accurate  up to machine precision.
}
\end{figure}

Next we explain how to include the two pruning strategies from \cref{section:dynamicProgrammingScheme}. 
The first strategy with condition \eqref{eq:AMpruning} requires to run 
over the $l$-index in a descending way, i.e. in the order $l = r, r-1,\ldots, 1$. On the other hand, 
the second strategy demands checking \eqref{eq:KFpruning} for all $1<l<r$ after $P_r^*$ was determined. Consequently, if \eqref{eq:AMpruning} holds for some $l$ at $(1:r)$, condition \eqref{eq:KFpruning} cannot be checked for $l' < l$ since $\Ec^{l':r}$ has not been computed yet. In order to overcome this issue without obliterating the first pruning, we  proceed as follows. At domain $(1:r)$, run through all $l$ in the descending list $L$
and update successively the corresponding approximation errors to $\Ec^{\dint{l}{r}}$. After each update step, check whether \eqref{eq:KFpruning} is satisfied for the current upper interval bound of the error and if so, delete $l$ from $L$ and start again with the next entry in $L$. By this, it is not necessary to  adapt the second pruning strategy essentially: check condition \eqref{eq:AMpruning} after testing if a not pruned $l$ is the current optimal last changepoint.
By combining the pruning strategies we effectively decrease the total number of error updates \eqref{eq:ErrorUpdate_new} that have to be performed; see \cref{Section Computing Time} for a numerical study.

We provide a pseudocode for the proposed
solver in the appendix (Algorithm~\ref{alg:solverPenalized}).
Let us summarize the above derivation:
\begin{theorem}
     Let $f\in \R^N,$ $k \in \N,$ and $\beta, \gamma  > 0.$ 
	 The proposed algorithm computes 
	 a global minimizer of \eqref{eq:penalizedProblem}. 
	 The worst case time complexity is $\Oc(N^2).$
\end{theorem}
\begin{proof}
It follows from the Bellman equation \eqref{eq:recurrencePenalized}
that the algorithm computes indeed a global minimizer.
The double loop over the the $l$ and $r$ indices 
has quadratic worst-case complexity.
It remains to show that for each $r \in (1:N)$
we can compute $\Ec^{\dint{r-k}{r}},$
$\Ec^{\dint{r-k-1}{r}}, $ \ldots , $\Ec^{\dint{1}{r}}$ 
in $O(1)$ per element. 
As each line of $A$ has at most $k+1$ entries, the elimination
of one line  requires $k+1$ elimination steps.
By the band structure of $A$ each 
elimination step by Givens rotations creates only new non-zeros 
in a band of $k+1$ entries
above the diagonal $A_{ii},$ $i=1, \ldots, N.$ 
Thus, computing the  recurrence coefficients
needs only $\Oc(k^2N)$ operations.
As applying a Givens rotation to a vector only needs a constant amount of operations,
the multiplication in \eqref{eq:recurrence_q} is in $O(k).$
(Note that the matrix $G^{(r+1)}$ is not explicitly created.)
Hence, executing the recurrence \eqref{eq:ErrorUpdate_new}
is $\Oc(k).$
It follows that computing the errors for all $\Oc(N^2)$ intervals sums up to $\Oc(kN^2).$
The reconstruction step from a partition is in $\Oc(kN)$
as it reduces to solving a least squares system
of band matrices whose 
number of rows sum up at most $2N - k.$
As $k$ is fixed the overall worst case time complexity
is $\Oc(N^2).$
\end{proof}

\subsection{Fast computation of the approximation errors for higher order Potts problems} \label{sec:potts_solver}

We describe a stable yet fast procedure to compute the
approximation errors for the higher order Potts problems \eqref{eq:defEpsLRPoly}.
To this end,  we first rewrite \eqref{eq:defEpsLRPoly} in terms of the polynomial coefficients $p \in \R^k$ as
\begin{equation}\label{eq:defEpsLRPoly2}
	\mathcal{E}^{l:r} =  \min_{p \in \R^k} \| B_{l:r, 1:r}p - f_{l:r} \|_2^2, 
	\end{equation}
where $B$ is the  $\R^{N \times k}$ matrix defined by
\begin{equation}
\label{eq:systemMatrixPoly}
	B = \mtrx{ 1 & 1 & \cdots  & 1^{k-1} \\ 
			   1 & 2 & \cdots  & 2^{k-1}  \\
			   \vdots & \vdots & & \vdots\\
			   1 & N-1 &  \cdots & (N-1)^{k-1} \\
			   1 & N & \cdots  & N^{k-1} 
	} \in \R^{N \times k}.
\end{equation}
As in \cref{sec:comp_eps_mumfordShah}, 
we describe the method for the prototypical case $l =1.$
 Furthermore, we assume that $r > k$ since otherwise $\Ec^{\dint{1}{r}} = 0$.
Denoting the submatrix $B_{1:r,1:k}$ by $B^{(r)}$  and its QR decomposition by $Q^{(r)}, R^{(r)},$
 we obtain in analogy to \eqref{eq:epsWithQR} that
\begin{equation}
\label{eq:epsWithQRPoly}
\begin{split}
\Ec^{\dint{1}{r}} &= \min_{p\in\R^k} 
\bigg\|
\begin{pmatrix} R^{(r)} \\ 0 \end{pmatrix}
p - (Q^{(r)})^T f_{1:r}
\bigg\|_2^2 = \| q^{(r)}_{k+1:r}\|_2^2,
\end{split}
\end{equation} 
where $q^{(r)}$ is given by
\[	
	q^{(r)} = (Q^{(r)})^T f_{1:r}.
\]
The recurrence coefficients for the error update
$\Theta_{r+1, j}$ for $ j = 1, \ldots, k$
are the Givens rotation angles 
for eliminating the entry $B_{r+1,j}$ with the pivot element $R^{(r)}_{j,j}.$
Now assume that we have computed $q^{(r)}$ and $\Ec^{\dint{1}{r}}.$
Then, $q^{(r+1)}$ can be expressed by the recurrence relation
\begin{equation}\label{eq:recurrence_q_Potts}
	q^{(r+1)} =  G^{(r+1)}\mtrx{q^{(r)} \\ f_{r+1}},
\end{equation}
where $G^{(r+1)}$ comprises the Givens rotations $G(j,r+1,\Theta_{r+1,j})$ for $j = 1, \ldots, k;$ that is,
\begin{equation}
 G^{(r+1)} =	\prod_{j=1}^k
G(j,r+1,\Theta_{r+1,j}),
\end{equation}
where we again use the convention $	\prod_{j=1}^k Z_j = Z_kZ_{k-1}\cdots Z_1$.
As $ G^{(r+1)}$ operates only on the first $k$ entries
and the last entry of $q^{(k)},$
we obtain by
\eqref{eq:epsWithQRPoly} 
\begin{equation}
\label{eq:ErrorUpdatePoly}
\begin{split}
\Ec^{\dint{1}{r+1}} 
= \| q^{(r+1)}_{k+1:r+1}\|_2^2 
= \| q^{(r)}_{k+1:r}\|_2^2 + (q^{(r+1)})_{r+1}^2 = 
\Ec^{\dint{1}{r}} + (q^{(r+1)})_{r+1}^2
\end{split}.
\end{equation}

\begin{remark} 
\begin{figure}[t]
	\centering
\begin{subfigure}[b]{0.49\textwidth}
\centering
	\includegraphics[width=0.85\textwidth]{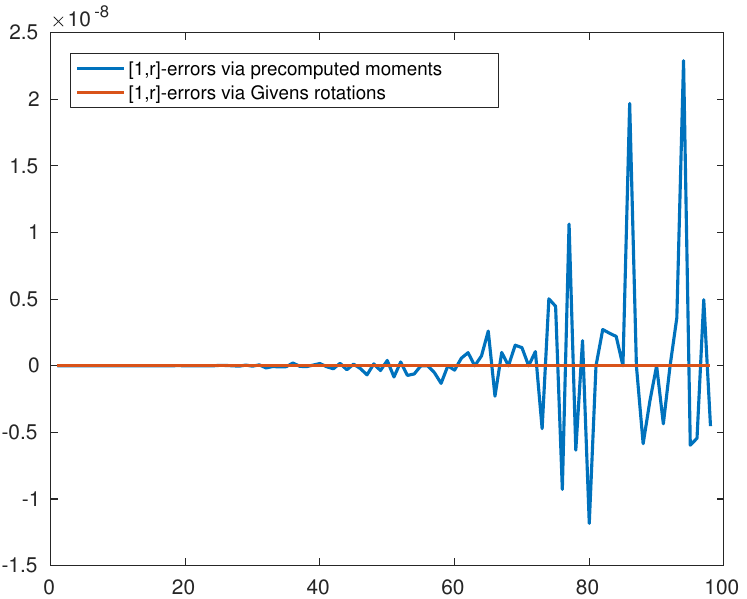}
\end{subfigure}
\begin{subfigure}[b]{0.49\textwidth}
\centering
	\includegraphics[width=0.844\textwidth]{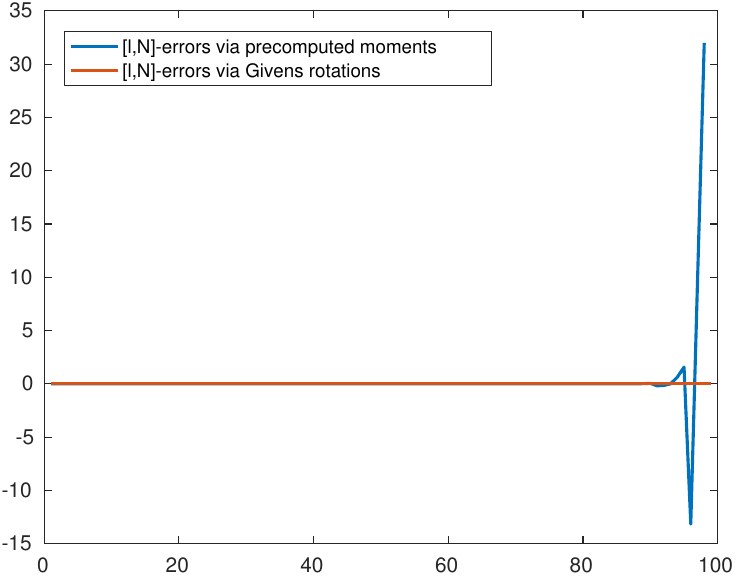}
\end{subfigure}
		\caption{
\label{Fig Stability Potts}
Approximation errors for a parabolic signal
% $n\mapsto n^2/100$, $n=0,...,N$, $N=100$ 
for the higher order Potts model of order $k=3.$ 
The graphs show the results
based on precomputed moments (blue) and based on the proposed scheme (red).
The true approximation errors are all equal to zero. 
\emph{Left:}~ Errors $\Ec^{1:r}$ for $r=1,...,N$. The computation based on precomputed moments is distorted beyond machine precision and gives even negative values.
\emph{Right:}~Errors $\Ec^{l:N}$ for $l=1,...,N-1$. The values derived from precomputed moments are strongly distorted when $l$ approaches $N.$ 
The proposed computation based on Givens rotations is accurate up to machine precision in either case.
}
\end{figure}

For the higher order Potts problems, 
there are also closed formulae for the evaluation of the errors $\Ec^{l:r}$
which one might consider to use directly.
Such formulae are derived in \cite{friedrich2008complexity} and in \cite{lemire2007better}
for the first and the second order Potts problem, respectively.
Using computer algebra, we have derived 
such formulae for $k = 3$ and $k=4.$ 
By precomputing moments,
the errors $\Ec^{l:r}$ can then be computed in $\Oc(1)$ per element.
The results are typically acceptable  for the piecewise 
 constant and piecewise affine linear  problems ($k = 1, 2$)
  and moderate signal lengths.
Unfortunately, 
for higher orders or longer signals, 
the approach based on the precomputation of moments is prone to numerical instability.
This is illustrated by the following experiment (cf.~\Cref{Fig Stability Potts}).
We consider the parabolic signal $f_n = n^2/100$, $n=0,...,N$
where $N = 100.$ 
The true approximation errors for the higher order Potts model of order $k=3$
are given by $\Ec^{l:r} = 0$ for all $l,r$ with $1 \leq l \leq r \leq N.$ 
\Cref{Fig Stability Potts} 
shows that
the results for $\Ec^{1:r}$ are distorted when using the approach based on the precomputation of moments,
in particular if $r$ is close to $N.$ 
The errors $\Ec^{l:N}$ are even more severely affected 
because of loss of significance.
We observe that -- in contrast to the moment precomputation approach --
the proposed method gives accurate results up to machine precision.
\end{remark}

\subsection{Stability results}

In this section, we investigate the stability of the proposed algorithm.
We start out with some basic lemmas we will need later on.
In the following, we consider the functional $\Gc_{[\Ic]}(f)$ defined in \eqref{eq:functional}
and omit the brackets and simply write $\Gc_{\Ic}(f).$

\begin{lemma}\label{lem:uniquePartitionInNhCond}
	We consider data $f \in \R^N.$
	If there is a partition $\Ic'$ and $\epsilon > 0$ such that  	 
	\begin{align}\label{eq:LowerWithEpsBand}
	\Gc_{\Ic'}(f)  <  \Gc_{\Ic}(f)    - \varepsilon   \quad \text{ for all } \quad \Ic \notin [\Ic']
	\end{align} 
	then there is an euclidean $\delta$-ball $B(f,\delta)$ 
	around $f$ such that for any $g \in B(f,\delta)$ holds:
	for data $g$ there is a unique optimal solution $u=u_{g,\Ic^*}$ of the problem \eqref{eq:penalizedProblem},
	and the corresponding partition $\Ic^*$ fulfills $[\Ic^*]=[\Ic'].$ We may choose $\delta$ by
	\begin{equation}\label{eq:ChooseDelta}
	\delta := \min \left( \frac{\varepsilon}{2   (2 \beta)^{2k} \ (\|f\| + 1/2)}, \frac{1}{2}             \right).
	\end{equation}	
\end{lemma}	

\begin{proof}
	The essential argument here relies on the continuity of the quadratic forms $x \mapsto \bar \Gc_{\Ic}(x)$ given by 
	\begin{align}\label{eq:functionalMainPart}
	\bar \Gc_{\Ic}(f) :=
	\beta^{2k}\, \| L_\Ic S_\Ic f \|_2^2 + \| S_\Ic f - f\|_2^2
	\end{align}
	which are the main parts of the $\Gc_{\Ic}$ 
	given by \eqref{eq:functional}.	 
	Each $\bar \Gc_{\Ic}$ may be represented w.r.t.\ the euclidean standard scalar product $\langle\cdot, \cdot \rangle  $ via a symmetric matrix $A_\Ic$ 
	as $\bar \Gc_{\Ic}(x) = \langle A_\Ic x  , x \rangle.$ The operator norm of $A_\Ic$ equals the norm of the corresponding bilinear form which in turn, since the $\bar \Gc_{\Ic}$ are positive (semi-definite), corresponds to
	\begin{align}\label{eq:def_A_I}
	  \|A_\Ic \| = \sup_{x:\|x\|=1} \bar \Gc_{\Ic}(x).  	  
	\end{align} 
	We first let $\delta'$ be defined by
	\begin{equation}\label{eq:ChooseDeltaPrime}
		\delta' := \min \left( \frac{1}{2} , \frac{\varepsilon}{2 \max_\Ic \|A_{\Ic}\| \ (\|f\| + 1/2) }             \right).
	\end{equation}
	We want to estimate $\Gc_{\Ic'}(f)$ from above for $g$ in a $\delta'$-ball around $f.$	 
	For brevity, we write $\Gc_{\Ic'}(g) = \bar \Gc_{\Ic'}(g) + \gamma N_{\Ic'}  $ where we let $N_{\Ic'} :=|[\Ic']|$. 
	Then we may estimate 
	\begin{align}
	\Gc_{\Ic'}(g) = \bar \Gc_{\Ic'}(g) + \gamma |[\Ic']|  
	              & = \bar \Gc_{\Ic'}(f) + \bar \Gc_{\Ic'}(f-g) - 2 \langle A_\Ic' f  , f-g \rangle + \gamma |[\Ic']| \notag \\
	              & <  \bar \Gc_{\Ic}(f)   - \varepsilon + \delta'^2 \ \|A_{\Ic'}\|   +  \delta' \ \|A_{\Ic'}\| \  \|f\|
	              + \gamma |[\Ic]|    \notag \\
	              & = \bar \Gc_{\Ic}(g) + \bar \Gc_{\Ic}(f-g) + 2 \langle A_\Ic f  , f-g \rangle 
	                  - \varepsilon + \delta'^2 \ \|A_{\Ic'}\|   +  \delta' \ \|A_{\Ic'}\| \  \|f\| + \gamma |[\Ic]|  \notag  \\
	              & \leq  \bar \Gc_{\Ic}(g)   - \varepsilon + 2 \delta' \ \max_\Ic \|A_{\Ic}\| \   \  ( \|f\|  + \delta') 
	              + \gamma |[\Ic]| \notag \\
	              & \leq  \bar \Gc_{\Ic}(g)   - \varepsilon + \varepsilon + \gamma |[\Ic]|
	              =  \bar \Gc_{\Ic}(g) + \gamma |[\Ic]|  = \Gc_{\Ic}(g). 	              \label{eq:prfLemStEq1}
	\end{align} 
	For the first inequality, we applied \eqref{eq:LowerWithEpsBand} for $\Gc_{\Ic'}(f)$ and used the assumption that $g \in B(f,\delta').$ For the second inequality, we employed \eqref{eq:ChooseDeltaPrime}.	
	In order to relate \eqref{eq:ChooseDeltaPrime} with \eqref{eq:ChooseDelta}, we now estimate  $\max_\Ic \|A_{\Ic}\|$ using basic spectral theory for self-adjoint bounded operators.
	Since $A_\Ic$ is the matrix representing the bilinear form corresponding to 
	$\bar \Gc_{\Ic},$ we may estimate using \eqref{eq:functionalMainPart}   that 
	\begin{align}\label{eq:estApart1}
		\| A_\Ic \|  \leq   \beta^{2k} \|L_\Ic^T L_\Ic\| \ \|S_\Ic\|^2 + \|I-S_\Ic\|^2, \quad \text{ for any partition }\Ic,
	\end{align}
	with the definitions of $L_\Ic$ given in \eqref{eq:lpc} and that of $S_\Ic$ given in \eqref{eq:GetUFromMatrix};
	here we only employed the triangle inequality and the submultiplicativity of operator norms.
	By \eqref{eq:GetUFromMatrix}, $S_{\Ic} =$ $(\beta^{2k} L_\Ic^T L_\Ic + \mathrm{id})^{-1}.$
	Since $L_\Ic^T L_\Ic$ is self-adjoint and positive, the spectrum  of $\beta^{2k} L_\Ic^T L_\Ic + \mathrm{id}$ is contained 
	in $[1,\infty).$ Hence, its inverse $S_{\Ic}$ has its spectrum contained in $[0,1].$ 
	Being again self-adjoint, and positive, $\|S_\Ic\| \leq 1.$ Further, since $S_{\Ic}$ has its spectrum contained in $[0,1],$    
	 $I-S_{\Ic}$ has its spectrum contained in $[0,1]$ as well. Then, with the same argument, $\|I- S_\Ic\| \leq 1.$
	In order to estimate $L_\Ic^T L_\Ic,$  we consider $L_\Ic$ in \eqref{eq:lpc}, and notice that $L_\Ic^T L_\Ic$ is block diagonal
	with entries consisting of convolutions of $k$th differences with themselves. Thus the row-sums as well as the column sums
	of $L_\Ic^T L_\Ic$ are bounded by $2^{2k}.$ The using the Schur criterion, the operator norm of $L_\Ic^T L_\Ic$ w.r.t. euclidean norm in the base space can be estimated by $2^{2k}.$  
	Summing up, we conclude invoking these estimates in \eqref{eq:estApart1} that 
	\begin{align}\label{eq:estApart2}
	\| A_\Ic \|  \leq   \beta^{2k} 2^{2k}  + 1, \quad \text{ for any partition }\Ic.
	\end{align}
	
	We now can show the assertion of the lemma. If $g \in B(f,\delta),$ then  $g \in B(f,\delta'),$
	by the estimate \eqref{eq:estApart2} relating \eqref{eq:ChooseDeltaPrime} with \eqref{eq:ChooseDelta}.
	In consequence, the estimate \eqref{eq:prfLemStEq1} applies to $g.$ Hence the solution for $g$ is unique 
	and given by  \eqref{eq:GetUFromMatrix}; in particular, the corresponding equivalence class of partitions equals $\Ic'$
	which shows the assertion. 	
\end{proof}

Next, we need a backward stability result for the QR algorithm \cite{Gentleman1973}. We present it adapted to our setup 
as needed later on.
In analogy to \eqref{eq:GetUFromMatrix},
we denote the linear mapping from $f_I$ (restricted to the interval $I$)
to the solution $u_I$ by $S_I.$ 

\begin{theorem} \label{thm:Stability}
	The QR algorithm $\tilde S_{\dint{l}{r}}$ needed for computing the  $\Ec^{\dint{l}{r}}$ is backward stable, i.e.,
	given data $f_{\dint{l}{r}}$ living on the subinterval $\dint{l}{r},$ there is a perturbation $\tilde f_{\dint{l}{r}}$ of $f_{\dint{l}{r}}$ 
	such that  
	\begin{equation}
	 \tilde S_{\dint{l}{r}} (f_{\dint{l}{r}}) = S_{\dint{l}{r}} (\tilde f_{\dint{l}{r}})  \quad \text{ with } \quad  \| \tilde f_{\dint{l}{r}} - f_{\dint{l}{r}} \|\leq \delta_{\dint{l}{r}}, 
	\end{equation}
	where $\delta_{\dint{l}{r}}$ depends on the machine precision $\tau$ and on the norm $\|f_{\dint{l}{r}}\|$ via
	\begin{equation}\label{eq:estDeltalrThm7}
		\delta_{\dint{l}{r}} \leq  6\tau \sqrt{r-l+1}\cdot\(\frac{9(r-l+1)-5}{4}-k\)(1+6\tau)^{3(r-l)-k} \| f_{\dint{l}{r}} \|.
	\end{equation}	
	Here, the QR algorithm is understood as in the analysis setup 
	of	\cite{Gentleman1973, wilkinson1965algebraic}.
\end{theorem}
\begin{proof}
	If $r-l < k,$ then $\tilde{S}_{\dint{l}{r}}(f_{\dint{l}{r}})=S_{\dint{l}{r}}(f_{\dint{l}{r}})=0$. So we may assume $r-l \geq k$.
	Recall that calculating $\Ec^{\dint{l}{r}}$ corresponds to computing the residual vector of  the least squares problem with system matrix $A\in\R^{2(r-l+1)-k\times (r-l+1)}$ from \eqref{eq:AmatrixGvector}  and data $f_{\dint{l}{r}}$. 
	In \cite{Gentleman1973} it is shown that 
	\begin{align*}
		&\tilde{R} = \bar{Q}^T(A+\Delta A),\quad \| \Delta A \|_F \leq \mu_{\dint{l}{r}}(\tau) \| A \|_F, \\
		 &\mu_{\dint{l}{r}}(\tau) = 6\tau \sqrt{r-l+1}\cdot\(\frac{9(r-l+1)-5}{4}-k\)(1+6\tau)^{3(r-l)-k},
	\end{align*}
	where $\tilde{R}$ is the computed upper triangular matrix by means of Givens rotations and note that $\bar{Q}^T$ is the orthogonal matrix that is the product of exact Givens rotations we apply. Analogously, for the data vector it is
	shown in \cite{Gentleman1973} that
	\begin{align*}
		\widetilde{Q^T f_{\dint{l}{r}}} = \bar{Q}^T(f_{\dint{l}{r}}+\Delta f_{\dint{l}{r}}), \quad \| \Delta f_{\dint{l}{r}}\| \leq \mu_{\dint{l}{r}}(\tau) \| f_{\dint{l}{r}} \|,
	\end{align*}
	hence
	$\tilde{S}_{\dint{l}{r}}(f_{\dint{l}{r}}) = S_{\dint{l}{r}}(f_{\dint{l}{r}}+\Delta f_{\dint{l}{r}})$ which implies \eqref{eq:estDeltalrThm7}.
\end{proof}

\begin{corollary}\label{cor:StableOnPartition}
	We consider bounded data $f \in \mathbb R^N,$ $\|f\|<C.$  
	For any partition $\Ic$ considered in the proposed algorithm for the higher order Potts and Mumford-Shah problem, 
	there is a perturbation $\tilde f$ of $f$ such that  
	\begin{equation}
	\tilde S_{\Ic} (f) = S_{\Ic} (\tilde f)  \quad \text{ where } \quad  \| \tilde f - f \|< \delta_\Ic(\tau),
	\end{equation}
	where $\delta_\Ic (\tau) ^2 = \sum_i \delta_{l_i:r_i} (\tau)^2$ depends on the machine precision $\tau$
    via the dependence of the  $\delta_{l_i:r_i}$ on $\tau$ given in Theorem~\ref{thm:Stability}
    and on $C,$ but not on $f.$
	More precisely, $\delta_\Ic(\tau)$ can be estimated from above by  
	\begin{equation}\label{eq:estDeltaThm8}
		\delta_\Ic(\tau)^2 \leq 
		36C^2\tau^2\sum_i (r_i-l_i+1)\(	\frac{9(r_i-l_i+1)-5}{4}-k \)^2(1+6\tau)^{6(r-l)-2k}.
	\end{equation}
\end{corollary}

\begin{proof}
	The statement is a consequence of Theorem~\ref{thm:Stability} since, for fixed partition, 
	the proposed algorithm computes optimal solutions 
	 $u_{f,\Ic} = S_{\Ic, \beta} f$ (cf. \eqref{eq:GetUFromMatrix})
	 using the QR algorithm on intervals.
	In particular, \eqref{eq:estDeltaThm8}  is a consequence of \eqref{eq:estDeltalrThm7}.	
\end{proof}

For the formulation of the next statement, we use the notation $\tilde \Gc_{\Ic}$ to denote the algorithm to compute  
the energy $\Gc_{\Ic}$ given by \eqref{eq:functional}. Further, we use the notation $g(\tau)$ to bound the approximation error between $\tilde \Gc_{\Ic}$ and $\Gc_{\Ic}$ for all $\Ic$ in dependence of the precision $\tau.$

\begin{proposition}\label{prop:Stab4EpsConditionOnF}
	We consider bounded data $f \in \R^N,$  $\|f\|<C$,  and assume that \eqref{eq:LowerWithEpsBand} is fulfilled for $f.$
	Let 
	\begin{equation} \label{eq:DeltaAstTauProp9}
		 \delta^*(\tau) = \max_\Ic \delta_\Ic (\tau) \leq 6C\tau N^{\frac{3}{2}} \(\frac{9N-5}{4}-k \)(1+6\tau)^{3(N-1)-k}
	\end{equation}	 
	for $\delta_{\mathcal{I}}(\tau)$ in Corollary \ref{cor:StableOnPartition} and assume that 
    $\tau$ is small enough such that $\delta^*(\tau) < \delta/2$ 
    with $\delta$ given by \eqref{eq:ChooseDelta}
    and such that $g(\tau)\leq \varepsilon/4$ with $\varepsilon$ given in \eqref{eq:LowerWithEpsBand}. 	
	Then, the  higher order Potts and Mumford-Shah problem \eqref{eq:penalizedProblem} has a unique minimizer $u_f,$ 
	and the proposed algorithm for computing this minimizer of the higher order Potts and Mumford-Shah problem \eqref{eq:penalizedProblem} is backward stable in the sense that  
	\begin{equation}\label{eq:StableUnderCond}
	\tilde u_f = u_{\tilde f} \quad \text{ where } \quad  \| \tilde f - f \|< \delta^*(\tau). 
	\end{equation}
	Here, $\tilde u_f$ is the result produced by the proposed algorithm for data $f$ and $u_{\tilde f}$ is the (unique) solution of 
	the higher order Potts and Mumford-Shah problem \eqref{eq:penalizedProblem}
	for perturbed data $\tilde f.$	 
\end{proposition}

\begin{remark}
  A more explicit relation of $\varepsilon$ and the precision $\tau$ without using the $\delta$'s   
 sufficient for the assumptions of Proposition~\ref{prop:Stab4EpsConditionOnF} to hold is given by
 \begin{align}
	\label{eq:delta*1}
 (1+6\tau)^{3N-k}\tau &< \frac{\varepsilon}{12CN(2\beta)^{2k}(C+\frac{1}{2})\(\frac{9N-5}{4}-k\)}
\quad\text{if } \epsilon \leq (2\beta)^{2k}(C+\frac{1}{2}),
 \\
 \label{eq:delta*2}
 (1+6\tau)^{3N-k}\tau &<\frac{1}{12CN\(\frac{9N-5}{4}-k\)}
 \qquad\qquad\qquad ~~
\text{if } \epsilon > (2\beta)^{2k}(C+\frac{1}{2}),
 \\
\label{eq:g(tau)}
\mu_{[1,N]}(\tau) &< \frac{1}{2} \( \frac{4C^2N+\epsilon}{C^2N} \)^{\frac{1}{2}}-1
 \end{align}
  w.r.t. $\mu_{[1,N]}(\tau)$ from the proof of \Cref{thm:Stability}. Conditions \eqref{eq:delta*1} and \eqref{eq:delta*2} are sufficient for $\delta^*(\tau) < \delta /2$ which is an immediate implication of combining
\eqref{eq:LowerWithEpsBand} and \eqref{eq:DeltaAstTauProp9}. 
From \eqref{eq:g(tau)} follows $g(\tau) \leq \epsilon/4$ since: 
 for any admissible $\dint{l}{r}$ we have
 \begin{align*}
	  \abs{\| \widetilde{Q^Tf_{\dint{l}{r}}} \|^2 - \| Q^Tf_{\dint{l}{r}} \|^2 } &\leq
	  \( \|Q^Tf_{\dint{l}{r}} \| + \|\widetilde{Q^Tf_{\dint{l}{r}}} \| \)
	  \abs{\| Q^Tf_{\dint{l}{r}}\| - \| \widetilde{Q^Tf_{\dint{l}{r}} }\|  }\\
	  &\leq	  \( C + \|\widetilde{Q^Tf_{\dint{l}{r}}} \| \) 
			  \| Q^Tf_{\dint{l}{r}} - \widetilde{Q^Tf_{\dint{l}{r}}}\|
 \end{align*}
 and 
 \[
	 \| \widetilde{Q^Tf_{\dint{l}{r}}} \| \leq \| Q^Tf_{\dint{l}{r}}\| + \| Q^Tf_{\dint{l}{r}} -\widetilde{Q^Tf_{\dint{l}{r}}}\|
	 \leq C+\| Q^Tf_{\dint{l}{r}} -\widetilde{Q^Tf_{\dint{l}{r}}}\|
	 \leq C+ \mu_{[1,N]} C.
 \]
Combining both yields
 \begin{align*}
	 g(\tau) &= \abs{
		 	\tilde{\Gc}_\Ic(f) - \Gc_\Ic(f)
	 	  }
	 	    \leq \sum_{\dint{l}{r}\in\Ic} \abs{ \tilde{\Ec}^{\dint{l}{r}}-\Ec^{\dint{l}{r}}}
	 	    \\&\leq \sum_{\dint{l}{r}\in\Ic} \abs{\| \widetilde{Q^Tf_{\dint{l}{r}}} \|^2 - \| Q^Tf_{\dint{l}{r}} \|^2 }
	 	    \leq NC^2 \mu_{[1,N]}(\tau) \(2+\mu_{[1,N]}(\tau) \).
 \end{align*}
An easy computation shows that \eqref{eq:g(tau)} is equivalent to requiring the latter to be smaller than $\epsilon /4$.
\end{remark}

\begin{proof}[Proof of Proposition~\ref{prop:Stab4EpsConditionOnF}]	
	By the proof of Theorem~\ref{thm:AlmostEverywhereUniqueness}, the solution $u_f$ of \eqref{eq:penalizedProblem} is unique for data $f.$ We denote the equivalence class of partitions corresponding to this optimal solution $u_f$ by its representer $\Ic'.$
	As a first step, we show that the solution $\tilde u_f$ computed by the proposed algorithm for data $f$ has partition $\Ic'$
	as well. 
	To that end, we first notice that by Corollary~\ref{cor:StableOnPartition}, 
	there is $\tilde f_\Ic$ with  $\| \tilde f_\Ic - f \| < \delta_{\Ic}(\tau)$  
	such that  
	$ \tilde S_{\Ic} (f)  =  S_{\Ic} (\tilde f_\Ic),$ for any partition $\Ic$ which is considered by the algorithm.
	In particular, using the notation of 
    Theorem~\ref{thm:Stability},
	\begin{align}\notag
	\|\tilde S_{\dint{l}{r}} (f_{\dint{l}{r}}) - S_{\dint{l}{r}} (f_{\dint{l}{r}}) \|    
		& =	\| S_{\dint{l}{r}} (\tilde f_{\dint{l}{r}}) - S_{\dint{l}{r}} (f_{\dint{l}{r}}) \|     \\
		& \leq \| S_{\dint{l}{r}} \|  \	\| \tilde f_{\dint{l}{r}} - f_{\dint{l}{r}} \| \leq 	
		\| \tilde f_{\dint{l}{r}} - f_{\dint{l}{r}} \| <   \delta_{\dint{l}{r}}. \label{eq:EstPertOnInterval}
	\end{align}
	For the second before last inequality, we used that  $\| S_{\dint{l}{r}} \| \leq 1$ which we have shown in the proof of Lemma~\ref{lem:uniquePartitionInNhCond}. 
	In consequence, summing over all intervals of $\Ic$  of length at least $k+1,$
	we obtain from \eqref{eq:EstPertOnInterval} that 
	\begin{equation}
	   \|\tilde S_{\Ic} (f) - S_{\Ic} (f) \| = \| S_{\Ic} (\tilde f_\Ic) - S_{\Ic} (f) \| 
	      \leq  \|\tilde f_\Ic - f \|  < \delta_\Ic(\tau).
	\end{equation}	
	for any partition $\Ic$ which is considered by the algorithm.
	Then, using the notation $\tilde \Gc_{\Ic}$ for the algorithmic variant of $\Gc_{\Ic},$
	we have (with the notation as in Lemma~\ref{lem:uniquePartitionInNhCond}) that
   	\begin{align}
   	\tilde \Gc_{\Ic'}(\tilde f_{\Ic'}) 
   	&\leq \Gc_{\Ic'} (\tilde f_{\Ic'}) + |\tilde \Gc_{\Ic'}(\tilde f_{\Ic'}) - \Gc_{\Ic'} (\tilde f_{\Ic'}) | \notag \\
   	&\leq \Gc_{\Ic'} (\tilde f_{\Ic'}) + g(\tau) \notag  \\  
   	&< \Gc_{\Ic}(f)   - \varepsilon + \delta_{\Ic'}(\tau)^2 \ \|A_{\Ic'}\|   +  \delta_{\Ic'}(\tau) \ \|A_{\Ic'}\| \  \|f\|  + g(\tau)  \notag \\  	 	 
   	&\leq \Gc_{\Ic}(\tilde f_\Ic)   - \varepsilon + 2 \delta^\ast(\tau) \ \max_\Ic \|A_{\Ic}\| \   \  ( \|f\|  + \delta^\ast(\tau)) 
   	+ g(\tau) \notag  \\
   	&\leq  \Gc_{\Ic}(\tilde f_\Ic) - \varepsilon +  \varepsilon/2  +  g(\tau) \notag \\   	
   	&\leq \tilde \Gc_{\Ic}(\tilde f_\Ic) - \varepsilon +  \varepsilon/2  + 2 g(\tau)
   	\leq \tilde \Gc_{\Ic}(\tilde f_\Ic).  \label{eq:EstFbyAlgo}
   	\end{align} 
   	Here, the third inequality is the central estimate which is obtained in analogy to the first part of the computation in \eqref{eq:prfLemStEq1}
    replacing the role of the vector $g$ there (not to be confused with $g(\tau)$) by that of the perturbation $\tilde f_\Ic'$ of $f$ here. The fourth inequality  is obtained in analogy to the second part of the computation in \eqref{eq:prfLemStEq1}
    with the role of the vector $g$ there replaced by the perturbation $\tilde f_\Ic$ of $f.$ 
    The second before last and last inequality follow by our assumptions made on $\tau.$
	Together, \eqref{eq:EstFbyAlgo} tells us 
	that the solution $\tilde u_f$ computed by the proposed algorithm has partition $\Ic'$ and  
	\begin{equation}\label{eq:FirstStepStabPro}
	      \tilde u_f   =   \tilde S_{\Ic'} (f).   
	\end{equation}
	Using again Corollary~\ref{cor:StableOnPartition}, we have 
	 \begin{equation}\label{eq:SecondStepStabPro}
	   \tilde S_{\Ic'} (f)  =  S_{\Ic'} (\tilde f) \quad \text{ for } \quad  \| \tilde f - f \|
	   < \delta_{\Ic'}(\tau)  \leq \delta^*(\tau), 
	 \end{equation}	
	with the perturbation $\tilde f$ of $f.$    
	We have that $ \| \tilde f - f \|<  \delta^*(\tau) < \delta $ with  
	$\delta$ defined by \eqref{eq:ChooseDelta}. 
	Therefore, we may now employ 
    Lemma~\ref{lem:uniquePartitionInNhCond}
    to conclude that
    the  solution of 
    the higher order Potts and Mumford-Shah problem \eqref{eq:penalizedProblem}
    denoted by $u_{\tilde f}$ agrees with the optimal solution for the partition $\Ic'$  
    which we have denoted by $u_{\tilde f,\Ic'} =  S_{\Ic} (\tilde f),$ i.e.,
    \begin{equation}
    	u_{\tilde f} = u_{\tilde f,\Ic'} =  S_{\Ic'} (\tilde f).
    \end{equation}    
    Combined with \eqref{eq:FirstStepStabPro} and \eqref{eq:SecondStepStabPro}, this shows \eqref{eq:StableUnderCond} which completes the proof. 	
\end{proof}

\begin{lemma}\label{lem:MeasureEpsSetQuadForm}
   We consider a nonzero quadratic form $H$ in a ball of radius $C$ in $\mathbb R^N$.
   Then, the Lebesgue measure $\lambda$ of the set 
   $H_{\varepsilon,c} = \{x: \|x\| \leq C, c-\varepsilon < H(x) < c + \varepsilon \}$
   fulfills	
   \begin{equation}
   \lambda(H_{\varepsilon,c}) \leq 2 \sqrt{ \frac{\varepsilon}{ \| A \|}} \ C^{N-1} 
   \end{equation}
   where $\|A\|$ denotes the spectral norm of the representing matrix $A$ of $H.$    
\end{lemma}

\begin{proof}
Without loss of generality, we may use a orthogonal transformation of the coordinate system to represent 
$H$ by $H(x) = \sum_i \alpha_i x_i^2$ with the eigenvalues $\alpha_i$ of the corresponding representing matrix of the quadratic form.
We sort the $\alpha_i$ by modulus, i.e., $|\alpha_1| \geq |\alpha_2| \geq \ldots.$ 
With repect to this coordinate system, we consider the $C$-ball with respect to the 
infinity norm $D=\{x: \|x\|_\infty \leq C\}$. We distinguish the eigenvalue $\alpha_1$ of highest modulus
which agrees with the norm $\|A\|$ of the representing matrix $A$ of $H.$
We estimate the Lebesgue measure of $\{x: c-\varepsilon < H(x) < c + \varepsilon \}$ on the larger set $D$
which provides an upper bound for that of $H_{\varepsilon,c}.$
To this end, we notice that,
for fixed $x_2,\ldots,x_N,$ we may estimate the univariate Lebesgue measure $\lambda^1$ of the section 
$\{x_1:  \leq C, c-\varepsilon < H(x_1, x_2,\ldots,x_N ) < c + \varepsilon\}$
\begin{equation}
  \lambda^1 \left(
  c + \sum_{i=2}^N  \frac{\alpha_i}{|\alpha_1|} x_i^2  - \varepsilon    
  < \sign(\alpha_1)  \ x_1^2 <  
   c + \sum_{i=2}^N  \frac{\alpha_i}{|\alpha_1|} x_i^2 + \varepsilon
  \right)
  \leq 2 \sqrt{ \frac{\varepsilon}{|\alpha_1|}}.
\end{equation}
(Notice that if $\alpha_1= 0$ the quadratic form would be zero.)
Hence, on $D$,
the Lebesgue measure of $\{x: c-\varepsilon < H(x) < c + \varepsilon \}$ is bounded by 
$2 \sqrt{ \frac{\varepsilon}{|\alpha_1|}} \ C^{N-1}$ which implies the assertion of the lemma. 
\end{proof}

\begin{theorem}
  Let $\varepsilon >0$ be given and assume that the precision $\tau$ fulfills the assumptions of 
  Proposition~\ref{prop:Stab4EpsConditionOnF}. We consider the set of bounded data $\{f:\|f\|\leq C\}$ in $\mathbb R^N$ for some $C>0.$ 
  Then, up to a set of Lebesgue measure $2 {\binom{\sigma_{N,k}}{2}} \sqrt{ \frac{\varepsilon}{ \sup_\Ic \| A_\Ic \|}} \ C^{N-1}$, $A_\Ic$ given by \eqref{eq:def_A_I},
  $\sigma_{N,k}$ the number of different means to choose intervals of length at least $k+1$ from a (discrete) set of length $N,$
  the proposed algorithm for computing a minimizer of the higher order Potts and Mumford-Shah problem \eqref{eq:penalizedProblem} is backward stable in the sense that  
   \begin{equation}\label{eq:StableUnderCond2}
   \tilde u_f = u_{\tilde f} \quad \text{ where } \quad  \| \tilde f - f \|< \delta^*(\tau), 
   \end{equation}
   where $\delta^*(\tau)$ is given by \eqref{eq:DeltaAstTauProp9}.
   Here, $\tilde u_f$ is the result produced by the proposed algorithm for data $f$ and $u_{\tilde f}$ is the (unique) solution of 
   the higher order Potts and Mumford-Shah problem \eqref{eq:penalizedProblem}
   for perturbed data~$\tilde f.$	 
\end{theorem}

\begin{proof}
   We proceed similar to the proof of 
   Theorem~\ref{thm:AlmostEverywhereUniqueness}
   to show that the set of those data which do not fulfill \eqref{eq:LowerWithEpsBand}	
   have Lebesgue measure smaller or equal to $2 {\binom{\sigma_{N,k}}{2}} \sqrt{ \frac{\varepsilon}{ \sup_\Ic \| A_\Ic \|}} \ C^{N-1}.$
   We choose two different partitions $\Ic,\Ic'$ with $[\Ic] \neq [\Ic'],$
   i.e., their equivalence classes do not agree, and consider the corresponding quadratic forms  
   $\Gc_{\Ic},\Gc_{\Ic'}.$  
   Their difference $\Gc_{\Ic}-\Gc_{\Ic'}$ is again a quadratic form (plus a constant).
   By Lemma~\ref{lem:MeasureEpsSetQuadForm}, the set where $\Gc_{\Ic}$ and $\Gc_{\Ic'}$ 
   are closer than $\varepsilon$ has Lebesgue measure $2 \sqrt{ \frac{\varepsilon}{|\alpha_1|}} \ C^{N-1}.$  
   Iterating this for all ${\binom{\sigma_{N,k}}{2}}$ different bilinear forms $\Gc_{\Ic}-\Gc_{\Ic'}$
   shows that the Lebesgue measure of those data where  \eqref{eq:LowerWithEpsBand} is not fulfilled can be estimated from above
   by the quantity written in the formulation of the theorem. 
   To the complementary set, we may now apply   
   Proposition~\ref{prop:Stab4EpsConditionOnF} 
   to conclude the assertion of the theorem.
\end{proof}

\section{Numerical study}

\label{sec:numerical_results}

%%%%% FIGURE WIDTH
\def\figwidth{0.32\textwidth}
%%%%%

% !TEX root = ../../../higherOrderBZ.tex
\begin{figure}[]
	\captionsetup[subfigure]{justification=centering}
	\centering
	\begin{subfigure}[t]{\figwidth}
		\includegraphics[width=1\textwidth]{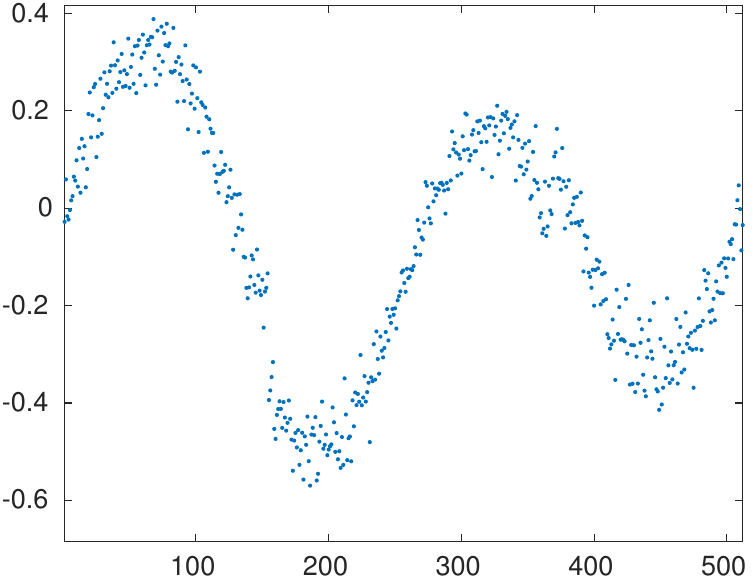}
		\subcaption{Noisy data}
	\end{subfigure}
	\begin{subfigure}[t]{\figwidth}
		\includegraphics[width=1\textwidth]{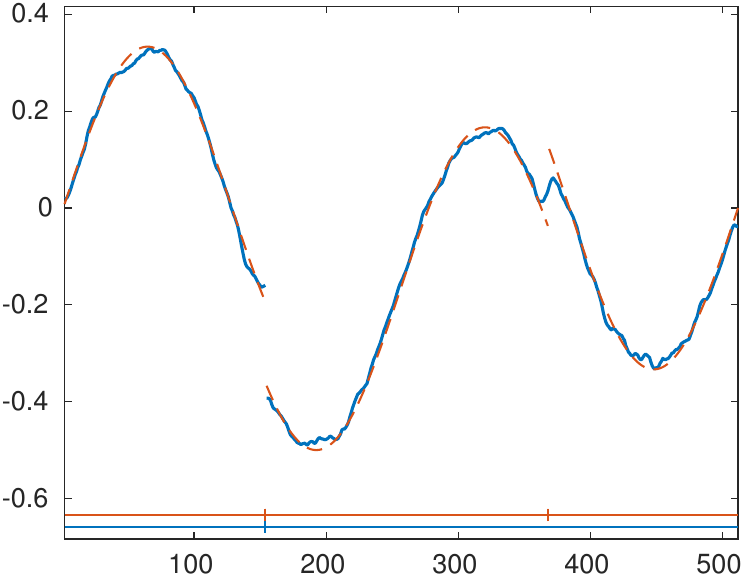}
		\subcaption{
		$(\Pc_{1,\beta,\gamma})$,
		$\beta= \protect\input{Experiments/Synthetic_Data/Donoho_Heavisine_A/beta_1.tex}$,
		$\gamma =\protect\input{Experiments/Synthetic_Data/Donoho_Heavisine_A/gamma_1.tex}$\\
		$\varepsilon_{\text{rel}} = 
		\protect\input{Experiments/Synthetic_Data/Donoho_Heavisine_A/rerr_1.tex}$, 
		$R_{\text{ind}} = 
		\protect\input{Experiments/Synthetic_Data/Donoho_Heavisine_A/randi_1.tex}$
		}
	\end{subfigure}~
	\begin{subfigure}[t]{\figwidth}
		\includegraphics[width=1\textwidth]{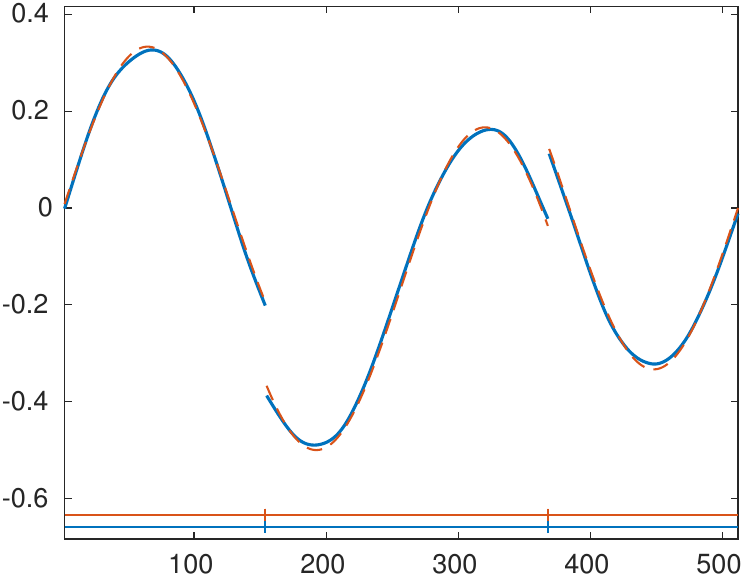}
		\subcaption{
			$(\Pc_{2,\beta,\gamma})$,
			$\beta= \protect\input{Experiments/Synthetic_Data/Donoho_Heavisine_A/beta_2.tex}$,
			$\gamma =\protect\input{Experiments/Synthetic_Data/Donoho_Heavisine_A/gamma_2.tex}$\\
			$\varepsilon_{\text{rel}} = 
			\protect\input{Experiments/Synthetic_Data/Donoho_Heavisine_A/rerr_2.tex}$,
			$R_{\text{ind}} = 
			\protect\input{Experiments/Synthetic_Data/Donoho_Heavisine_A/randi_2.tex}$			 
		}
	\end{subfigure}\\[1em]
	\begin{subfigure}[t]{\figwidth}
		\includegraphics[width=1\textwidth]{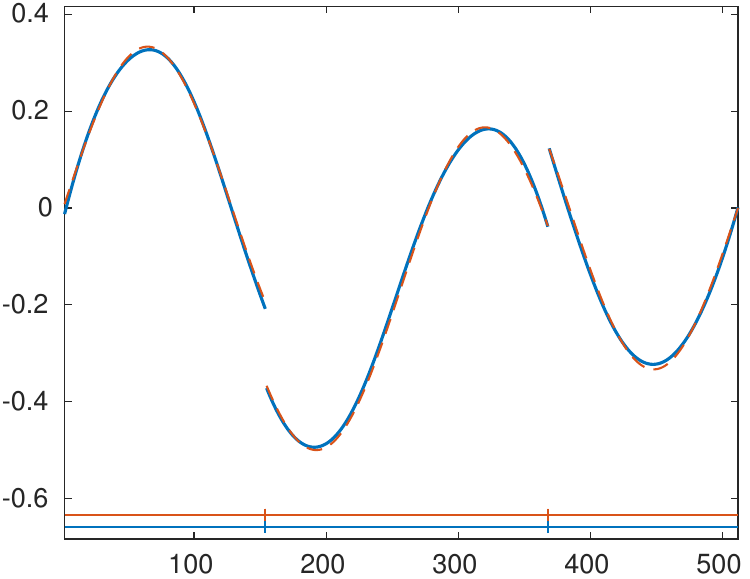}
		\subcaption{
			$(\Pc_{3,\beta,\gamma})$,
			$\beta= \protect\input{Experiments/Synthetic_Data/Donoho_Heavisine_A/beta_3.tex}$,
			$\gamma =\protect\input{Experiments/Synthetic_Data/Donoho_Heavisine_A/gamma_3.tex}$\\
			$\varepsilon_{\text{rel}} = 
			\protect\input{Experiments/Synthetic_Data/Donoho_Heavisine_A/rerr_3.tex}$, 
			$R_{\text{ind}} = 
			\protect\input{Experiments/Synthetic_Data/Donoho_Heavisine_A/randi_3.tex}$	
		}		
	\end{subfigure}~
		\begin{subfigure}[t]{\figwidth}
	\includegraphics[width=1\textwidth]{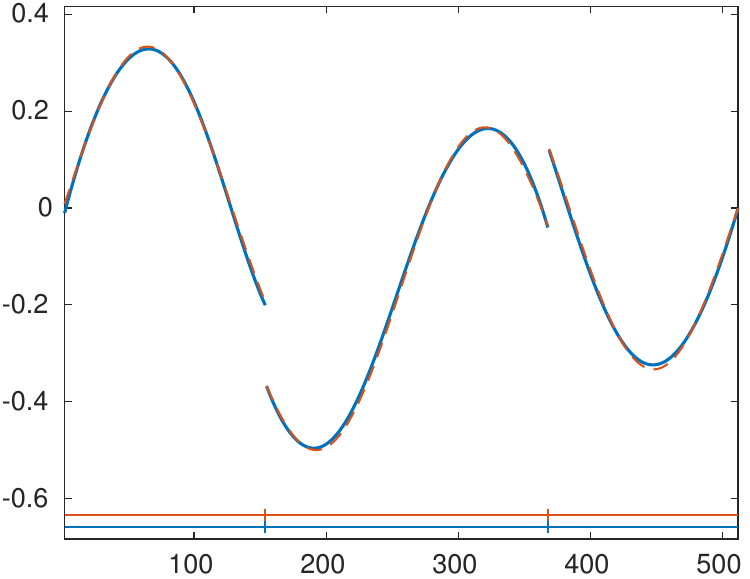}
	\subcaption{
			$(\Pc_{4,\beta,\gamma})$,
			$\beta= \protect\input{Experiments/Synthetic_Data/Donoho_Heavisine_A/beta_4.tex}$,
			$\gamma =\protect\input{Experiments/Synthetic_Data/Donoho_Heavisine_A/gamma_4.tex}$\\
			$\varepsilon_{\text{rel}} = 
			\protect\input{Experiments/Synthetic_Data/Donoho_Heavisine_A/rerr_4.tex}$, 
			$R_{\text{ind}} = 
			\protect\input{Experiments/Synthetic_Data/Donoho_Heavisine_A/randi_4.tex}$			
	}	
	\end{subfigure}~
	\begin{subfigure}[t]{\figwidth}
	\includegraphics[width=1\textwidth]{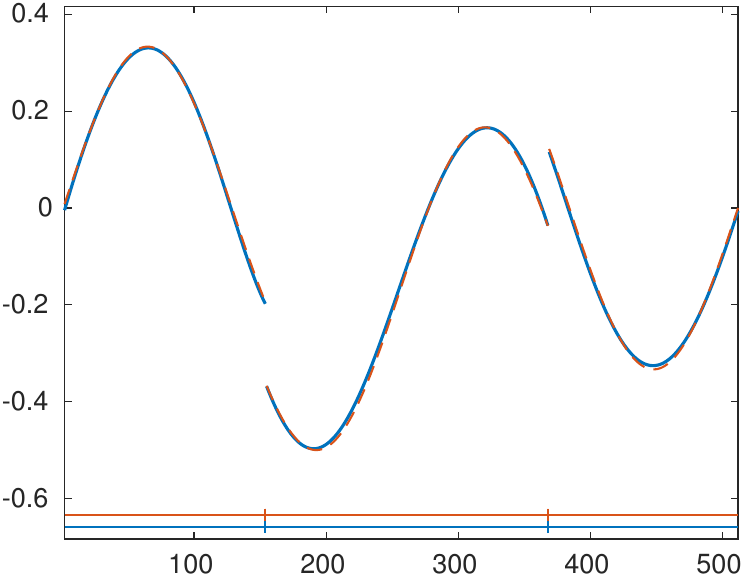}
	\subcaption{
			$(\Pc_{5,\beta,\gamma})$,
			$\beta= \protect\input{Experiments/Synthetic_Data/Donoho_Heavisine_A/beta_5.tex}$,
			$\gamma =\protect\input{Experiments/Synthetic_Data/Donoho_Heavisine_A/gamma_5.tex}$\\
			$\varepsilon_{\text{rel}} = 
			\protect\input{Experiments/Synthetic_Data/Donoho_Heavisine_A/rerr_5.tex}$, 
			$R_{\text{ind}} = 
			\protect\input{Experiments/Synthetic_Data/Donoho_Heavisine_A/randi_5.tex}$			
	}			
	\end{subfigure}
%	\begin{subfigure}[]{\figwidth}
%		\includegraphics[width=1\textwidth]{Experiments/Synthetic_Data/Donoho_Heavisine_A/Heavysine_6.pdf}
%		\subcaption{
%			$(\Pc_{6,\beta,\gamma})$,
%			$\beta= \protect\input{Experiments/Synthetic_Data/Donoho_Heavisine_A/beta_6.tex}$,
%			$\gamma =\protect\input{Experiments/Synthetic_Data/Donoho_Heavisine_A/gamma_6.tex}$\\
%			$\varepsilon_{\text{rel}} = 
%			\protect\input{Experiments/Synthetic_Data/Donoho_Heavisine_A/rerr_6.tex}$, 
%			$R_{\text{ind}} = 
%			\protect\input{Experiments/Synthetic_Data/Donoho_Heavisine_A/randi_6.tex}$			
%		}			
%	\end{subfigure}	
	\caption{
		\label{Fig Heavy Sine}
		Reconstructions of ``Heavy Sine''-signal from noisy data.
		\emph{(a)} Data corrupted by Gaussian noise of level~$\protect\input{Experiments/Synthetic_Data/Donoho_Heavisine_A/noise_lvl.tex}$. \emph{(b--f)} Reconstructions for higher order Mumford-Shah and Potts models.
	}
\end{figure}

% !TEX root = ../../../higherOrderBZ.tex
\begin{figure}[]
	\captionsetup[subfigure]{justification=centering}
	\centering
	\begin{subfigure}[t]{\figwidth}
		\includegraphics[width=1\textwidth]{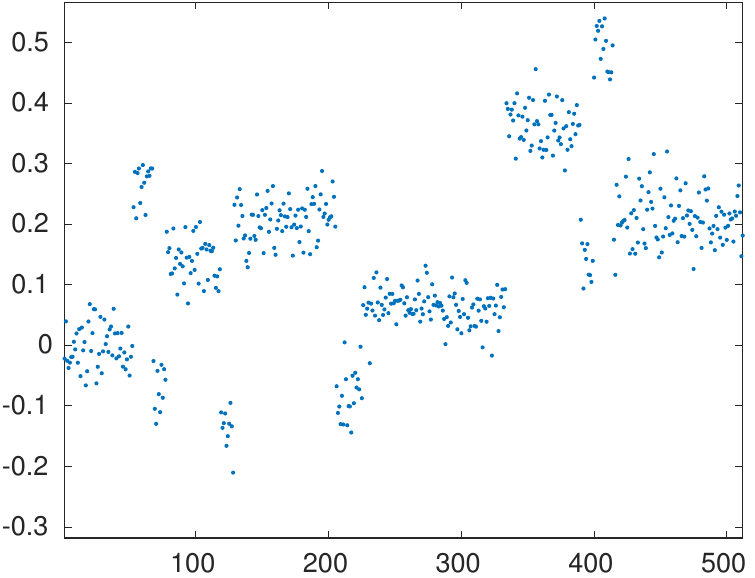}
		\subcaption{Noisy data}
	\end{subfigure}~
	\begin{subfigure}[t]{\figwidth}
		\includegraphics[width=1\textwidth]{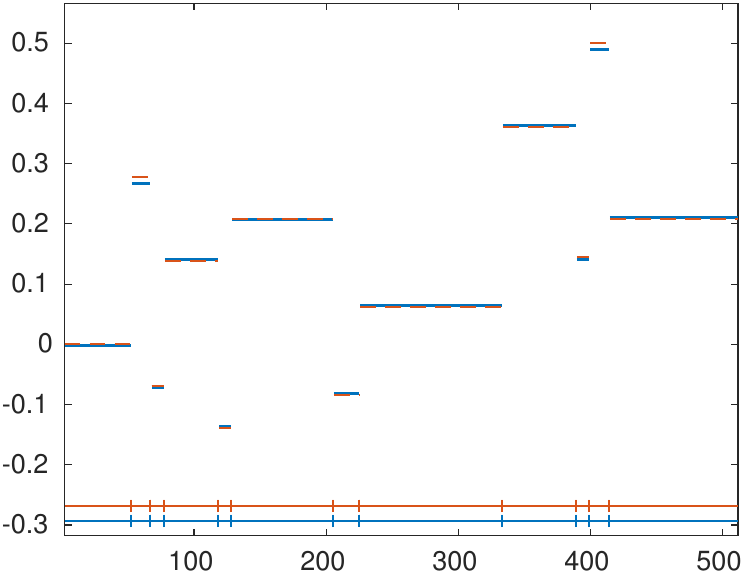}
		\subcaption{
		$(\Pc_{1,\beta,\gamma})$,
		$\beta= \protect\input{Experiments/Synthetic_Data/Donoho_Blocks_A/beta_1.tex}$,
		$\gamma =\protect\input{Experiments/Synthetic_Data/Donoho_Blocks_A/gamma_1.tex}$\\
		$\varepsilon_{\text{rel}} = 
		\protect\input{Experiments/Synthetic_Data/Donoho_Blocks_A/rerr_1.tex}$,
		$R_{\text{ind}} = 
		\protect\input{Experiments/Synthetic_Data/Donoho_Blocks_A/randi_1.tex}$		
		}
	\end{subfigure}~
	\begin{subfigure}[t]{\figwidth}
		\includegraphics[width=1\textwidth]{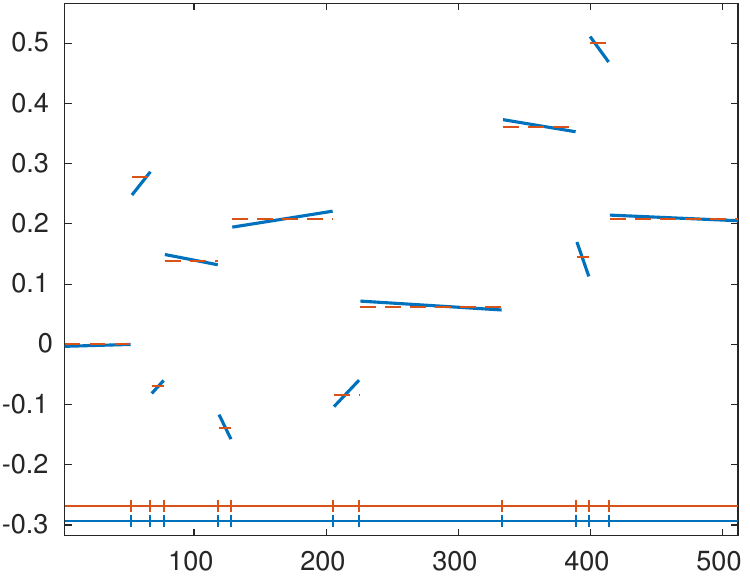}
		\subcaption{
			$(\Pc_{2,\beta,\gamma})$,
			$\beta= \protect\input{Experiments/Synthetic_Data/Donoho_Blocks_A/beta_2.tex}$,
			$\gamma =\protect\input{Experiments/Synthetic_Data/Donoho_Blocks_A/gamma_2.tex}$\\
			$\varepsilon_{\text{rel}} = 
			\protect\input{Experiments/Synthetic_Data/Donoho_Blocks_A/rerr_2.tex}$, 
			$R_{\text{ind}} = 
			\protect\input{Experiments/Synthetic_Data/Donoho_Blocks_A/randi_2.tex}$		
		}		
	\end{subfigure}\\[1em]
	\begin{subfigure}[t]{\figwidth}
		\includegraphics[width=1\textwidth]{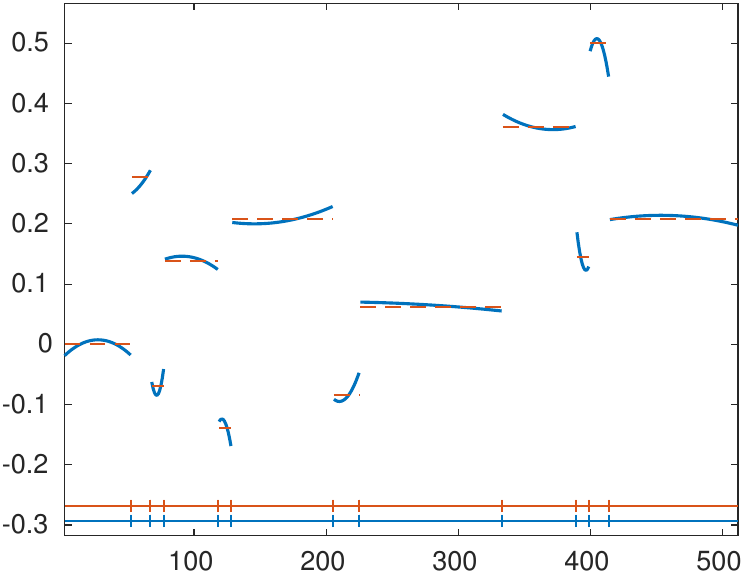}
		\subcaption{
			$(\Pc_{3,\beta,\gamma})$,
			$\beta= \protect\input{Experiments/Synthetic_Data/Donoho_Blocks_A/beta_3.tex}$,
			$\gamma =\protect\input{Experiments/Synthetic_Data/Donoho_Blocks_A/gamma_3.tex}$\\
			$\varepsilon_{\text{rel}} = 
			\protect\input{Experiments/Synthetic_Data/Donoho_Blocks_A/rerr_3.tex}$, 
			$R_{\text{ind}} = 
			\protect\input{Experiments/Synthetic_Data/Donoho_Blocks_A/randi_3.tex}$		
		}		
	\end{subfigure}~
	\begin{subfigure}[t]{\figwidth}
	\includegraphics[width=1\textwidth]{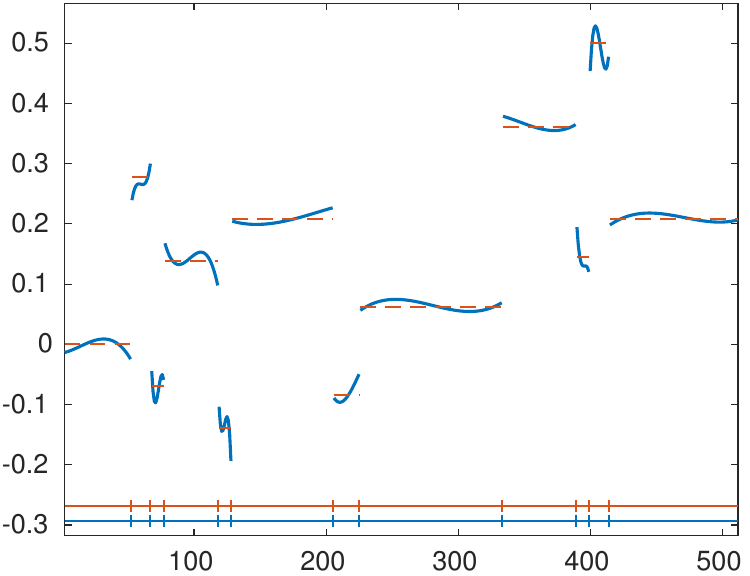}
	\subcaption{
			$(\Pc_{4,\beta,\gamma})$,
			$\beta= \protect\input{Experiments/Synthetic_Data/Donoho_Blocks_A/beta_4.tex}$,
			$\gamma =\protect\input{Experiments/Synthetic_Data/Donoho_Blocks_A/gamma_4.tex}$\\
			$\varepsilon_{\text{rel}} = 
			\protect\input{Experiments/Synthetic_Data/Donoho_Blocks_A/rerr_4.tex}$, 
			$R_{\text{ind}} = 
			\protect\input{Experiments/Synthetic_Data/Donoho_Blocks_A/randi_4.tex}$		
	}	
	\end{subfigure}~
	\begin{subfigure}[t]{\figwidth}
	\includegraphics[width=1\textwidth]{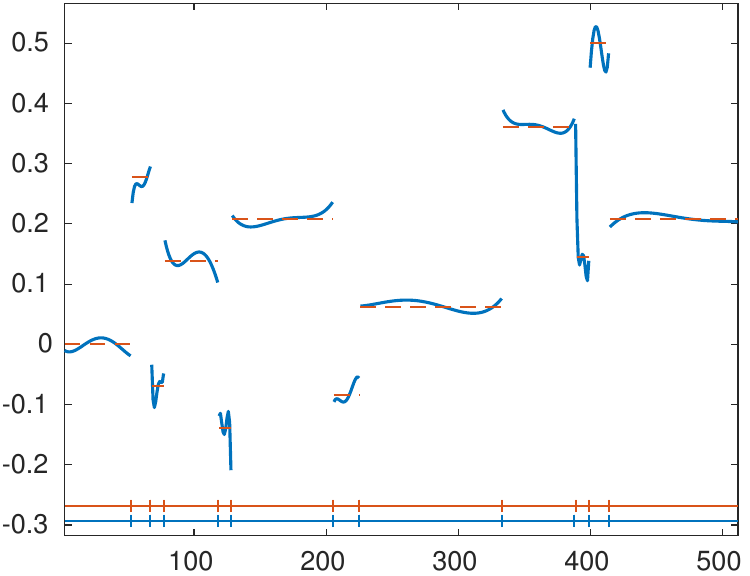}
	\subcaption{
			$(\Pc_{5,\beta,\gamma})$,
			$\beta= \protect\input{Experiments/Synthetic_Data/Donoho_Blocks_A/beta_5.tex}$,
			$\gamma =\protect\input{Experiments/Synthetic_Data/Donoho_Blocks_A/gamma_5.tex}$\\
			$\varepsilon_{\text{rel}} = 
			\protect\input{Experiments/Synthetic_Data/Donoho_Blocks_A/rerr_5.tex}$, 
			$R_{\text{ind}} = 
			\protect\input{Experiments/Synthetic_Data/Donoho_Blocks_A/randi_5.tex}$		
	}				
	\end{subfigure}
%~
%	\begin{subfigure}{\figwidth}
%		\includegraphics[width=1\textwidth]{Experiments/Synthetic_Data/Donoho_Blocks_A/Blocks_6.pdf}
%		\subcaption{
%			$(\Pc_{6,\beta,\gamma})$,
%			$\beta= \protect\input{Experiments/Synthetic_Data/Donoho_Blocks_A/beta_6.tex}$,
%			$\gamma =\protect\input{Experiments/Synthetic_Data/Donoho_Blocks_A/gamma_6.tex}$\\
%			$\varepsilon_{\text{rel}} = 
%			\protect\input{Experiments/Synthetic_Data/Donoho_Blocks_A/rerr_6.tex}$, 
%			$R_{\text{ind}} = 
%			\protect\input{Experiments/Synthetic_Data/Donoho_Blocks_A/randi_6.tex}$	
%		}			
%	\end{subfigure}	
	\caption{
		\label{Fig Blocks}
		Reconstructions of ``Blocks''-signal from noisy data. %The relative reconstruction error is denoted by $\varepsilon_{rel}$.
		\emph{(a)} Data corrupted by Gaussian noise of level~$\protect\input{Experiments/Synthetic_Data/Donoho_Blocks_A/noise_lvl.tex}$. \emph{(b--f)} Reconstructions using higher order Mumford-Shah and Potts models.
	}
\end{figure}

We conduct a numerical study
on the reconstruction quality of the considered higher order Mumford-Shah and Potts models,
and on  the computation time of the proposed solvers.
We implemented the proposed solvers for the higher order Mumford-Shah and Potts models  in 
C++ with wrappers to Matlab using mex-files. The pseudocode is given as Algorithm~\ref{alg:solverPenalized} in the appendix.
All experiments were conducted on a desktop computer with 3.1 GHz Intel Core i5-2400 processor 
and 8 GB RAM.

\subsection{Reconstruction results}

% !TEX root = ../../../higherOrderBZ.tex
\begin{figure}
	\centering
	\footnotesize
\begin{subfigure}{0.32\textwidth}
\subcaption*{$\eta=\protect\input{Experiments/Parameter_Choice/Elasticity_B/noise_lvl_1.tex}$}
\makebox[0pt][r]{\makebox[3em]{\raisebox{40pt}{\rotatebox[origin=c]{90}{Noisy data}}}}%
\includegraphics[width=\textwidth]{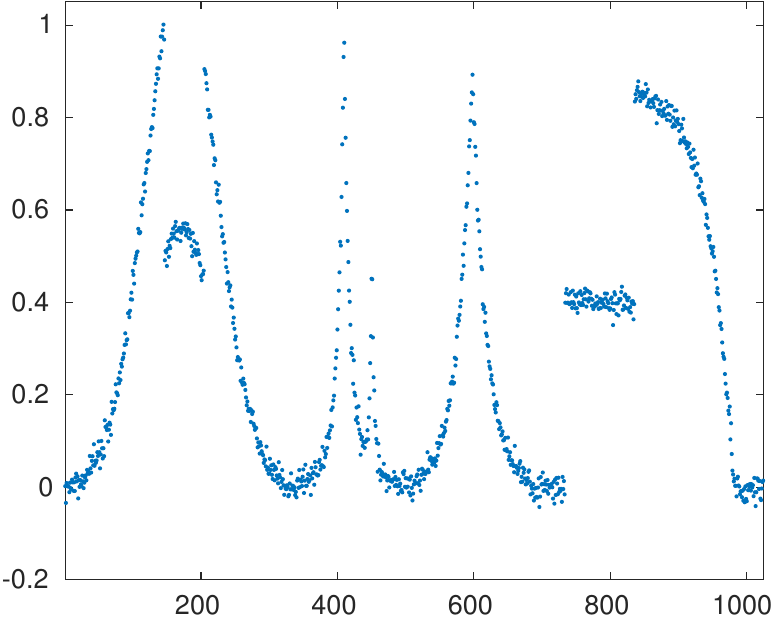}\\[0.5em]
\makebox[0pt][r]{\makebox[3em]{\raisebox{40pt}{\rotatebox[origin=c]{90}{$k=1$}}}}%
\includegraphics[width=\textwidth]{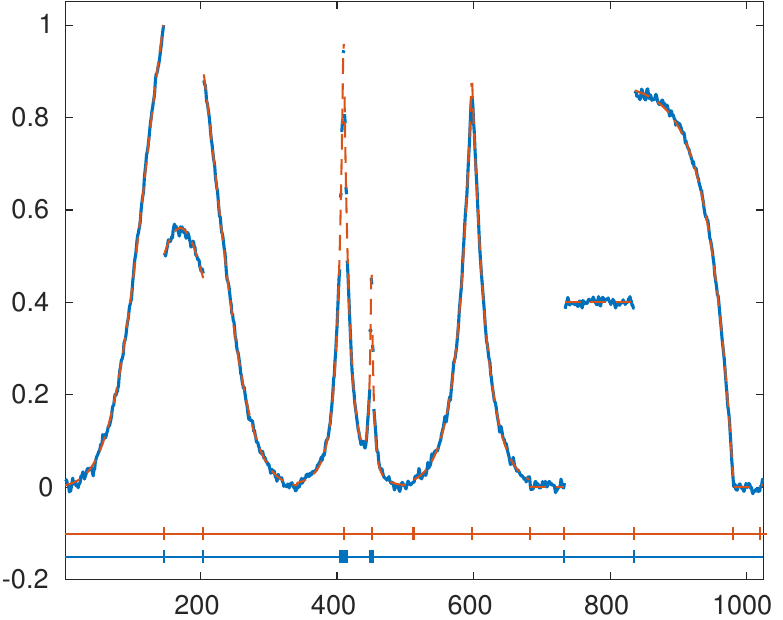}\\[0.5em]
%\subcaption{\label{}}
\makebox[0pt][r]{\makebox[3em]{\raisebox{40pt}{\rotatebox[origin=c]{90}{$k=2$}}}}%
\includegraphics[width=\textwidth]{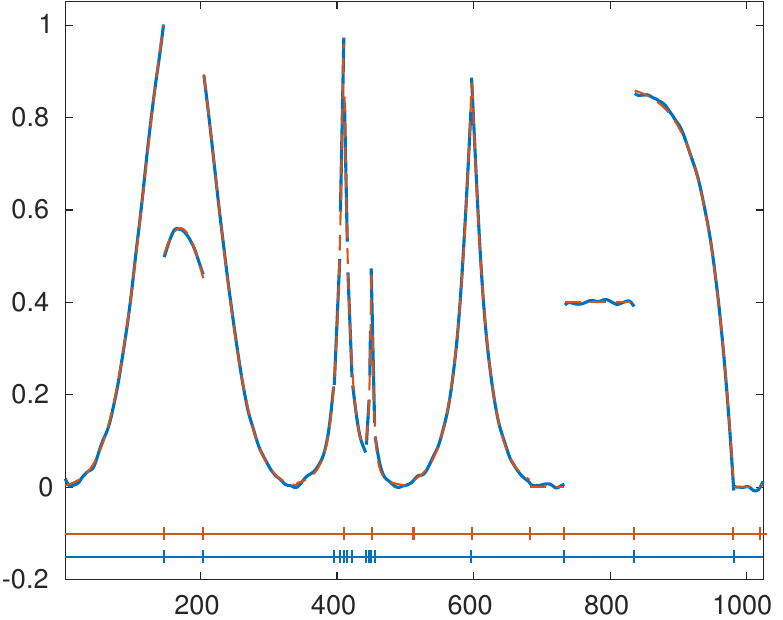}\\[0.5em]
%\subcaption{\label{}}
\makebox[0pt][r]{\makebox[3em]{\raisebox{40pt}{\rotatebox[origin=c]{90}{$k=3$}}}}%
\includegraphics[width=\textwidth]{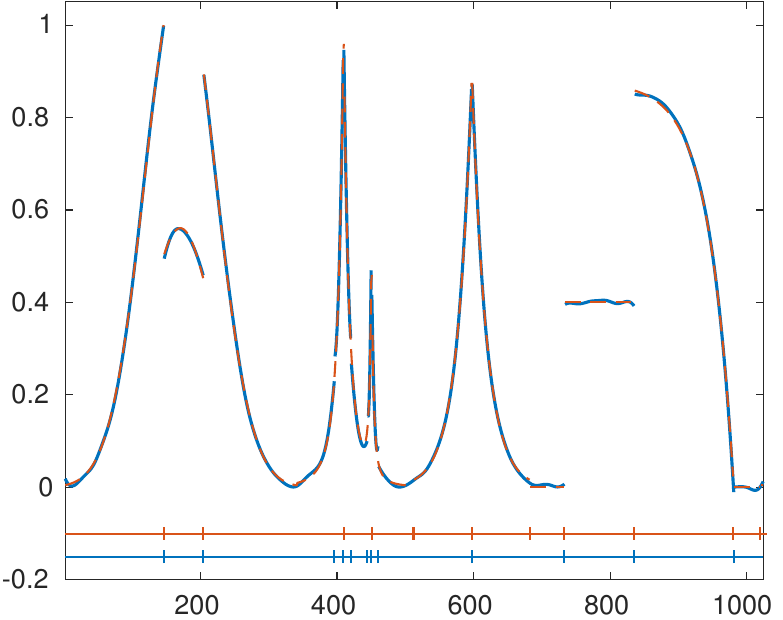}\\[0.5em]
	%\subcaption{\label{}}
\makebox[0pt][r]{\makebox[3em]{\raisebox{40pt}{\rotatebox[origin=c]{90}{$k=4$}}}}%
\includegraphics[width=\textwidth]{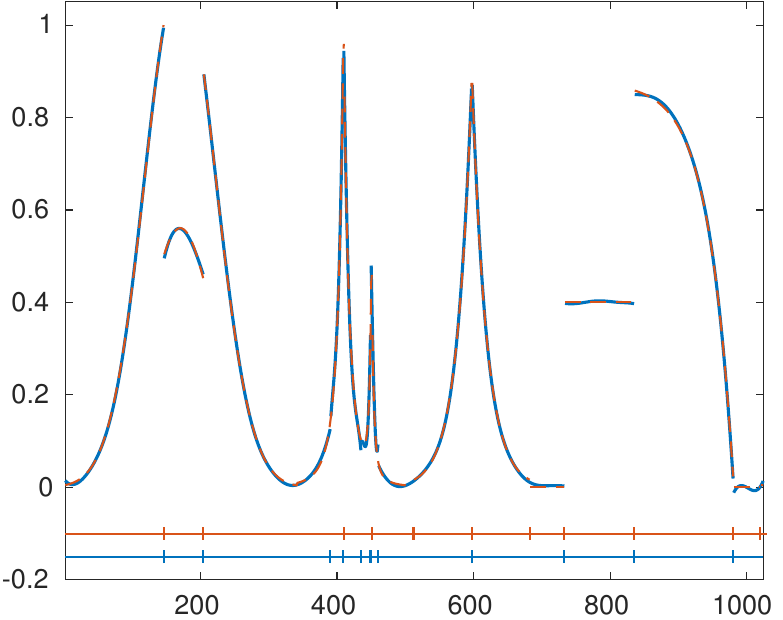}
%\subcaption{\label{}}	
\end{subfigure}
\begin{subfigure}{0.32\textwidth}
\subcaption*{$\eta=\protect\input{Experiments/Parameter_Choice/Elasticity_B/noise_lvl_2.tex}$}
\includegraphics[width=\textwidth]{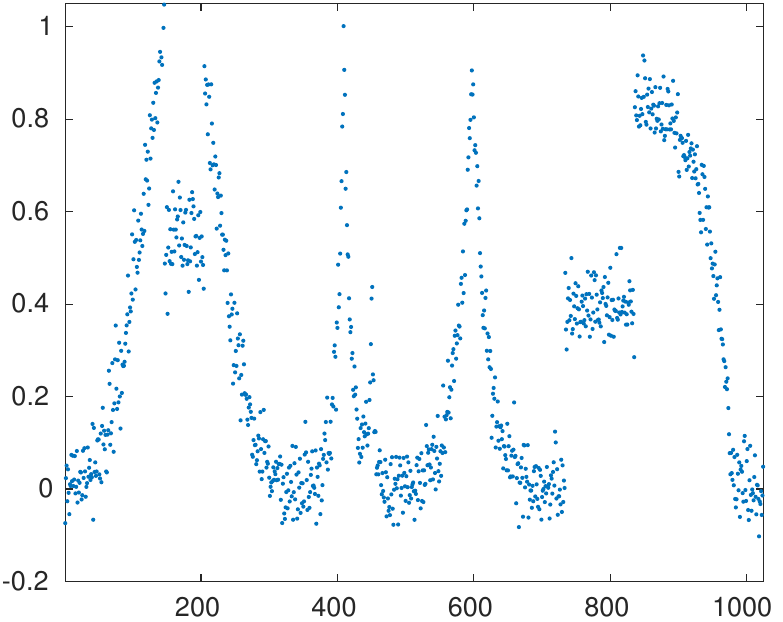}\\[0.5em]
\includegraphics[width=\textwidth]{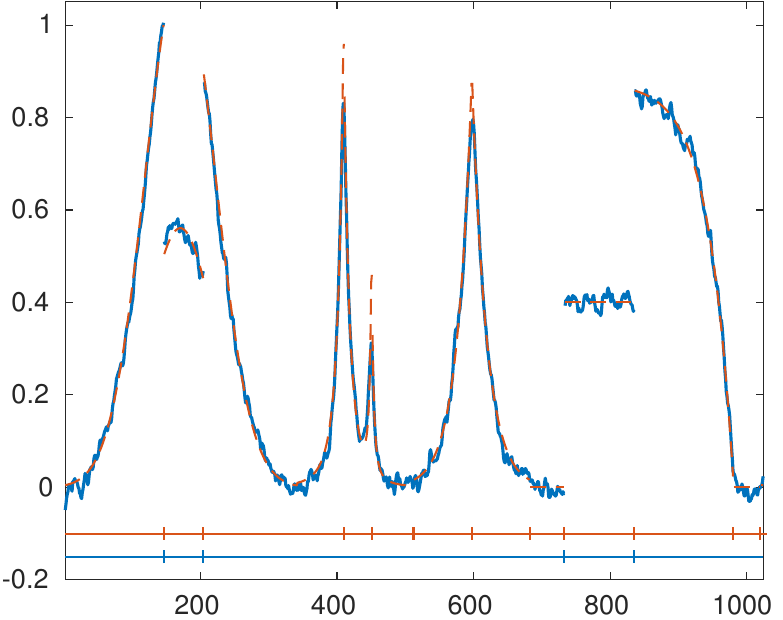}\\[0.5em]
\includegraphics[width=\textwidth]{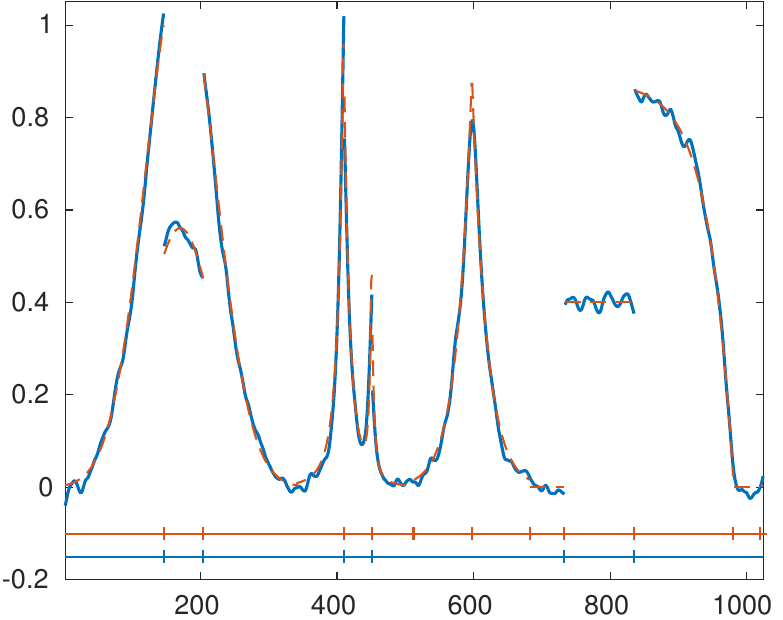}\\[0.5em]
\includegraphics[width=\textwidth]{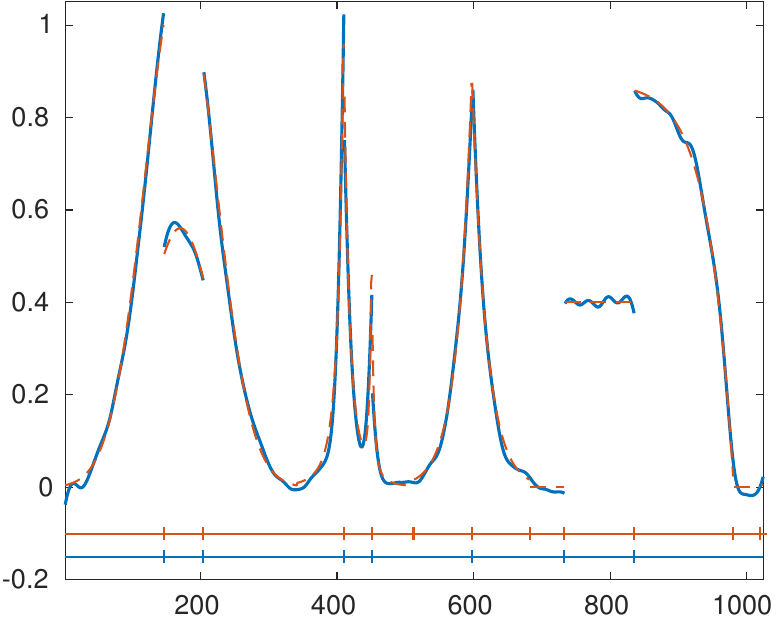}\\[0.5em]
	\includegraphics[width=\textwidth]{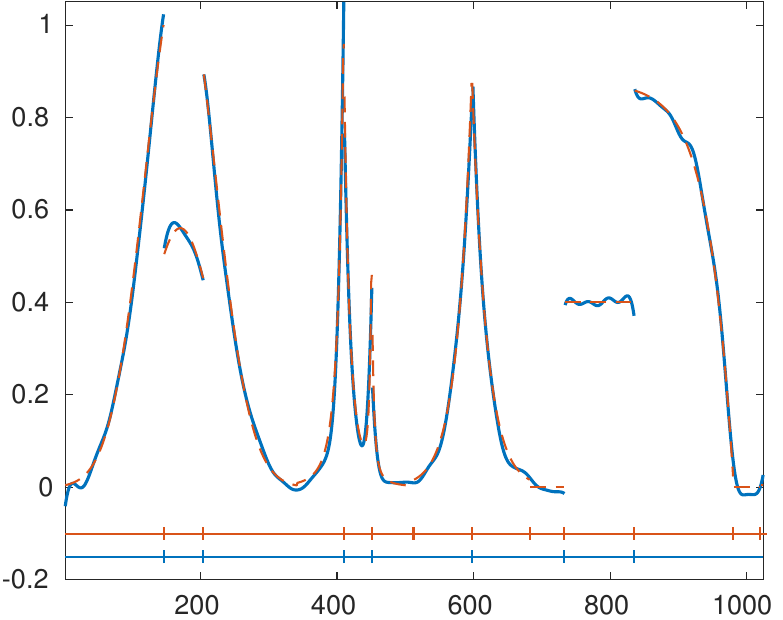}
\end{subfigure}
\begin{subfigure}{0.32\textwidth}
\subcaption*{$\eta=\protect\input{Experiments/Parameter_Choice/Elasticity_B/noise_lvl_3.tex}$}
\includegraphics[width=\textwidth]{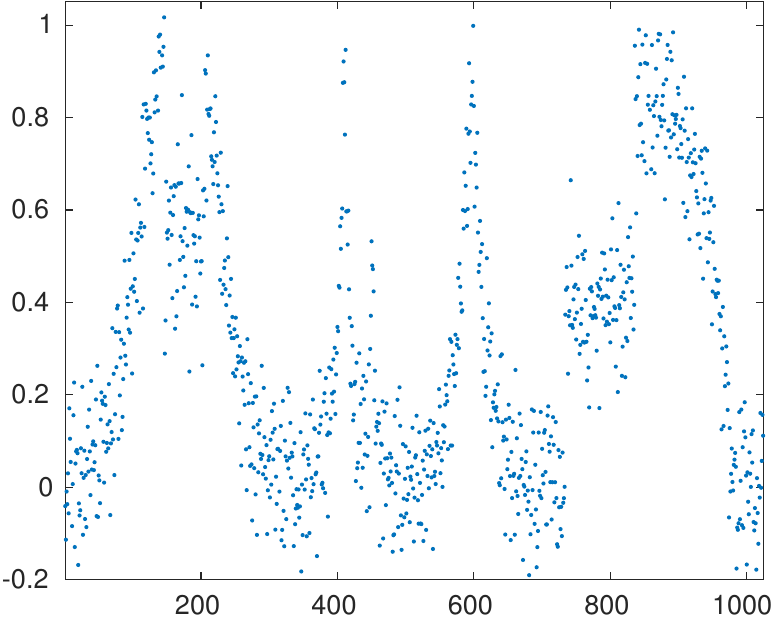}\\[0.5em]
\includegraphics[width=\textwidth]{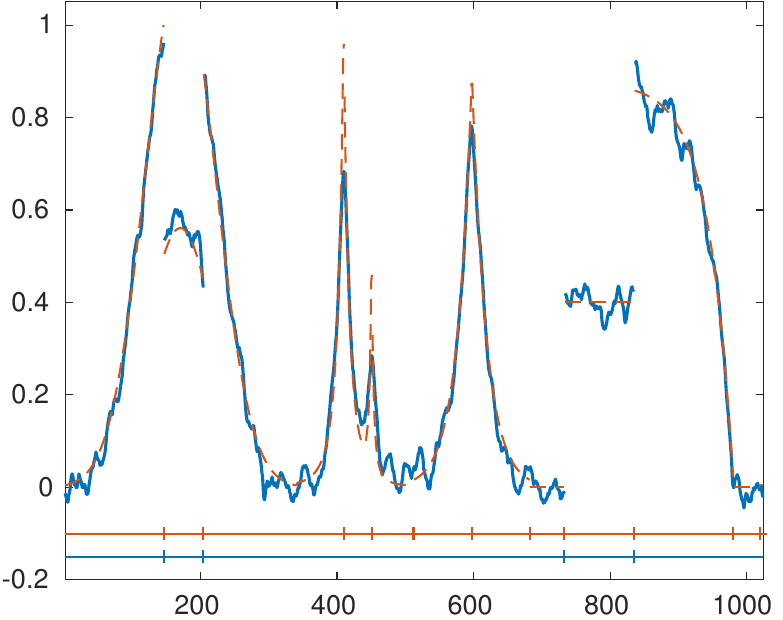}\\[0.5em]
\includegraphics[width=\textwidth]{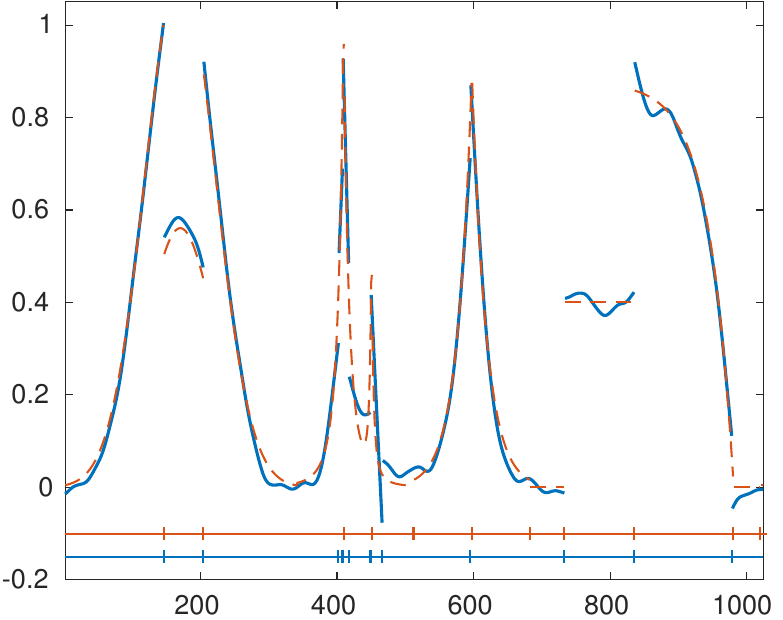}\\[0.5em]
\includegraphics[width=\textwidth]{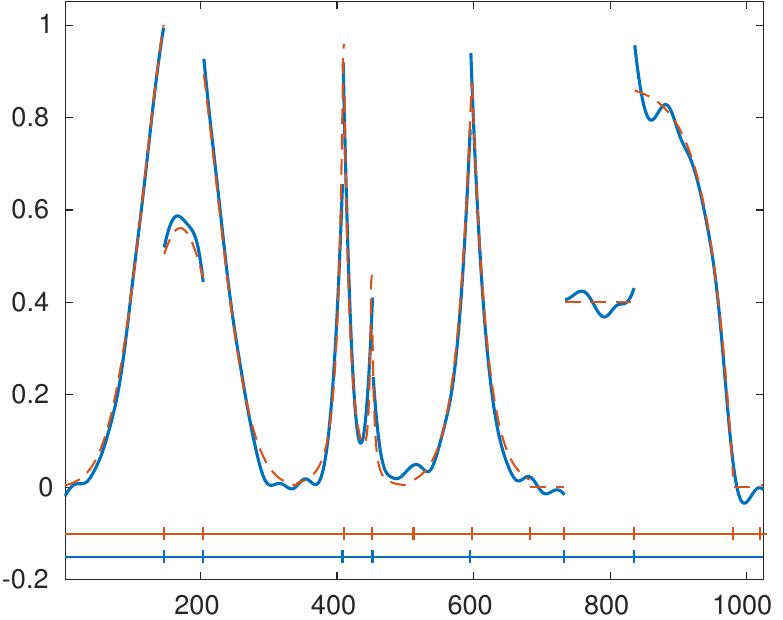}\\[0.5em]
\includegraphics[width=\textwidth]{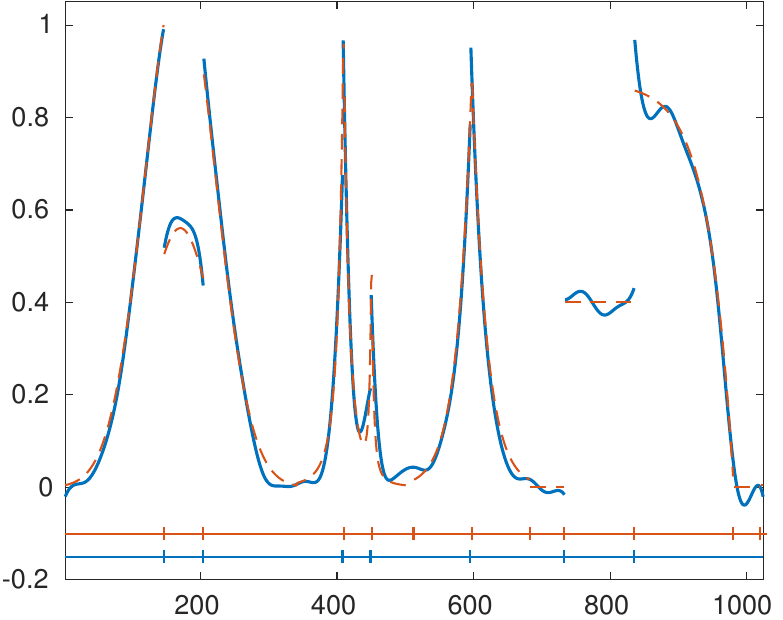}
\end{subfigure}
\caption{
\label{Fig Mallat Reconstructions}
Smoothing of piecewise defined signal 
%shown in 	\cref{Fig Mallat nlvl_1,Fig Mallat nlvl_2,Fig Mallat nlvl_3} 
of increasing noise level $\eta$ \emph{(top row)}.
We observe that the segmentation quality gets higher 
and that the noise is smoothed out better on the segments
when using higher order Mumford-Shah models.
%We observe that (\textbf{todo}) the higher order models yield better reconstruction results than the first order models. 
%Parameters were chosen empirically by means of the relative $\ell_2$-error.
%\textbf{Beobachtung: $k=5$ ist in allen noise levels das Optimum und höhere Ordnungen werden schlechter. Mit Aufnehmen? Wie visualisieren (Fehlergraph in separater Figure?)?}
%where best res the improvement saturates at around order $k=3.$
}
\end{figure}

We first investigate 
the potential of the higher order Mumford-Shah and Potts models
with respect to reconstruction quality.
We employ commonly used test signals with discontinuities; see \cite{donoho1994ideal, mallat2008wavelet}.
We corrupt the signals by additive zero mean Gaussian noise with variance $\sigma^2.$
We let the noise level $\eta$ be given by $\eta = \sigma N/ \|g\|_1,$ 
where $g$ denotes the clean signal.
To obtain a meaningful comparison
of the models' potentials,
we determined parameters $\beta$ and $\gamma$ 
such that the result $u^*$ has the best 
relative $\ell_2$-error $\varepsilon_\text{rel},$ given by
$
\varepsilon_\text{rel}=\|u^*-g \|_2 / \|g \|_2.
$
(We use a full grid search over 
$\gamma = (0,1]$
with stepsize $0.001,$ 
and $\beta  \in
(0,25]$ 
with stepsize $0.025$
and $\beta = \infty.$)
We are further interested in the quality of the computed partition $\Ic^*.$
A commonly used measure for  segmentation quality is the Rand index \cite{Rand1971}
which we briefly explain.
The Rand index $\RI$ of two partitions $\Ic, \Ic'$ is given by 
$
	\RI(\Ic, \Ic') = \binom{N}{2}\sum_{\{i,j:\, 1 \leq i < j \leq N\}} t_{ij}
$
where $t_{ij}$ is equal to one 
if there are  $I \in \Ic$ and $I' \in \Ic'$ such that $i$ and $j$ are in both $I$ and $I',$
or if $i$ is in both $I$ and $I'$ while $j$ is in neither $I$ and $I'.$ Otherwise, $t_{ij} = 0.$ Further, $N$ denotes the length of the signal.
The Rand index is bounded from above by one and a higher value means a better match.
A value of one means that $\Ic$ and $\Ic'$ agree.
Here, we report the Rand index $\RI$ of the computed segmentation and the ground truth segmentation.%
\footnote{For the numerical evaluation of the Rand index, we used the implementation of K.~Wang and D.~Corney available at the Matlab File Exchange.}
It is worth recalling that a high parameter $\beta$ 
leads to stronger smoothing on the segments, and that a high parameter $\gamma$ leads to less segments.

The first signal is a sinusoidal with two steps (Figure~\ref{Fig Heavy Sine}).
We observe that the first order model requires choosing a relatively
small $\beta$ parameter to avoid the 
gradient limit effect.
As tradeoff, the resulting signal remains visibly affected by the noise
and the second discontinuity is smoothed out.
Increasing the order $k$ to values greater than one leads to better results
with respect to the segmentation quality.
Furthermore, the relative error improves when increasing the order.
The reconstruction quality 
starts decaying from order $k=6$ on which can be attributed to overfitting.

The second example is a piecewise constant signal (Figure~\ref{Fig Blocks}).
As the signal has no variation on the segments,
the experiment confirms the intuition that large elasticity parameters 
-- mostly $\beta = \infty$ -- are preferable here. 
The best result is obtained by the first order Potts model
as its search space is restricted to piecewise constant functions
which perfectly matches the signal.
Yet, using higher order models leads to very good segmentation results up to order $k = 5$
and good reconstructions up to order $k = 3,$ as well.

\cref{Fig Mallat Reconstructions}
shows the reconstruction results for a piecewise smooth signal
for different noise levels.
We observe that the results of the first order model 
remain relatively noisy on the segments.
A reason for this is that the elasticity parameter needs to be relatively small
to prevent the gradient limit effect.
Using the second order model improves the reconstruction results significantly
but also tends to produce spurious segments, 
in particular at the parts of high curvature.
Increasing the order to $k \geq 3,$ 
leads to better segmentations and improved smoothing on the segments.

\interfootnotelinepenalty=10000
An example of the effect of higher order Mumford-Shah on real data times series is given in \cref{Fig:Windspeed}. The data are time-averaged (hourly) wind speeds at the summit of highest German mountain Zugspitze
from November to December 2016.\footnote{
The data were collected by German climate data center and are available via ftp at
	\url{
		ftp://ftp-cdc.dwd.de/pub/CDC/observations_germany/climate/hourly/wind/historical/
		}
(station id: 02115).
}
 We observe that strong changes of the windspeed result in breakpoints of the higher order Mumford-Shah estimate.
Some breakpoints can be associated with a meaning:
the break at 492 and the two breaks near 1154 and 1182
 can be linked with the days of strongest squalls in November and December 2016, respectively.%
\footnote{\emph{Monatsrückblick der Wetterwarte Garmisch-Partenkirchen/Zugspitze} at
\url{http://www.schneefernerhaus.de/fileadmin/web_data/bilder/pdf/MontasrueckblickeZG/MORZG1116.pdf} and at
\url{http://www.schneefernerhaus.de/fileadmin/web_data/bilder/pdf/MontasrueckblickeZG/MORZG1216.pdf}
}

% !TEX root = ../../higherOrderBZ.tex
\begin{figure}[!t]
	\begin{subfigure}{\textwidth}
		\includegraphics[width=\textwidth]{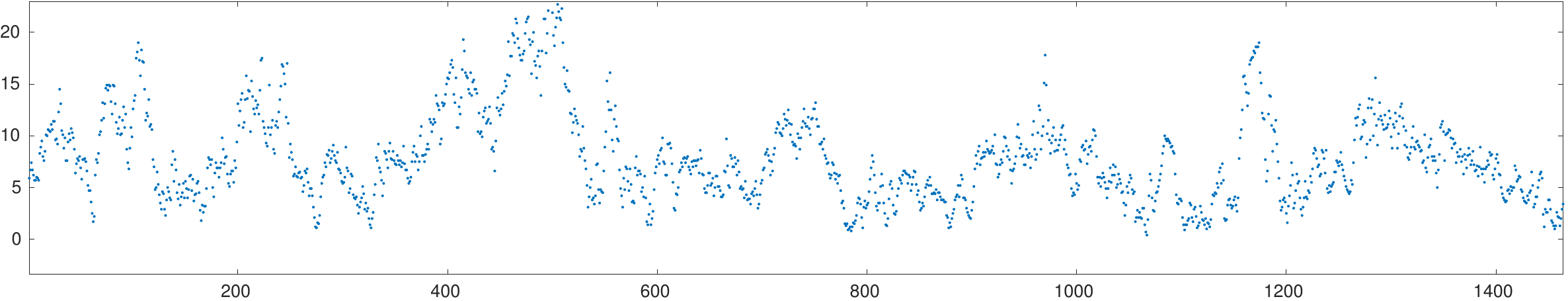}
		%\subcaption{Original data.}
	\end{subfigure}~\\[1em]
	\begin{subfigure}{\textwidth}
		\includegraphics[width=\textwidth]{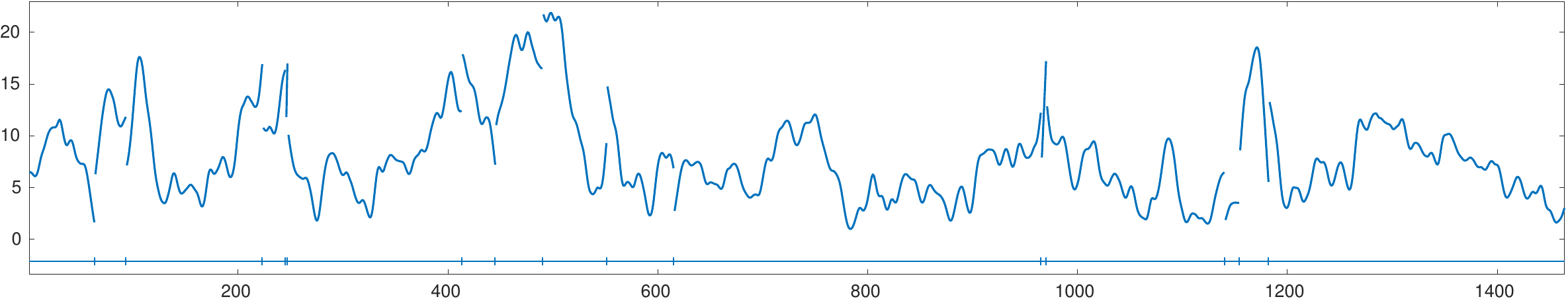}
		%\subcaption{Result of higher order Mumford-Shah model $(\Pc_{\protect\input{Experiments/Real_Data/k.tex};\protect\input{Experiments/Real_Data/beta.tex};\protect\input{Experiments/Real_Data/gamma.tex}}).$}
%$k=$ ,~
%			 $\beta=\protect\input{Experiments/Real_Data/beta.tex}$ ,~
%			$\gamma=\protect\input{Experiments/Real_Data/gamma.tex}$}
	\end{subfigure}
	%\\[1em]
%	\begin{subfigure}{\textwidth}
%		\includegraphics[width=\textwidth]{Experiments/Real_Data/spline_Zugspitze.pdf}
%		%\subcaption{Second order spline, temp}
%	\end{subfigure}
\caption{\label{Fig:Windspeed}
\emph{Top:} Hourly averaged windspeeds [m/s] at the summit of the Zugspitze from November to December 2016.
\emph{Bottom:} Result of higher order Mumford-Shah model $(\Pc_{\protect\input{Experiments/Real_Data/k.tex};\protect\input{Experiments/Real_Data/beta.tex};\protect\input{Experiments/Real_Data/gamma.tex}}).$
%The piecewise smooth spline emphasizes changes within the data by imposing a changepoint, 
}
\end{figure}

\subsection{Computation time}
\label{Section Computing Time}

% !TEX root = ../../../higherOrderBZ.tex
\begin{figure}[]
	\centering
\begin{subfigure}[b]{0.49\textwidth}
\centering
	\small
	\begin{tabular}{r  r  r  r  r  }
\toprule
		\input{Experiments/Run_Time/Run_Time_B/Run_Times_BZ_k_Table.tex}
		\bottomrule
	\end{tabular}
\end{subfigure}
\begin{subfigure}[b]{0.49\textwidth}
\centering

		\small
		\begin{tabular}{ r  r  r  r  r  }
		\toprule
		\input{Experiments/Run_Time/Run_Time_B/Run_Times_Potts_k_Table.tex}
		\bottomrule
		\end{tabular}
\end{subfigure}
\\[1em]
\centering
\begin{subfigure}[b]{0.45\textwidth}
\centering
	\includegraphics[width=1\textwidth,height=0.725\textwidth]{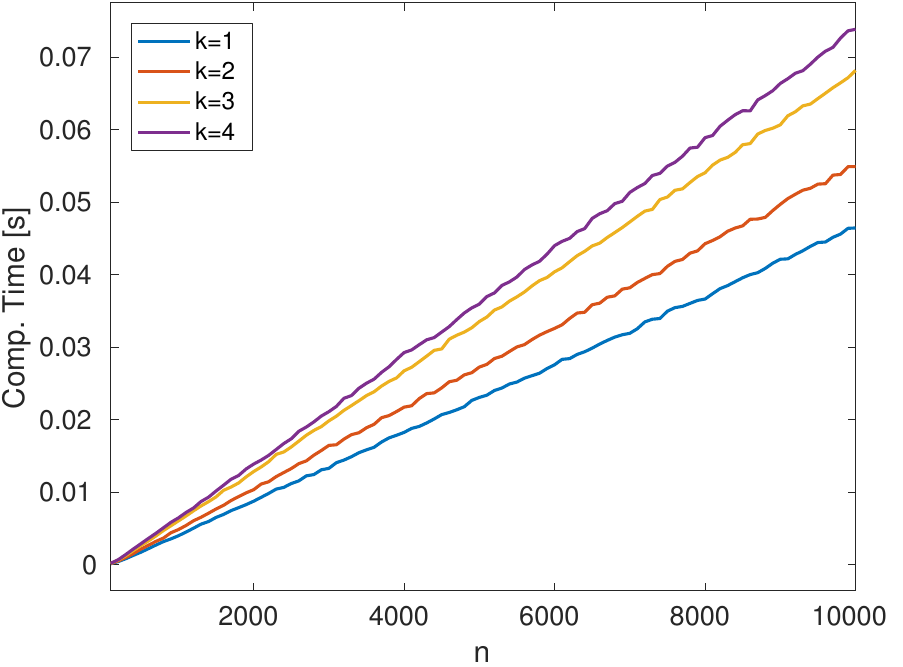}
	\includegraphics[width=1\textwidth,height=0.78\textwidth]{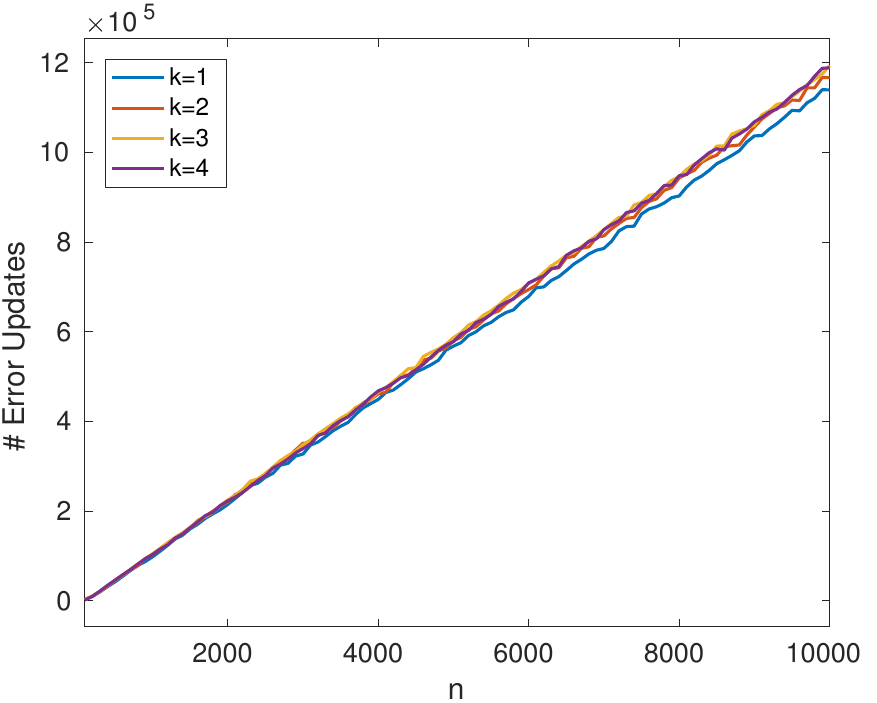}
		\subcaption{Higher order Mumford-Shah model.}	
\end{subfigure} \hspace{0.05\textwidth}
\begin{subfigure}[b]{0.45\textwidth}
\centering 
	\includegraphics[width=\textwidth]{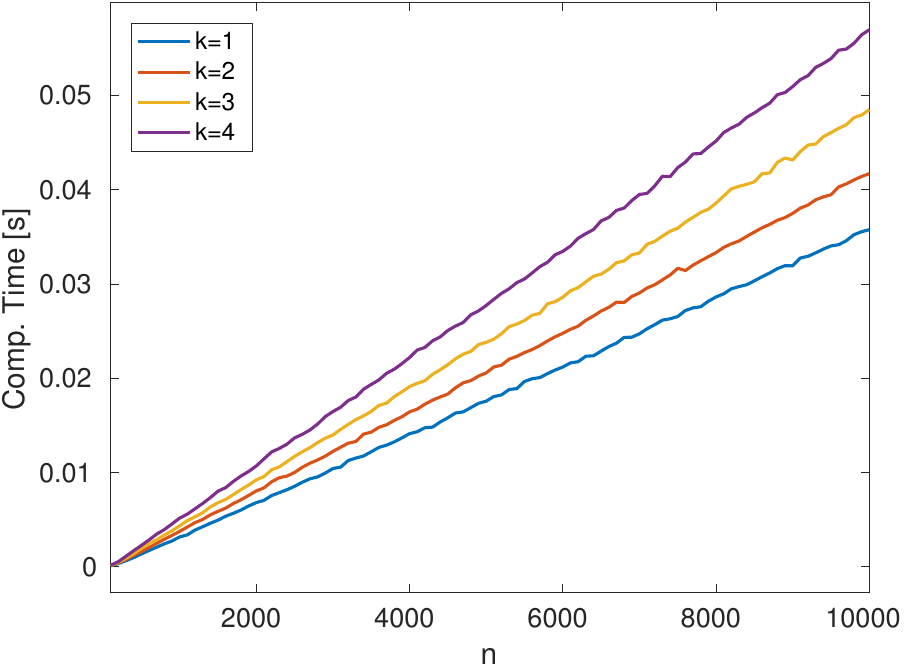}
	\includegraphics[width=1\textwidth]{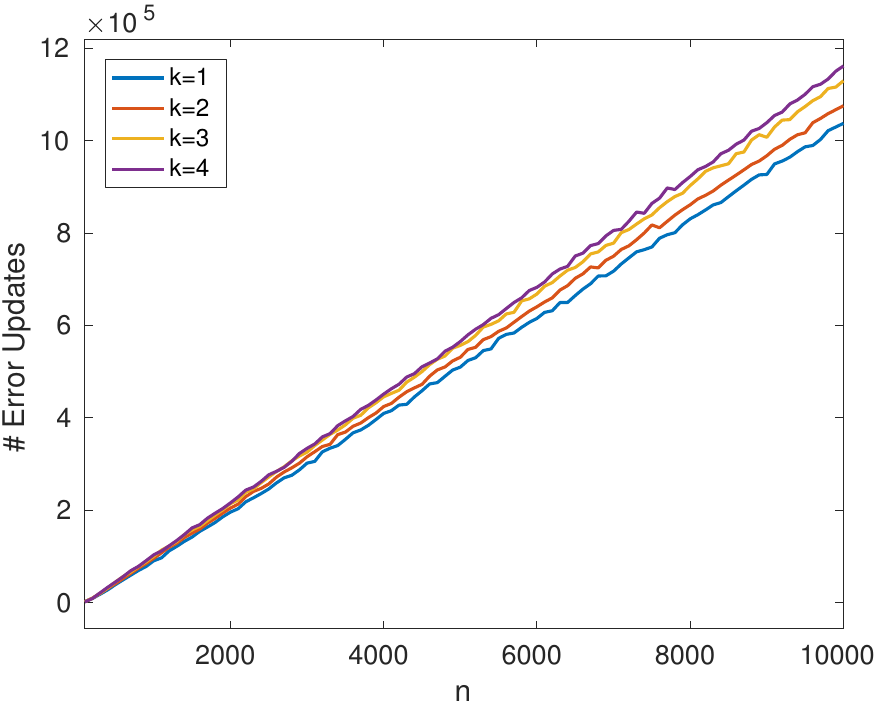}
		\subcaption{Higher order Potts model.}
\end{subfigure}~
\caption{
	\label{Figure_Experiments_RunTime_PwPoly}
	Computational costs of 
	the proposed algorithm
	for the randomly generated piecewise polynomial signals corrupted by Gaussian noise of level $\eta=\protect\input{Experiments/Run_Time/Run_Time_B/noise_lvl_Potts.tex}.$
	Computation times for selected lengths $N$ are tabulated \emph{(top)}
	and plotted for $N=100, 200, \ldots, 10000$ \emph{(center)}.
	We observe that the runtime only grows linearly in $N,$
	so much more favorable than the worst case scenario of quadratic growth.
	This means that the pruning strategies show their full effectiveness for this type of signals,
	which is reflected by the linear growth in the number of error updates \emph{(bottom).}
		}
\end{figure}

We investigate the computation time and the number of executed error updates
depending on the signal length $N.$ 
To this end, we generate two types of synthetic signals:
signals with increasing number of discontinuities 
and signals with constant number of discontinuities.

The signals are generated as follows.
For the first type we let for each $g_i, i= 1, \ldots, N,$ the probability of a jump discontinuity be
$p=
0.01$; that is, the length of each smooth segment $I$ of $g$ follows a geometric distribution with parameter $p$. Hence, the expected segment length is $1/p = 100,$ and the expected number of segments grows linearly with respect to $N$.
Within a segment $I$ the signal $g$ is polynomial of degree $k-1$ with coefficients generated by the random variables $\frac{1}{(j+1)^2} X_j$, $j=1,...,k$, where $X_j$ are i.i.d.~uniformly distributed on $[-1,1]$. For the length $h$ of $I$, the domain of $g_I$ is $[0,hp]$ sampled with step size $p$.
For spline order $k$, the degree of polynomials is set to $k-1$.
The second type of signals is created by taking $N$ equidistant samples
of the continuously defined signal shown in \cref{Fig Mallat Reconstructions}.
In all cases, the signals are corrupted by additive Gaussian noise with noise level 
$\eta=
0.1
$. For every considered $N$, we computed $
	\protect 1000
	$ realizations and report the mean computation time and the mean number of performed error updates, respectively.

The results for the first type of
signals are shown in \cref{Figure_Experiments_RunTime_PwPoly}.
It is an important observation that 
the runtime and the errors updates
exhibit linear growth in the signal length.
Thus, the proposed algorithm 
does not show its worst case complexity.
That means, that the utilized pruning strategies are highly effective.
The results for the second type of
signals are shown in \cref{Table Mallat Times}.
In contrast to the first type,  
the computation time grows approximately quadratic 
in the number of elements,
which means that 
the proposed algorithm
attains its worst case complexity.
These results suggest that an increasing number of discontinuities 
is beneficial for the efficiency of 
the proposed algorithm.

% !TEX root = ../../../higherOrderBZ.tex
\begin{table}[]
\small
\begin{subfigure}{1\textwidth}
\centering
\begin{tabular}{ r r r r r r r r r r}
\toprule
	 &  &  &  &  & $N$ &  &  &  & \\ 
 $k$ & $2^{9}$& $2^{10}$& $2^{11}$& $2^{12}$& $2^{13}$& $2^{14}$& $2^{15}$& $2^{16}$& $2^{17}$\\ 
 \midrule    
 1  & 0.0025 & 0.0091 & 0.0363 & 0.1132 & 0.4230 & 1.6557 & 6.5899 & 28.8285 & 123.2784\\ 
 2  & 0.0025 & 0.0064 & 0.0210 & 0.0714 & 0.4875 & 1.9040 & 7.5987 & 33.0966 & 139.5842\\ 
 3  & 0.0030 & 0.0075 & 0.0196 & 0.0873 & 0.3435 & 2.3475 & 9.4270 & 40.4083 & 166.7910\\ 
 4  & 0.0032 & 0.0078 & 0.0569 & 0.0935 & 0.3685 & 1.6173 & 10.1401 & 44.7423 & 179.0001\\ 

\bottomrule
 \end{tabular}
 \subcaption{Runtime [s] for (higher order) Mumford-Shah solver}
\end{subfigure}
\\[1em]
\begin{subfigure}{1\textwidth}
\centering
\begin{tabular}{ r r r r r r r r r r}
\toprule 
	 &  &  &  &  & $N$ &  &  &  & \\ 
 $k$ & $2^{9}$& $2^{10}$& $2^{11}$& $2^{12}$& $2^{13}$& $2^{14}$& $2^{15}$& $2^{16}$& $2^{17}$\\ 
 \midrule    
 1  & 0.0010 & 0.0030 & 0.0055 & 0.0147 & 0.0502 & 0.1627 & 0.5713 & 2.0363 & 7.1102\\ 
 2  & 0.0016 & 0.0037 & 0.0090 & 0.0292 & 0.1036 & 0.3693 & 1.3562 & 4.7746 & 16.8562\\ 
 3  & 0.0019 & 0.0048 & 0.0116 & 0.0423 & 0.1613 & 0.6022 & 2.2298 & 8.2050 & 30.1515\\ 
 4  & 0.0024 & 0.0057 & 0.0172 & 0.0517 & 0.1977 & 0.7536 & 2.8576 & 10.8020 & 41.7345\\ 

\bottomrule
\end{tabular}
 \subcaption{Runtime [s] for (higher order) Potts solver}
\end{subfigure}
\caption{\label{Table Mallat Times}
Mean computation times of the proposed algorithm %Algorithm~\ref{alg:solverPenalized} 
(in seconds) for the signal from \cref{Fig Mallat Reconstructions} sampled on $N$ points
%on successively finer scales. 
and corrupted by Gaussian noise with noise level $\protect 0.1.$
We observe that the runtime grows approximately quadratic in $N$;
that is, the worst case complexity is attained.
The relevant difference to the experiment in \cref{Figure_Experiments_RunTime_PwPoly}  is that the number 
of discontinuities does not increase with $N.$
Yet, the solver processes signals of size $2^{16}$ in less than one
minute.
%\emph{Top:} Mumford-Shah model of orders $k=1,...,4$.
%\emph{Bottom:} Potts model of orders $k=1,...,4$.
}
\end{table}

\section{Conclusion}

We have studied higher order Mumford-Shah and Potts models.
Their central advantage compared with classical first order models
is that they do not penalize polynomial trends of order $k-1$ on the segments.
This leads to improved estimation for data 
with piecewise  linear or polynomial trends.
We have shown that the defining functionals have unique minimizers 
for almost all input signals.
We have proposed a fast solver for higher order Mumford-Shah and Potts models.
We have obtain stability results.
We have shown that the worst case complexity 
of the proposed algorithm is quadratic in the length of the signal for arbitrary orders $k \geq 1.$ 
In the numerical experiments, we have further observed
that the runtime grows only linear for signals with 
linearly increasing number of discontinuities.
Further, the numerical experiments confirm the robustness and stability of the proposed method.
Our reference implementation processes even long signals in reasonable time;
for example signals of length $10,000$ need less than one second.
This way, the family of higher order Mumford-Shah and Potts models 
can serve as efficient smoothers for signals with discontinuities.

\section*{Acknowledgement}
This work was supported by the 
German Research Foundation (DFG grants STO1126/2-1 and \mbox{WE5886/4-1}).

\appendix
\section{Pseudocode of the proposed solver}
A pseudocode for the proposed
solver for \eqref{eq:penalizedProblem} is given in Algorithm~\ref{alg:solverPenalized}.

%%%%%%%%%%%%%%%%%%%%%%%%%%%%%%%%%%%%%%%%%%%%%%%%%%%%%%%%%%%%%%
\begin{algorithm}[p]
\small
\caption{Solver for the higher order Mumford-Shah problem and higher order Potts problem}
\label{alg:solverPenalized}
\SetCommentSty{footnotesize}
\DontPrintSemicolon
\BlankLine
\KwIn{Data $f \in \RN$; model parameters $k\in\N, \beta \in (0, \infty], \gamma > 0;$}
\BlankLine
\KwOut{Global minimizer $(u^*,\mathcal{I}^*)$ of
 \eqref{eq:penalizedProblem}
or \eqref{eq:penalizedProblemPoly}
}
\BlankLine
\tcc{Precomputations}
Row-wise transform the matrix
$\quad\begin{cases}
\eqref{eq:penalizedProblem}:&\text{ $A$ from \eqref{eq:AmatrixGvector}} \\
\eqref{eq:penalizedProblemPoly}:&\text{ $B$ from \eqref{eq:systemMatrixPoly}}
\end{cases}
$
\newline
to upper triangular form  using successive Givens rotations and store the rotation angles in $\Theta$.

\BlankLine
Compute $\Ec^{\dint{1}{r}}$ for all $r=1,\ldots,N$ with $\Theta$\\
\BlankLine
\tcc{Find optimal changepoints}
Initialize lists $L = [2]$, $R=[2]$, $E = [0]$\\
 $J_1 \leftarrow 0$, $P^*_1 \leftarrow 0$\\	
\For{$r\leftarrow 2,\ldots, N$}{
	\tcc{Initialization}
	$J_r \leftarrow 0$,
	$P_r^* \leftarrow \Ec^{\dint{1}{r}}$
\\
	\tcc{Find optimal $P_r^*$ using \eqref{eq:recurrencePenalized}}
	\For{$i=1,..., \text{length of }L$}
	{
		\nllabel{alg:GoToDestination}
		\While{$r_i < r$}{
			\tcc{Update approximation error
				}
			$\begin{cases}
			\eqref{eq:penalizedProblem}:
			\text{
				Compute $\Ec^{l_i:r_i+1}$ from $\Ec^{l_i:r_i}$ using the recurrence \eqref{eq:recurrence_q}-\eqref{eq:ErrorUpdate_new}
			}
			\\
			\eqref{eq:penalizedProblemPoly}:
			\text{
				Compute $\Ec^{l_i:r_i+1}$ from $\Ec^{l_i:r_i}$ using the recurrence \eqref{eq:recurrence_q_Potts}-\eqref{eq:ErrorUpdatePoly}
			}
			\end{cases}$~\\
			$E_i \leftarrow \Ec^{\dint{l_i}{r_{i}+1}}$,\\ 
			$r_i \leftarrow r_{i}+1$ \\
			\tcc{Pruning \eqref{eq:KFpruning}}
			\If{$P^*_{l_i-1}+E_i\geq P_{r_i}^*$}
			{
				Delete: $l_i$ from $L$, $r_i$ from $R$ and $E_i$ from $E$
				\\
				\textbf{go to} \ref{alg:GoToDestination}
			}
		}
		$b\leftarrow P^*_{l_i-1} + \gamma + E_i$\\
		\If{$b\leq P_r^*$}
		{
			$P_r^*\leftarrow b$, 
			$J_r \leftarrow l-1$	
		}
		\tcc{Pruning \eqref{eq:AMpruning}}
		\textbf{If }$E_i+\gamma > P^*_r$ \textbf{then break end}
	}
	\tcc{Update lists}
	Prepend: $r+1$ to $L$,
	$r+1$ to $R$, $0$ to $E$
}
\BlankLine
\tcc{Recover partition $\mathcal{I}^*$ from segment boundary locations $J$}
$r\leftarrow N$, $\mathcal{I}^* \leftarrow \emptyset$\\
\While{$r>0$}{
	$l\leftarrow J(r)+1$, 
	$\mathcal{I}^* \leftarrow \mathcal{I}^*\cup \{ (l:r)  \} $,
	$r \leftarrow l-1$
}
\BlankLine
\tcc{Reconstruction of $u^*$ by solving linear systems on segments (using QR decomposition and reusing $\Theta$ for speedup)}
\For{$I\in\mathcal{I^*}$}
{
	$\begin{cases}
	\text{
	\eqref{eq:penalizedProblem}: Solve 
	$u_I^* = \argmin\limits_{v \in \R^{|I|}} \|  v - f_I \|_2^2 +    \beta^{2k} \| \nabla^k v \|_2^2$ 
	}\\
	\text{
	\eqref{eq:penalizedProblemPoly}: Solve 
	$u_I^* = \argmin
	\|  v - f_I \|_2^2$ 
	such that $v$ is polynomial of degree $\leq k-1$ on $I$ 	}
	\end{cases}$
}
\end{algorithm}
%%%%%%%%%%%%%%%%%%%%%%%%%%%%%%%%%%%%%%%%%%%%%%%%%%%%%%%%%%%%%%

{\footnotesize
\bibliographystyle{myplainnat}
\bibliography{higherOrderBZ}
}

\end{document}